\documentclass[a4paper,12pt]{amsart}

\usepackage{amssymb, enumitem}

\usepackage[breaklinks=true,colorlinks=true,linkcolor=green,citecolor=red,urlcolor=blue]{hyperref}

\allowdisplaybreaks

\newcommand \al{\alpha}

\newcommand \on{\overline{\nabla}}
\newcommand \G{\Gamma}

\newcommand \la{\lambda}
\newcommand \ve{\varepsilon}
\newcommand \br{\mathbb{R}}
\newcommand \bc{\mathbb{C}}

\newcommand \so{\mathfrak{so}}
\newcommand \su{\mathfrak{su}}
\newcommand \rk{\operatorname{rk}}
\newcommand \ad{\operatorname{ad}}
\newcommand \Ad{\operatorname{Ad}}
\newcommand \Span{\operatorname{Span}}
\newcommand \diag{\operatorname{diag}}
\newcommand \id{\operatorname{id}}
\newcommand \<{\langle}
\renewcommand \>{\rangle}
\newcommand \ip{\< \cdot, \cdot \>}
\renewcommand \Re{\operatorname{Re}}
\renewcommand \Im{\operatorname{Im}}

\newcommand \mL{\mathcal{L}}
\newcommand \mV{\mathcal{V}}
\newcommand \mN{\mathcal{N}}
\newcommand \mU{\mathcal{U}}

\newcommand \Tr{\operatorname{Tr}}
\newcommand \Ker{\operatorname{Ker}}

\newcommand \tM{\overline{M}}

\newcommand \tR{\overline{R}}

\newcommand \g{\mathfrak{g}}
\newcommand \gv{\mathfrak{v}}
\newcommand \gn{\mathfrak{n}}

\newcommand \ga{\mathfrak{a}}
\newcommand \gs{\mathfrak{s}}
\newcommand \gk{\mathfrak{k}}
\newcommand \ogs{\overline{\mathfrak{s}}}
\newcommand \gp{\mathfrak{p}}

\newcommand \Sh{\mathrm{S}}
\newcommand \oSh{\overline{\mathrm{S}}}

\newcommand \rG{\mathrm{G}}
\newcommand \rK{\mathrm{K}}
\newcommand \rA{\mathrm{A}}
\newcommand \rN{\mathrm{N}}
\newcommand \ri{\mathrm{i}}
\newcommand \Sym{\mathrm{Sym}}
\newcommand \SO{\mathrm{SO}}
\newcommand \SU{\mathrm{SU}}
\newcommand \SL{\mathrm{SL}}
\newcommand \rGA{\mathrm{\Gamma}}

\theoremstyle{plain}
\newtheorem{theorem}{Theorem}
\newtheorem*{theorem*}{Theorem}
\newtheorem*{corollary*}{Corollary}

\newtheorem*{conj*}{Conjecture}
\newtheorem{lemma}{Lemma}
\newtheorem{proposition}{Proposition}
\newtheorem*{prop*}{Proposition}

\theoremstyle{definition}

\newtheorem*{definition*}{Definition}

\theoremstyle{remark}
\newtheorem{remark}{Remark}
\newtheorem{example}{Example}

\makeatletter
\@namedef{subjclassname@2020}{%
  \textup{2020} Mathematics Subject Classification}
\makeatother

\begin{document}

\title{Einstein hypersurfaces in irreducible symmetric spaces}

\author{Yuri Nikolayevsky}
\address{Department of Mathematics and Statistics, La Trobe University, Melbourne, Victoria, 3086, Australia}
\email{y.nikolayevsky@latrobe.edu.au}

\author{JeongHyeong Park}
\address{Department of Mathematics, Sungkyunkwan University, Suwon, 16419, Korea}
\email{parkj@skku.edu}

\thanks {The first author was partially supported by ARC Discovery grant DP210100951. \\
\indent The second author were supported by the National Research Foundation of Korea(NRF) grant funded by the Korea government(MSIT) (NRF-2019R1A2C1083957).}
 %; Secondary

% 53C25  Special Riemannian manifolds (Einstein, Sasakian, etc.)
% 53B20  Local Riemannian geometry
% 53B25  Local submanifolds
% 53C30  Homogeneous manifolds
% 53C35  Symmetric spaces

\begin{abstract}
We show that if $M$ is an Einstein hypersurface in an irreducible Riemannian symmetric space $\tM$ of rank greater than $1$ (the classification in the rank-one case was previously known), then either $\tM$ is of noncompact type and $M$ is a codimension one Einstein solvmanifold, or $\tM=\SU(3)/\SO(3)$ (respectively, $\tM=\SL(3)/\SO(3)$) and $M$ is foliated by totally geodesic spheres (respectively, hyperbolic planes) of $\tM$, with the space of leaves  parametrised by a special Legendrian surface in $S^5$ (respectively, by a proper affine sphere in $\br^3$).
\end{abstract}

\maketitle

\section{Introduction}
\label{s:intro}

Einstein hypersurfaces in Riemannian spaces are quite rare --- for example, many homogeneous spaces (including $\bc P^n$) admit no Einstein hypersurfaces at all, not necessarily homogeneous. Nevertheless, the list of spaces whose Einstein hypersurfaces are fully understood is very short. By \cite[Theorem~7.1]{Fia}, an Einstein hypersurface in a \emph{space of constant curvature} is locally either totally umbilical, or totally geodesic, or is of conullity $1$ (developable), or is the product of two spheres of the same Ricci curvature in the sphere. The classification for \emph{rank-one symmetric spaces} of non-constant curvature is summarised in the following theorem (the metric is scaled in such a way that the sectional curvature lies in $[\frac14, 1]$). %of dim at least 4

\begin{theorem}\label{th:rk1}
{\ }

  \begin{enumerate}[label=\emph{\arabic*.},ref=\arabic*]
    \item %\label{t1it}
    There are no Einstein \emph{(}real\emph{)} hypersurfaces in the complex projective space and in the complex hyperbolic space \cite[Theorem~4.3]{Kon}, \cite[Corollary~8.2]{Mon}, \cite[Theorem~8.69]{CR}.

    \item
    There are no Einstein \emph{(}real\emph{)} hypersurfaces in the quaternionic hyperbolic space \cite[Corollary~1]{OP}. A connected \emph{(}real\emph{)} hypersurface in $\mathbb{H}P^m, \, m \ge 2$, is Einstein if and only if it is an open domain of a geodesic sphere of radius $r \in (0, \pi)$, where $\cos r = \frac{1-2m}{1+2m}$ \cite[Corollary~7.4]{MP}.

    \item
    There are no Einstein hypersurfaces in the Cayley hyperbolic plane. A connected hypersurface in the Cayley projective plane is Einstein if and only if it is an open domain of a geodesic sphere of radius $r \in (0, \pi)$ such that $\cos r = -\frac{5}{11}$ \cite{KNP2}.
  \end{enumerate}
\end{theorem}

In \cite{KNP1} we proved that there are no Einstein hypersurfaces in Damek-Ricci spaces (the closest ``relatives" of noncompact rank-one symmetric spaces). By \cite{LPS}, Einstein hypersurfaces in $S^n \times \br$ and $H^n \times \br$ have constant curvature; such hypersurfaces are classified in \cite{MT}.%; in particular, when $n \ge 4$ one gets hypersurfaces of revolution of special curves and the product of a horosphere by $\br$

In this paper, we classify Einstein hypersurfaces in irreducible symmetric spaces of rank at least two. It turns out that the picture is exactly the opposite to what one observes in the rank-one case (Theorem~\ref{th:rk1}): spaces of compact type admit no Einstein hypersurfaces, with a single exception, the Wu manifold $\SU(3)/\SO(3)$, while one has many Einstein hypersurfaces in spaces of noncompact type. We introduce three classes of Einstein hypersurfaces in three examples below, and then the classification in Theorem~\ref{th:main} will state that there are no others.

\begin{example} \label{ex:solv}
  Any symmetric space $\tM$ of noncompact type is isometric to a \emph{solvmanifold} with a left-invariant metric, via the Iwasawa decomposition (see Section~\ref{ss:iwa} for details and unexplained terminology), so that $\tM$ is a simply connected, solvable Lie group whose metric Lie algebra $\gs$ is the semi-direct, orthogonal sum of an abelian Lie algebra $\ga,\; \dim \ga = \rk \tM$, and the nilradical $\gn$. One has a distinguished, nonzero vector $A_H \in \ga$, the mean curvature vector of the Lie group of $\gn$ viewed as a submanifold of $\tM$. If $\rk \tM \ge 2$, we can take an arbitrary nonzero $\xi \in \ga$ orthogonal to $A_H$; then $\xi^\perp$ is a codimension one solvable subalgebra of $\gs$. The Lie subgroup $M \subset \tM$ with the Lie algebra $\xi^\perp$ is then a hypersurface, which is Einstein, provided $\tM$ is. So constructed hypersurface $M$ is minimal and has the same Einstein constant as $\tM$. If $\rk \tM = 2$, all such hypersurfaces are congruent; if $\rk \tM > 2$, they do not even have to be locally isometric.
\end{example}

\begin{example} \label{ex:inWu}
  Let $\tM=\SU(3)/\SO(3)$. We identify $\tM$ with the set of symmetric matrices of $\SU(3)$ via the map $g \mapsto gg^t, \; g \in \SU(3)$. Identify $\bc^3$ with a Euclidean space $\br^6$ with a complex structure (multiplication by $\ri$). Under these identifications, for a unit vector $Z \in \bc^3=\br^6$, the subset $S_Z=\{x \in \SU(3)/\SO(3) \, | \, x \overline{Z} = Z\}$ is a totally geodesic $2$-sphere of $\SU(3)/\SO(3)$ (see Section~\ref{s:su3so3} for details). A two-dimensional regular surface $F^2$ in the unit sphere $S^5 \subset \bc^3=\br^6$ is called \emph{Legendrian} if the vector $\ri Z$ is normal to $F^2$ at any point $Z \in F^2$. A Legendrian surface $F^2 \subset S^5$ is called \emph{$\beta$-special Legendrian} ($\beta \in \br/ \pi\mathbb{Z}$), if $\det_{\bc}(Z , Z_1 , Z_2)=\pm e^{i\beta}$, where $\{Z_1, Z_2\}$ is an orthonormal basis for $T_ZF^2$. We prove that a regular hypersurface $M \subset \SU(3)/\SO(3)$ which is an open domain of $M_{F^2}=\cup_{Z \in F^2} S_Z$, where $F^2 \subset S^5$ is a $\frac{\pi}{2}$-special Legendrian surface, is Einstein. A hypersurface $M$ so constructed is of conullity $2$, with the nullity leaves $S_Z$, and is locally isometric to $\bc P^2$, with the same Einstein constant as $\tM$. The simplest example of such a hypersurface is obtained if we take for $F^2$ the sphere $S^5 \cap \ri \br^3$; then $M$ is the set of symmetric matrices of $\SU(3)$ having an eigenvalue $-1$, or equivalently, the $\SO(3)$-orbit of a particular geodesic of $\tM$ (see Section~\ref{s:su3so3} for more examples).
\end{example}

\begin{example} \label{ex:inncWu} % Also note S_pq is defined differently here and in the section
  Let $\tM=\SL(3)/\SO(3)$. We identify $\tM$ with the set of positive definite quadratic forms of determinant $1$ on an equi-affine space $\br^3$ with a volume form $\omega$ (see Section~\ref{s:sl3so3} for details). In the space $\br^6 = \br^3 \oplus {\br^3}^\ast$ (the cotangent bundle of $\br^3$), consider the ``unit sphere" $S^5_- =\{(p,\rho) \, | \, p \in \br^3, \, \rho \in {\br^3}^\ast, \, \rho(p)=1\}$. For $(p,\rho) \in S^5_-$, the subset $S_{(p, \rho)} = \{h \in \SL(3)/\SO(3) \, | \, h(p, \cdot)=\rho\}$ is a totally geodesic hyperbolic plane of $\SL(3)/\SO(3)$. We call a two-dimensional regular surface $F^2 \subset S_-^5 \subset \br^3 \oplus {\br^3}^\ast=\br^6$ \emph{para-Legendrian} if at any point $(p, \rho) \in F^2$, the $1$-form $\rho$ vanishes on the tangent space to $p$. A para-Legendrian surface $F^2 \subset S_-^5$ is called \emph{special para-Legendrian}, if the $3$-form $\omega \oplus \omega^\ast$ on $\br^3 \oplus {\br^3}^\ast$ vanishes on the $3$-space spanned by $(p,\rho)$ and the tangent space to $F^2$ at $(p,\rho)$, where the $3$-form $\omega^\ast$ on ${\br^3}^\ast$ is dual to $\omega$. We prove that a regular hypersurface $M \subset \SL(3)/\SO(3)$ which is an open domain of $M_{F^2}=\cup_{(p,\rho) \in F^2} S_{(p,\rho)}$, where $F^2 \subset S_-^5$ is a special para-Legendrian surface, is Einstein. A hypersurface $M$ so constructed is of conullity $2$, with the nullity leaves $S_{(p,\rho)}$, and is locally isometric to $\bc H^2$, with the same Einstein constant as $\tM$.

  At generic points, at which the projection of $F^2$ to $\br^3$ has rank two, $\rho$ is uniquely determined by that projection, which then must be an open domain of a \emph{proper indefinite affine sphere} (\emph{Tzitz\`{e}ica surface}) in $\br^3$ with the centre at the origin and of centro-affine curvature $-1$ (equivalently, $K\<p,n\>^{-4} = -1$, where $K$ is the Gauss curvature, $p$ is a position vector and $n$ is a unit normal); see Section~\ref{s:sl3so3}. The simplest example is obtained by taking the hyperboloid $p_1^2+p_2^2-p_3^2=1$ as a Tzitz\`{e}ica surface; then the hypersurface $M$ is the $\SO(2,1)$-orbit of a particular geodesic of $\tM$. If the projection of $F^2$ to $\br^3$ has locally rank one, then $F^2$ is locally an affine plane and the corresponding hypersurface $M$ is an open domain of a codimension one Einstein solvmanifold of $\tM$, as in Example~\ref{ex:solv}.
\end{example}

Our main result is the following theorem.
\begin{theorem} \label{th:main}
  Let $\tM$ be an irreducible, simply connected, globally symmetric space of rank at least $2$, and let $M$ be a connected Einstein $C^4$-hypersurface in $\tM$.
  \begin{enumerate}[label=\emph{\arabic*.},ref=\arabic*]
    \item \label{thit:solv}
    If $\tM$ is different from $\SL(3)/\SO(3)$ and $\SU(3)/\SO(3)$, then $\tM$ is of noncompact type and $M$ is \emph{(}an open domain of\emph{)} a codimension one Einstein solvmanifold, as constructed in Example~\ref{ex:solv}.

    \item \label{thit:wu}
    If $\tM$ is isometric to $\SU(3)/\SO(3)$, then $M$ is an open domain of $M_{F^2}=\cup_{Z \in F^2} S_Z$, where $F^2 \subset S^5$ is a $\frac{\pi}{2}$-special Legendrian surface, as constructed in Example~\ref{ex:inWu}. Such a hypersurface $M$ is of conullity $2$, with the nullity leaves $S_Z$, and is locally isometric to $\bc P^2$, with the same Einstein constant as $\tM$.

    \item \label{thit:ncwu}
    If $\tM$ is isometric to $\SL(3)/\SO(3)$, then $M$ is an open domain of $M_{F^2}=\cup_{(p, \rho) \in F^2} S_{(p, \rho)}$, where $F^2 \subset S_-^5$ is a special para-Legendrian surface, as constructed in Example~\ref{ex:inncWu}. Such a hypersurface $M$ is of conullity $2$, with the nullity leaves $S_{(p, \rho)}$, and is locally isometric to $\bc H^2$, with the same Einstein constant as $\tM$.
  \end{enumerate}
\end{theorem}

We will show that in the cases~\eqref{thit:wu} and \eqref{thit:ncwu} of Theorem~\ref{th:main}, every complete leaf of $M_{F^2}$ contains singular points (except for in the case when $\tM=\SL(3)/\SO(3)$ and $F^2$ is an affine plane --- see the very end of Example~\ref{ex:inncWu}), which gives the following.

\begin{corollary*} \label{cor:complete}
    Let $\tM$ be an irreducible, simply connected, globally symmetric space of rank at least $2$, and let $M$ be a complete, connected Einstein $C^4$-hypersurface in $\tM$. Then $\tM$ is of noncompact type and $M$ is a codimension one Einstein solvmanifold, as constructed in Example~\ref{ex:solv}.
\end{corollary*}

\begin{remark} \label{rem:Ck}
  The requirement that $M$ is of regularity class $C^4$ in Theorem~\ref{th:main} is only necessary for all parts of the proof to work, and may probably be relaxed. The induced metric of an Einstein submanifold must be analytic by the well-known result of \cite[Theorem~5.2]{DTK}, but the submanifold itself does not have to be: for example, a cylinder or a cone over a rectifiable curve in $\br^3$. We observe a similar phenomenon in our settings. While in the first case of Theorem~\ref{th:main}, $M$ is a Lie subgroup (and hence is analytic), and in the second case, $M$ is foliated by totally geodesic leaves parametrised by a special Legendrian surface (which is minimal in the sphere, and hence analytic), in the third case, the affine sphere from which the Einstein hypersurface $M$ is constructed in Example~\ref{ex:inncWu} may belong to a finite regularity class (see Example~\ref{ex:affsph} in Section~\ref{s:sl3so3}).
\end{remark}

%The paper is organised as follows.

\section{Preliminaries}
\label{s:pre}

\subsection{Restricted roots}
\label{ss:roots}

In this section, we give a brief introduction to Riemannian symmetric spaces and introduce some notation (a standard reference is \cite{Hel}). Let $\tM=\rG/\rK$ be an irreducible Riemannian symmetric space, where $\rG$ is the identity component of the full isometry group of $\tM$ (then $\rG$ is semisimple) and $\rK$ is the isotropy subgroup. We assume that $\tM$ is of rank at least $2$. Denote $\ve = \pm 1$ the sign of the Einstein constant of $\tM$. Let $B$ be the Killing form of $\g$, and $\g=\gk \oplus \gp$ be the Cartan decomposition, where $\g$ and $\gk$ are the Lie algebras of $\rG$ and $\rK$ respectively, and $\gp$ is the orthogonal complement to $\gk$ relative to $B$. We have $[\gk,\gk], [\gp,\gp] \subset \gk$ and $[\gk,\gp] \subset \gp$. The tangent space $T_x\tM$ at a point $x \in \tM$ can be identified with $\gp$. The inner product $\ip$ on $\gp$ is defined by $\<X,Y\>=-\ve B(X,Y)$ (we will occasionally scale it by a positive constant for convenience). The \emph{curvature tensor} of $\tM$ is given by $\tR(X,Y)Z=-[[X,Y],Z]$ for $X,Y,Z \in \gp$. A subspace $\gp' \subset \gp$ is called a \emph{Lie triple system}, if $[[\gp',\gp'],\gp'] \subset \gp'$. The exponent of a Lie triple system is a totally geodesic submanifold in $\tM$, and conversely, the tangent space to any totally geodesic submanifold of $\tM$ is a Lie triple system.

Let $\ga \subset \gp$ be a \emph{Cartan subalgebra}, a maximal abelian subalgebra of $\gp$. An element $X \in \gp$ is called \emph{regular} if it is contained in a unique Cartan subalgebra. The set of regular elements is open and dense in $\gp$ and in any $\ga$. For $\beta \in \ga$ we denote $\g_\beta=\{X \in \g \, | \, [A,[A,X]]=-\ve\<\beta,A\>^2X, \, \text{for all } A \in \ga\}$. A \emph{restricted root} is an element $\beta \in \ga \setminus \{0\}$ such that $\g_\beta \ne \{0\}$. Note that $\g_\beta=\g_{-\beta}$. Denote $\Delta \subset \ga$ the set of restricted roots. As $\tM$ is irreducible, so is $\Delta$ (it does not lie in the union of two proper orthogonal subspaces of $\ga$). For $\beta \in \Delta$, denote $\gp_\beta=\gp_{-\beta} = \g_\beta \cap \gp, \; \gk_\beta=\gk_{-\beta}=\g_\beta \cap \gk$. We call $\gp_\beta$ the \emph{root spaces}. Choosing a hyperplane in $\ga$ orthogonal to a regular vector we denote $\Delta^+$ the set of positive roots, those elements of $\Delta$ which lie on one side of that hyperplane. Then we have the root decompositions
\begin{equation*}
    \gp=\ga \oplus \bigoplus\nolimits_{\beta \in \Delta^+}\gp_\beta, \qquad \gk=\gk_0 \oplus \bigoplus\nolimits_{\beta \in \Delta^+}\gk_\beta,
\end{equation*}
(note that the first one is orthogonal relative to $\ip$). For any $\beta \in \Delta$, there exists a linear bijection $\theta_\beta:\gp_\beta \to \gk_\beta, \; \theta_{-\beta}=-\theta_\beta$, such that $[A, X_\beta]=\<\beta,A\>\theta_\beta X_\beta$, $[A, \theta_\beta X_\beta]=-\ve \<\beta,A\> X_\beta$ (and so $\<[X_\beta, \theta_\beta X_\beta],A\>= \ve \<\beta, A\> \|X_\beta\|^2$), for all $A \in \ga, \; X_\beta \in \gp_\beta$. In particular, $\dim \gp_\beta=\dim \gk_\beta$, the \emph{multiplicity} of the root $\beta \in \Delta$. For any $\beta, \gamma \in \Delta \cup \{0\}$, we have $[\gp_\beta,\gp_\gamma], [\gk_\beta,\gk_\gamma] \subset \gk_{\beta+\gamma} \oplus \gk_{\beta-\gamma}$ and $[\gk_\beta,\gp_\gamma] \subset \gp_{\beta+\gamma} \oplus \gp_{\beta-\gamma}$, where we denote $\gp_0=\ga$.

The following two facts are well known.
\begin{lemma}\label{l:rr}
{ \ }
\begin{enumerate}[label=\emph{(\alph*)},ref=\alph*]
  \item \label{it:rr1}
  If the roots $\beta, \gamma \in \Delta$ are not proportional and not orthogonal, then $\tR(A,X_\beta) X_\gamma \ne 0$, for all nonzero $X_\beta \in \gp_\beta, \; X_\gamma \in \gp_\gamma$, and all $A \in \ga$ such that $\<\beta,A\> \ne 0$.

  \item \label{it:rrnot2hyper}
  If $\tR(X,\gp,\gp,Y)=0$ for $X, Y \in \gp$, then $X=0$ or $Y=0$.
\end{enumerate}
\end{lemma}
\begin{proof}
Assertion~\eqref{it:rr1} is equivalent to saying that $[X_\gamma,\theta_\beta X_\beta] \ne 0$. Suppose this is false. Take $A \in \ga$ such that $\<\gamma,A\> = 0, \; \<\beta,A\> \ne 0$. Then $0=[A,[X_\gamma, \theta_\beta X_\beta]]=-\ve \<\beta,A\>[X_\gamma, X_\beta]$, and so $[X_\gamma, X_\beta]=0$.
We get $0=[X_\gamma,[X_\beta,\theta_\beta X_\beta]]$. As $[X_\beta,\theta_\beta X_\beta] \in \ve \|X_\beta\|^2 \beta + \gp_{2\beta}$, we obtain
$0=[X_\gamma,[X_\beta,\theta_\beta X_\beta]]= -\ve \|X_\beta\|^2 \<\gamma, \beta\>\theta_\gamma X_\gamma$ plus possibly some element from $\gk_{2\beta-\gamma}\oplus\gk_{2\beta+\gamma}$, a contradiction.

For assertion~\eqref{it:rrnot2hyper}, if $[X,Y] \ne 0$, then $\tR(X,Y,X,Y)=\ve B([[X,Y],X],Y)=\ve B([X,Y],[X,Y])$ $\ne 0$, as $B$ is negative definite on $\gk$. If $[X,Y]=0$, take a Cartan subalgebra $\ga \subset \gp$ containing both $X$ and $Y$. As $\Delta$ cannot lie in the union of two hyperplanes, we can choose $\beta \in \Delta$ such that $\<\beta, X\>\<\beta, Y\> \ne 0$. Then for a nonzero $X_\beta \in \gp_\beta$, we have $\tR(X,X_\beta,X_\beta,Y)=\ve \<\beta,X\>\<\beta,Y\> \|X_\beta\|^2 \ne 0$. % add: \Delta can't lie in the union of two orthogonal hyperplanes
\end{proof}

For detailed, explicit lists of restricted root systems the reader is referred to \cite{Ber} and \cite{T1}.

\subsection{Iwasawa decomposition}
\label{ss:iwa}
In this section, we explain the construction given in Example~\ref{ex:solv} (Theorem~\ref{th:main}\eqref{thit:solv}).

In the noncompact case, the group $\rG$ admits the Iwasawa decomposition $\rG=\rK\rA\rN$, where $\rA$ is abelian (with the Lie algebra $\ga$), $\rN$ is nilpotent and $\rA\rN$ solvable \cite[\S VI.5]{Hel}. All three subgroups $\rA, \rN$, and $\rA\rN$ are connected and simply connected, and the Cartesian product $\rK \times \rA \times \rN$ is diffeomorphic to $\rG$. Moreover, the solvable subgroup $\rA\rN \subset \rG$ acts \emph{simply transitively} on $\tM$, that is, $\tM$ can be identified with the solvmanifold $\rA\rN$ with a left-invariant Einstein metric. At the Lie algebra level, the Iwasawa decomposition is given by $\g = \gk \oplus \ga \oplus \gn$ (the direct sum of vector spaces), where the solvable subalgebra $\gs=\ga \oplus \gn$ (with the Einstein inner product inherited from $\gp$) is the Lie algebra of the group $\rA\rN \, (=\tM)$. Note that $\dim \ga = \rk \tM$ (and so is at least $2$ in the assumptions of Theorem~\ref{th:main}). The Iwasawa decomposition depends on the choice of a Cartan subalgebra $\ga$ and of the subset $\Delta^+$ of positive roots; however, all these decompositions are conjugate to one another. With these choices made, the nilpotent subalgebra $\gn$ is spanned by the vectors $X_\beta + \theta_\beta X_\beta$, where $X_\beta \in \gp_\beta, \, \beta \in \Delta^+$. We also note that the Iwasawa decomposition can be similarly constructed when $\tM$ is the Riemannian product of a Euclidean space and a symmetric space of noncompact type (such $\tM$ may be reducible and non-Einstein, with a reductive group $\rG$); we will use this fact in Section~\ref{s:solv} (in the construction before Lemma~\ref{l:Iwasawa}).

The Riemannian Einstein solvmanifold $\rA\rN$ (which we identify with $\tM$) is of \emph{Iwasawa type} \cite[Definition~1.2]{Wol}. This means that $\gs$ is the orthogonal semidirect product of the abelian subalgebra $\ga$ and the nilradical $\gn=[\gs,\gs]$, all the operators $\ad_A, \; A \in \ga$, are symmetric, and for some $A_0 \in \ga$, all the eigenvalues of $(\ad_{A_0})_{|\gn}$ are positive.

As a side remark, note that by a recent grounbreaking result of \cite{BL}, \emph{any homogeneous Einstein space with negative Einstein constant} is a solvmanifold of Iwasawa type.

Given a metric, Einstein, solvable Lie algebra $\gs$ of Iwasawa type (in particular, an algebra coming from the Iwasawa decomposition of a semisimple algebra, as above), there is a distinguished vector $A_H \in \ga$ (which can serve as $A_0$), the \emph{mean curvature vector}. Geometrically, $A_H$ is the mean curvature vector of the subgroup $\rN \subset \rA\rN$; algebraically, $A_H$ is defined by the condition $\<A_H,A\>=\Tr \ad_A$, for all $A \in \ga$. Assuming $\dim \ga \ge 2$ (which is so in our case) one can choose an arbitrary nonzero vector $\xi \in \ga$ such that $\xi \perp A_H$, and consider the codimension one subalgebra $\xi^\perp \subset \gs$. This subalgebra is still solvable and is of Iwasawa type; moreover, it is still Einstein (relative to the induced inner product), with the same mean curvature vector $A_H$ and the same Einstein constant \cite[Lemma~1.4]{Wol}. The corresponding connected Lie subgroup of $\rA\rN \, (=\tM)$ is then an Einstein hypersurface in $\tM$, with the same Einstein constant. Note that if $\rk \tM =2$, all such hypersurfaces are congruent, but if $\rk \tM > 2$, they do not have to be even isometric \cite[Theorem~1]{Ale}. For a more general construction of Einstein solvmanifolds of higher codimension from the Iwasawa decomposition, the reader is referred to \cite{T2} and to the recent survey paper \cite{DDS}.

\subsection{Gauss map and Gauss image}
\label{ss:Gm}

There are several definitions of the Gauss map of a hypersurface in a symmetric space available in the literature (e.g. the one in \cite[Section~2.1]{RR} is different from the one below; our definition is similar to \cite[Section~2]{Nik}).

Informally, given a hypersurface $M$ in a symmetric space $\tM$ of dimension $n+1$ and a point $x \in M$, we define the \emph{Gauss image of $M$} as a subset of the unit sphere of $T_x\tM$ which consists of all the vectors $d\mL_g^{-1} \xi(y)$ where $g \in \rG$ and $y \in M$ satisfy $g(x) =y$, and $\mL$ is the left action of $\rG$. Note that the Gauss image so defined is $\Ad_\rK$-invariant. More formally, for $x \in M$ we identify $\gp$ with $T_x\tM$, where $\g= \gk \oplus \gp$ is the Cartan decomposition, and choose a continuous unit normal vector field $\xi$ on a neighbourhood of $x$ in $M$. Consider Killing vector fields on a neighbourhood $\overline{\mU}$ of $x \in \tM$ defined by elements of a neighbourhood of the origin in $\gp$. Then for every $y \in \overline{\mU}$, there is a unique isometry $\phi_y$ defined by one of these Killing vector fields which sends $y$ to $x$. We define the \emph{Gauss map} $\rGA$ from $(\overline{\mU} \cap M) \times \rK$ to the unit sphere of $T_x\tM$ by $\rGA(y,h)=\Ad_h (d(\phi_y)_y \xi(y))$ (where $h \in \rK$).

\begin{lemma} \label{l:dgamma}
  Let $X \in T_xM, \, U \in \gk=T_e\rK$. Then $(d\rGA)_{x,e} (X,U)= \on_X\xi + [U,\xi]$, where $\on$ is the Levi-Civita connection on $\tM$.
\end{lemma}
\begin{proof}
  Let $Y_1, \dots, Y_{n+1}$ be orthonormal vector fields on $\overline{\mU}$ which are left-invariant relative to all the isometries defined by the Killing vector fields $Z$ with $Z(x) \in \gp$. Note that for any such Killing vector field $Z$ and for any $i=1, \dots, n+1$, we have $[Z, Y_i]=0$ (the Lie bracket of vector fields). As at $x$ we have $\on_{Y_i}Z=0$ (see e.g. \cite[proposition~7.28]{Bes}; note that $[\gp,\gp] \subset \gk$), it follows that $(\on_Z Y_i)(x)=0$. We will further assume that $\xi(x)=Y_{n+1}(x)$. Let $\xi = \sum_{i=1}^{n+1} a_i Y_i$. Then for $X \in T_xM$ we have $(d\rGA)_{x,e} (X,0)= \sum_{i=1}^{n+1} X(a_i)_x Y_i$, as the vector fields $Y_i$ are left-invariant relative to isometries defined by the Killing vector fields from $\gp$. It follows that $(d\rGA)_{x,e} (X,0)= \on_X\xi - \sum_{i=1}^{n+1} a_i \on_X Y_i = \on_X \xi$, from the above. On the other hand, for $h=\exp(tU) \in \rK$, we have $\rGA(x,h)= \Ad_h \xi(x)$, and so $(d\rGA)_{x,e} (0,U)= [U,\xi]$.
\end{proof}

We say that the Gauss map $\rGA$ (the Gauss image) at a point $x \in M$ is \emph{nonsingular}, if $\rk(d\rGA)_{x,e} = n$ (that is, the rank is the maximal possible), and is \emph{singular} otherwise. Clearly, these definitions do not depend on the choice of the neighbourhood $\overline{\mU}$. Moreover, if the Gauss map at a point $x \in M$ is nonsingular, then the Gauss image of a neighbourhood of $x$ in $M$ has a nonempty interior in the unit sphere of $T_x\tM$; in particular, the vector $\xi$ is regular (at nearby points). Further properties, under the Einstein assumption, are given in Lemma~\ref{l:nons} in the next section.

\section{Proof of Theorem~\ref{th:main}, generalities}
\label{s:pfgen}

The rest of the paper is the proof of Theorem~\ref{th:main}. It goes as follows. In this section, we establish several technical facts which will be useful in subsequent sections. Then in Proposition~\ref{p:allsing} (Section~\ref{s:reg}) we show that at all points of an Einstein hypersurface $M$ of an irreducible Riemannian symmetric space $\tM$ of rank at least $2$, the Gauss map (as defined in Section~\ref{ss:Gm}) is singular. Next, in Proposition~\ref{p:singH=0} (Section~\ref{s:singg}) we show that, unless $\tM$ is one of the spaces $\SU(3)/\SO(3)$ or $\SL(3)/\SO(3)$, the space $\tM$ is of noncompact type, and then the hypersurface $M$ is minimal, and all the eigenvalues of its shape operator and its normal Jacobi operator are constant. We continue with the latter case in Section~\ref{s:solv} proving in Proposition~\ref{p:solv} that such a hypersurface $M$ is an (open domain of an) Einstein solvmanifold of codimension $1$ constructed from an Iwasawa decomposition of $\tM$, as in Theorem~\ref{th:main}\eqref{thit:solv} (Example~\ref{ex:solv}). We then proceed to classification of Einstein hypersurfaces in the spaces $\SU(3)/\SO(3)$ and $\SL(3)/\SO(3)$. In Proposition~\ref{p:su3so3gen} (Section~\ref{s:su3/so3}) we show that any such hypersurface $M$ carries a $2$-dimensional totally geodesic nullity foliation whose leaves have constant curvature, and that $M$ is locally isometric to $\bc P^2$ (respectively, to $\bc H^2$), with the same Einstein constant as $\tM$. Then we treat these two cases separately. In Section~\ref{s:su3so3}, we briefly introduce necessary facts of Legendrian geometry, and then prove Proposition~\ref{p:su3so3Leg} which says that the set of leaves of an Einstein hypersurface in $\SU(3)/\SO(3)$ is a $\frac{\pi}{2}$-special Legendrian surface $F^2$ in the ``Grassmannian" $S^5=\SU(3)/\SU(2)$ of totally geodesic spheres of $\SU(3)/\SO(3)$ (and the converse is also true), as in Theorem~\ref{th:main}\eqref{thit:wu} (Example~\ref{ex:inWu}). Finally, in Section~\ref{s:sl3so3}, after briefly introducing necessary notions of para-Legendrian and equi-affine geometries, we show in Proposition~\ref{p:sl3so3Tz} that the set of leaves of an Einstein hypersurface in $\SL(3)/\SO(3)$ is a special para-Legendrian surface $F^2$ in the ``Grassmannian" $S_-^5=\SL(3)/\SL(2)$ of totally geodesic hyperbolic planes of $\SL(3)/\SO(3)$ (and again, the converse is also true), as in Theorem~\ref{th:main}\eqref{thit:ncwu}, and then show that at generic points, such a surface $F^2$ is uniquely determined by a proper indefinite affine sphere in the equi-affine space $\br^3$ (as in Example~\ref{ex:inWu}).

The proof of the Corollary is given in the end of Section~\ref{s:sl3so3}.

Let $\tM$ be an irreducible, simply connected, globally symmetric Riemannian space of dimension $n+1$ and of rank at least $2$, and let $M$ be a connected Einstein $C^4$-hypersurface in $\tM$. Denote $\ip$ the metric tensor on $\tM$ and the induced metric tensor on $M$. Let $\on, \tR$ and $\nabla, R$ denote the Levi-Civita connections and the curvature tensors of $\tM$ and of $M$, respectively; for $x \in \tM$ and $X \in T_x\tM$, the Jacobi operator $\tR_X$ is defined by $\tR_XY=\tR(Y,X)X$ for $Y \in T_x\tM$. Let $\xi$ be a unit normal vector field of $M$. For $x \in M$ we define the shape operator $S$ on $T_xM$ by $SX=-\on_X \xi$, so that $\<\on_XY,\xi\>=\<SX,Y\>$, where $X, Y \in T_xM$.

We call a point $x \in M$ \emph{regular} if the cardinalities of the sets of eigenvalues of both $S$ and $\tR_\xi$ are constant on a neighbourhood of $x$. As both these cardinalities are lower semi-continuous functions on $M$, the set $M'$ of regular points is open and dense in $M$ (but may be disconnected). Locally on $M'$, the eigenvalues of both $S$ and $\tR_\xi$ have constant multiplicities and are $C^2$ functions, and one can choose a $C^2$ orthonormal frame of eigenvectors of $S$.

On a neighbourhood of a regular point $x \in M' \subset M$, let $X_i$ be a ($C^2$) orthonormal frame of eigenvectors of $S$, with the corresponding eigenvalues (principal curvatures) $\la_i, \; i=1, \dots, n$. Denote $H = \Tr S = \sum_{i=1}^n \la_i$ the mean curvature of $M$.

By Gauss equation, $\tR(X_i,X_k,X_k,X_j) = R(X_i,X_k,X_k,X_j) + (\la_k^2 \delta_{ik}\delta_{jk}-\la_i \la_k\delta_{ij})$. Summing up by $k$ we obtain
  \begin{gather}\label{eq:Gauss}
    \<\tR_\xi X_i, X_j\>=\alpha_i \delta_{ij}, \quad \text{where } \\
    \alpha_i=-\la_i^2 + H \la_i + C, \label{eq:Gaussal}
  \end{gather}
and where $C$ is the difference of the Einstein constants of $\tM$ and of $M$, and $\alpha_i, \; i=1, \dots, n$, are the eigenvalues of the restriction of $\tR_\xi$ to $T_xM$. Equivalently, \eqref{eq:Gauss} and \eqref{eq:Gaussal}  can be written as
  \begin{equation}\label{eq:Gaussinv}
    (\tR_\xi)_{|T_xM} = -S^2 + HS + C \id_{T_xM}.
  \end{equation}
For an eigenvalue $\al$ of the restriction of $\tR_\xi$ to $T_xM$, denote $L_{\al}$ the corresponding eigenspace. Then for every $\al$, the eigenspace $L_{\al}$ is $S$-invariant (hypersurfaces with this property are called \emph{curvature-adapted}) and
  \begin{equation}\label{eq:Gaussinval}
    S_{|L_{\al}}^2 - HS_{|L_{\al}} + (\al-C) \id_{L_{\al}} = 0.
  \end{equation}
Let $V_{\la_i}$ be the $\la_i$-eigenspace of $S$. Then by \eqref{eq:Gaussinval}, every $L_{\al}$ is the direct orthogonal sum of no more than two subspaces $V_{\la_i}, V_{\la_j}$ with $\la_i \ne \la_j$ (but with $\al_i=\al_j=\al$). We also note that in the assumptions of Theorem~\ref{th:main}, zero is always an eigenvalue of the restriction of $\tR_\xi$ to $T_xM$, of multiplicity at least $\rk \tM -1$; moreover, on every connected component of $M'$, any $\al_i$ is either identically zero or is nowhere zero.

Codazzi equation takes the form %separate?
  \begin{equation}\label{eq:codazzi}
    \tR(X_k,X_i,X_j,\xi)  = \delta_{ij} X_k(\la_j) - \delta_{kj} X_i(\la_j) + (\la_i-\la_j) \G_{ki}^{\hphantom{k}j} - (\la_k-\la_j) \G_{ik}^{\hphantom{i}j},
  \end{equation}
where $\G_{ki}^{\hphantom{k}j}=\<\nabla_k X_i, X_j\>$ (note that $\G_{kj}^{\hphantom{k}i}=-\G_{ki}^{\hphantom{k}j}$).

As $\tM$ is a symmetric space, we have $\on \, \tR = 0$, and so differentiating equation~\eqref{eq:Gauss} in the direction of $X_k$ we obtain
\begin{gather}\label{eq:dG1nR}
  \la_k \<\tR_{X_i}\xi, X_k\> = -\tfrac12 X_k(\al_i),  \\
  \G_{ki}^{\hphantom{k}j} (\alpha_j - \alpha_i) = \la_k (\tR(X_k,X_i,X_j,\xi) + \tR(X_k,X_j,X_i,\xi)),  \label{eq:dG2nR}
\end{gather}
where $i \ne j$. Furthermore, we have $\tR(X,Y) \cdot \tR = 0$, for any vector fields $X,Y$ on $\tM$, where $\tR(X,Y)$ acts as a derivation of the tensor algebra on $\tM$. In particular, acting by $\tR(X_k,\xi)$ on \eqref{eq:Gauss} we get
\begin{equation}\label{eq:RR}
  (\al_j-\al_i) \tR(X_i,X_j,X_k,\xi) + \al_k \tR(X_k,X_i,X_j,\xi) - \al_k \tR(X_j,X_k,X_i,\xi) = 0,
\end{equation}
for all $i,j,k =1, \dots , n$. In particular, if $i=j$ we get $\tR_{X_i} \xi \in \Ker \tR_\xi$.

As the symmetric space $\tM$ is irreducible, we have $\al_i = \ve \eta_i^2$ for all $i =1, \dots , n$, with some $\eta_i \in \br$ defined up to a sign, where $\ve = \pm 1$ is the sign of the Einstein constant of $\tM$.

We start with the following two Lemmas (if $\xi \in \gp$ is regular, the statement of Lemma~\ref{l:Rijk}\eqref{it:Rijknon0eta} essentially follows from the root decomposition -- see Section~\ref{ss:roots}).

\begin{lemma} \label{l:Rijk}
  In the assumptions of Theorem~\ref{th:main} \emph{(}and in the above notation\emph{)}, at any point $x \in M'$ we have the following.
    \begin{enumerate}[label=\emph{(\alph*)},ref=\alph*]
      \item \label{it:ali0}
      If $\al_i=0$, then $\tR(X_i, \xi)=0$.

      \item \label{it:Rijknon0eta}
      Suppose the triple $(\tR(X_i,X_j,X_k,\xi), \tR(X_k,X_i,X_j,\xi)$, $\tR(X_j,X_k,X_i,\xi))$ is nonzero for some $i,j,k =1, \dots , n$. Then no more than one of $\al_i, \al_j, \al_k$ is zero, and for a particular choice of signs of $\eta_i, \eta_j, \eta_k$ we have $\eta_i+\eta_j+\eta_k=0$. With this choice of signs, the triple $(\tR(X_i,X_j,X_k,\xi)$, $\tR(X_k,X_i,X_j,\xi), \tR(X_j,X_k,X_i,\xi))$ is a nonzero multiple of the triple $(\eta_k, \eta_j, \eta_i)$, and
      \begin{equation}\label{eq:etala}
      \ve \eta_i \eta_j \eta_k = \la_i \la_j \eta_k + \la_k \la_i \eta_j + \la_j \la_k \eta_i.
      \end{equation}
  \end{enumerate}
\end{lemma}
\begin{proof}
  If $\al_i=0$, then $\tR(X_i,\xi,\xi,X_i)=0$, and so $B([X_i,\xi],[X_i,\xi])=0$. Then $X_i$ and $\xi$ commute, as the Killing form is negative definite on $\gk$, which proves assertion \eqref{it:ali0}.

  To prove assertion \eqref{it:Rijknon0eta} we first note that by \eqref{it:ali0}, at least two of $\al_i, \al_j, \al_k$ (and hence at least two of $\eta_i, \eta_j, \eta_k$) must be nonzero. Permuting $i, j, k$ in \eqref{eq:RR} and using the first Bianchi identity we obtain a system of linear equations for $\tR(X_i,X_j,X_k,\xi), \tR(X_k,X_i,X_j,\xi), \tR(X_j,X_k,X_i,\xi)$ which has a nonzero solution only if $\al_i^2+\al_j^2+\al_k^2 = 2(\al_i\al_j+\al_j\al_k+\al_k\al_i)$ which gives $\pm \eta_i \pm \eta_j \pm \eta_k=0$ for some choice of the signs. Choosing the signs accordingly we get $\eta_i+\eta_j+\eta_k=0$ and then the second statement follows from \eqref{eq:RR}.

  For equation \eqref{eq:etala}, we first note that if $i=j$, then $\eta_k=0$ from the previous statement, and then $\eta_i+\eta_j=0$, so that \eqref{eq:etala} is trivially satisfied. If $i, j$ and $k$ are pairwise unequal, we can assume that $\eta_j \ne \eta_i, \eta_k$. Eliminating $\G_{ki}^{\hphantom{k}j}$ and $\G_{ij}^{\hphantom{j}k}$ from \eqref{eq:dG2nR} and \eqref{eq:codazzi} and using \eqref{it:Rijknon0eta} we get \eqref{eq:etala}.
\end{proof}

\begin{lemma} \label{l:XkH}
  In the assumptions of Theorem~\ref{th:main} \emph{(}and in the above notation\emph{)}, the following holds.
    \begin{enumerate}[label=\emph{(\alph*)},ref=\alph*]
      \item \label{it:Xkali0}
      At any point $x \in M'$ and for all $k=1, \dots, n$, such that $\al_k \ne 0$, we have $X_k(\al_i)=\tR(X_k,X_i,X_i,\xi)=0$, for all $i=1, \dots, n$.

      \item \label{it:rkRxi2}
      If $\rk \tR_\xi \le 2$ at some point $x \in M$, then $\tM=\SL(3)/\SO(3)$ or $\tM=\SU(3)/\SO(3)$.

      \item \label{it:XkH0}
      One of the following conditions is satisfied.
      \begin{enumerate}[label=\emph{(\roman*)},ref=\roman*]
        \item \label{it:XkH0gen}
        At any point $x \in M'$ and for all $k=1, \dots, n$, such that $\al_k \ne 0$, we have $X_k(H)=X_k(\la_i)=0$, for all $i=1, \dots, n$.

        \item \label{it:XkH0Wu}
        The space $\tM$ is one of the spaces $\SL(3)/\SO(3)$ or $\SU(3)/\SO(3)$, and at any point $x \in M'$ there exists $i \in\{1,2,3,4\}$, such that $\al_i=\la_i=0$.
      \end{enumerate}
  \end{enumerate}
\end{lemma}
\begin{proof}
  Assertion~\eqref{it:Xkali0} follows from \eqref{eq:RR} (with $i=j$) and \eqref{eq:dG1nR}.

  For assertion~\eqref{it:rkRxi2}, we note that the kernel of $\tR_\xi$ is the centraliser of $\xi$ in $\gp=T_x\tM$, and so is a Lie triple system. The minimal codimension $\iota(\tM)$ of a proper Lie triple system (a proper totally geodesic submanifold of the symmetric space $\tM$) is called the \emph{index} of $\tM$. By \cite[Theorem~1.1]{BO} we have $\rk \tM \le \iota(\tM)$. As $\rk \tR_\xi \le 2$, we have $\iota(\tM) \le 2$, and so $\iota(\tM)=\rk \tM = 2$. Then by \cite[Theorem~1.2]{BO}, \cite{On}, we obtain that $\tM$ is isometric to either $\SL(3)/\SO(3)$ or its compact dual $\SU(3)/\SO(3)$, as required.

  For assertion \eqref{it:XkH0}, choose and fix a point $x \in M'$. We first note that there always exists $k=1, \dots, n$, such that $\al_k \ne 0$ (for otherwise $\xi$ commutes with $T_x\tM$ which contradicts the fact that $\tM$ is irreducible). Choose and fix one such $k$. Assuming $i=1, \dots, n$, is such that $\al_i \ne \al_k$ (at least one such $i$ exists, e.g. when $\al_i=0$), assertion \eqref{it:Xkali0} and \eqref{eq:dG2nR} give $\G_{ii}^{\hphantom{i}k}=0$, and then from~\eqref{eq:codazzi} we obtain $X_k(\la_i) = 0$.

  If we suppose that $\la_i \ne 0$ for at least one $i=1, \dots, n$, with $\al_i \ne \al_k$, then \eqref{eq:Gaussal} gives $X_k(H)=0$. To prove that $X_k(\la_i) = 0$ also for those $i$ for which $\al_i=\al_k$, we note that such $\la_i$ satisfies the quadratic equation \eqref{eq:Gaussal} whose coefficients are constant along $X_k$. Thus condition \eqref{it:XkH0gen} is satisfied.

  The only remaining case is the following: for all $i=1, \dots, n$, with $\al_i \ne \al_k$, we have $\la_i=0$. As one of the $\al_i$ is always zero, we obtain that up to relabelling, for some $1 \le r < n$ we have $\al_i=\la_i=0$ when $i=1, \dots, r$, and $\al_i=\al'$ when $i=r+1, \dots, n$. Note that $C=0$ from \eqref{eq:Gaussal}, and that $\al'$ is a nonzero constant as $\tM$ is Einstein. From \eqref{eq:Gaussal} we have $\la_i^2-H\la_i+\al'=0$ for $i=r+1, \dots, n$. If $p_\pm$ are the multiplicities of the roots $\frac12 H \pm \frac12 \sqrt{H^2-4\al'}$ of this equation respectively, the fact that $\sum_{i=1}^{n} \la_i = H$ gives $(p_+ - 1)(p_- - 1)H^2=-(p_+ - p_-)^2 \al'$. The right-hand side is constant on $M$. So either $H$ is constant on $M$ (and then all $\la_i$ also are, and we again get condition \eqref{it:XkH0gen}), or $p_+=p_-=1$. But then $\rk \tR_\xi = 2$, and by assertion~\eqref{it:rkRxi2} we get condition~\eqref{it:XkH0Wu}.
\end{proof}

Furthermore, in the notation of Section~\ref{ss:Gm} we have the following.

\begin{lemma} \label{l:nons}
{\ }
    \begin{enumerate}[label=\emph{(\alph*)},ref=\alph*]
      \item \label{it:sing}
      The Gauss map is singular at a point $x \in M'$ if and only if at the point $x$ we have $\la_i=\al_i=0$ for some $i=1, \dots, n$. If the Gauss map is singular at some point of $M'$, then $C=0$.

      \item \label{it:nonsing}
      Suppose that at a point $x \in M$, the Gauss map is nonsingular. Then there exists a connected, open subset $\mU \subset M'$ such that the following holds.
      \begin{enumerate}[label=\emph{(\roman*)},ref=\roman*]
        \item \label{it:nsing1}
        At any $y \in \mU$, the Gauss map is nonsingular, and the vector $\xi(y) \in \gp=T_y\tM$ is regular \emph{(}so that there is a uniquely defined Cartan subalgebra $\ga(y) \subset T_y\tM$ containing $\xi(y)$\emph{)}; moreover, for any $i=1, \dots, n$, with $\al_i \ne 0$, the eigenspace $L_{\al_i}$ of $\tR_\xi$ is a root subspace corresponding to a restricted root defined by $\ga(y)$.

        \item \label{it:nsing2} %\nabla or d for functions? here and everywhere
        For any $i=1, \dots, n$, we either have $\al_i(y)=0$ for all $y \in \mU$, or $\nabla \al_i(y) \ne 0$ for all $y \in \mU$. Moreover, for all $i=1, \dots, n$ and all $y \in \mU$, we have $\la_i(y) \ne 0$.
      \end{enumerate}
  \end{enumerate}
\end{lemma}
\begin{remark} \label{rem:nsroots}
  In the assumption of assertion~\eqref{it:nonsing} of Lemma~\ref{l:nons}, for $x \in \mU$, let $\ga \subset \gp = T_x\tM$ be the Cartan subalgebra containing $\xi$. Then, as $\xi$ is regular, we have $L_0=\ga \cap \xi^\perp$. Furthermore, if $\gamma$ is a restricted root and $\gp_\gamma$ is the corresponding root space, we have $\tR_\xi X = \ve \<\gamma, \xi\>^2 X$, for all $X \in \gp_\gamma$ (see Section~\ref{ss:roots}), so that $\gp_\gamma \subset L_\al$, with $\al=\ve \<\gamma, \xi\>^2$. The last statement of assertion~\eqref{it:nonsing}\eqref{it:nsing1} says that in fact $\gp_\gamma = L_\al$, that is, for no two restricted roots $\gamma_1 \ne \pm \gamma_2$, we can have $\<\gamma_1, \xi\> = \pm \<\gamma_2, \xi\>$. As for every $i=1, \dots, n$, we have $X_i \in L_{\al_i}$ by \eqref{eq:Gaussinval}, this fact implies that for every $i=1, \dots, n$, such that $\al_i \ne 0$ there is a well defined (up to a sign) root vector $\beta_i \in \Delta$ such that $X_i \in \gp_{\beta_i}$. Then $\al_i=\ve \<\beta_i, \xi\>^2$ and also $\tR_\xi X_i = \ve \<\beta_i, \xi\>^2 X_i$ and $\tR(X_i, \xi) X_k=\ve \<\beta_i, \xi\>\<\beta_i, X_k\>X_i$, for all $X_k \in L_0$.
\end{remark}
\begin{proof}
  By Lemma~\ref{l:dgamma}, the Gauss map is singular at $x \in M$ if and only if there exists a nonzero vector $Y \in T_xM$ such that $Y \perp \on_X \xi$, for all $X \in T_xM$, and $Y \perp [U, \xi]$, for all $U \in \gk$. The first condition is equivalent to $\on_Y \xi = 0$, and the second one, to $[\xi, Y]=0$, and so, to $\tR_\xi Y=0$. Hence $Y$ must be an eigenvector of both $\tR_\xi$ and $S$ with both corresponding eigenvalues being zeros. Then $C=0$ by \eqref{eq:Gaussal}. This proves assertion~\eqref{it:sing}.

  For assertion~\eqref{it:nonsing}, let $\mV \subset \gp$ the set of unit regular vectors. Every $T \in \mV$ defines a unique Cartan subalgebra $\ga \ni T$ and the root system $\Delta \subset \ga$. Let $\mV' \subset \mV$ be the set of vectors $T \in \mV$ which do not lie in any hyperplane $\<\gamma_1-\gamma_2, T\>=0$ where $\gamma_1, \gamma_2 \in \Delta, \, \gamma_1 \ne \gamma_2$. Then $\mV'$ is open and dense in the unit sphere of $\gp$. Suppose the Gauss map is nonsingular at $x \in M$ (as the set of such points is open, we can assume that $x \in M'$). Then the Gauss image of a (connected) neighbourhood of $x$ in $M'$ is open in the unit sphere of $\gp=T_x\tM$ and so there exists an open, connected subset $\mU$ of that neighbourhood such that for all $y \in \mU$, the Gauss map is still nonsingular at $y$ and that $\xi(y) \in \gp =T_y\tM$ lies in $\mV'$. This proves~\eqref{it:nsing1}, as from $\xi(y) \in \mV'$ it follows that for no $\al_i \ne 0$ the subspace $L_{\al_i}$ may contain more than one subspace $\gp_\gamma$ (see Remark~\ref{rem:nsroots}).

  Furthermore, suppose that for some $i=1, \dots, n$, the subset of $\mU$ on which $\nabla \al_i = 0$ has a nonempty interior $\mU_0$. Then for some $c \in \br$, we have $\al_i(y)=c$ on a connected component of $\mU_0$. For any $c \in \br$, the set of unit vectors $T \in \gp$ such that $\tR_T$ has eigenvalue $c$ is Zariski closed (satisfies a certain $\Ad_\rK$-invariant polynomial equation), and its complement is nonempty for any $c \ne 0$, as $\rk \tM > 1$. As the Gauss map is nonsingular at any $y \in \mU_0 \subset \mU$, we must have $c=0$. Replacing $\mU$ by a smaller subset if necessary we can assume that for any $i=1, \dots, n$, either $\al_i=0$ on $\mU$, or $\nabla \al_i$ is nowhere $0$ on $\mU$, as required. Now suppose that for some $i=1, \dots, n$, the zero set of $\la_i$ on $\mU$ has a nonempty interior $\mU_1$. Then by \eqref{eq:Gaussal}, $\al_i$ is constant on $\mU_1$, and so $\al_i=0$. But this contradicts assertion~\eqref{it:sing}, as at all points $y \in \mU_1 \subset \mU$, the Gauss map is nonsingular. This proves~\eqref{it:nsing2} (by replacing $\mU$ by a smaller open, connected subset if necessary).
\end{proof}

\section{Nonsingular Gauss map}
\label{s:reg}

\begin{proposition} \label{p:allsing}
  Let $\tM$ be an irreducible symmetric space of rank at least $2$, and let $M$ be an Einstein $C^4$-hypersurface in $\tM$. Then at any point of $M$ the Gauss map \emph{(}as defined in Section~\ref{ss:Gm}\emph{)} is singular.
\end{proposition}
\begin{proof}
Suppose there exists a point on $M$ at which the Gauss map is nonsingular. We choose an open, connected subset $\mU \subset M'$ as in Lemma~\ref{l:nons}\eqref{it:nonsing} and take $x \in \mU$. Let $\ga \subset \gp=T_x\tM$ be the Cartan subalgebra containing $\xi(x)$. Denote $\ga'=\ga \cap \xi^\perp$ (note that $\ga'=L_0$). For $i=1, \dots, n$ such that $\al_i \ne 0$, we have well-defined (up to a sign) root vectors $\beta_i \in \ga$ such that $\al_i=\ve \<\beta_i, \xi\>^2$ (see Remark~\ref{rem:nsroots}).

\begin{lemma} \label{l:reg} % maybe restructure, what is needed outside: 1: reg1 (if need)and {eq:regXkH}, 2: reg3, 3: as 4, pibeta there
  In the assumptions of Proposition~\ref{p:allsing} and in the above notation, the following holds.
    \begin{enumerate}[label=\emph{(\alph*)},ref=\alph*]
      \item \label{it:reg1}
      Let $x \in \mU$ and let $X_i, X_k$ be such that $\al_i \ne 0,\, \al_k=0$ \emph{(}so that $X_k \in \ga'$\emph{)}. Then $X_k(\la_i)=\<\beta_i,\xi\>^{-1} \<\beta_i,X_k\>$ $(C+\la_i(H-\la_k))$.

      \item \label{it:reg2}
      Let $x \in \mU$ and suppose $\al_i \ne 0$. Then
      \begin{equation}  \label{eq:pibeta}
        \begin{gathered}
        c_i \pi_{\ga'} \beta_i = -\la_i \al_i^{-1} \<\beta_i, \xi\> ((H\la_i+2C)S' +(2C \la_i-CH-2H\al_i) \id_{\ga'}) \nabla H,\\
        \text{where } c_i=(2H^3+8CH) \la_i + C H^2 - 4 H^2 \al_i+4 C^2
        \end{gathered}
      \end{equation}
      and $S'$ is the restriction of $S$ to $\ga'$, $\nabla H$ is the gradient of $H$ and $\pi_{\ga'}$ is the orthogonal projection to $\ga'$ \emph{(}note that $\nabla H \in \ga'$ by Lemma~\ref{l:XkH}\eqref{it:XkH0} and Lemma~\ref{l:nons}\eqref{it:sing}\emph{)}.

      \item \label{it:reg3}
      There exists an open subset $\mU' \subset \mU$ such that for all $y \in \mU'$ we have $H \ne 0, \; \nabla H \ne 0$, and moreover, if $\al_i=\al_j \ne 0$, then $\la_i=\la_j$, for all $i,j=1, \dots, n$.

      \item \label{it:regrkle3}
      $\rk \tM \in \{2, 3\}$.
  \end{enumerate}
\end{lemma}
\begin{proof}
For assertion~\eqref{it:reg1}, let $\al_k=0 \ne \al_i$. Then from \eqref{eq:dG2nR} we get $\G_{ii}^{\hphantom{i}k} = \la_i \alpha_i^{-1} \tR(X_k,X_i,X_i,\xi)$, and so $X_k(\la_i)=(1-(\la_k-\la_i)\la_i \alpha_i^{-1})\tR(X_k,X_i,X_i,\xi)$ by \eqref{eq:codazzi}. As $\al_i=\ve \<\beta_i, \xi\>^2$ and $\tR(X_i, \xi) X_k=\ve \<\beta_i, \xi\>\<\beta_i, X_k\>X_i$ by Remark~\ref{rem:nsroots}, the claim follows from~\eqref{eq:Gaussal}.

To prove assertion~\eqref{it:reg2}, take $\al_k=0 \ne \al_i$. Differentiating \eqref{eq:Gaussal} in the direction of $X_k$ and substituting $X_k(\al_i)$ from~\eqref{eq:dG1nR} we get $(H-2\la_i) X_k(\la_i) + \la_i X_k(H) = -2 \la_k \tR(X_k,X_i,X_i,\xi)$, and so we obtain $(-\la_i (H^2 + 2C) +H(2\al_i-C)+ (H\la_i+2C) \la_k )\tR(X_k,X_i,X_i,\xi) = -\la_i \al_i X_k(H)$ by \eqref{eq:Gaussal}, which gives
\begin{equation} \label{eq:regXkH}
  (-\la_i (H^2 + 2C) +H(2\al_i-C) + (H\la_i+2C) \la_k ) \<\beta_i, X_k\> = -\la_i \<\beta_i, \xi\> X_k(H).
\end{equation}
As this is satisfied for all $X_k \in \ga'$ we get $((-\la_i (H^2 + 2C) +H(2\al_i-C)) \id_{\ga'} + (H\la_i+2C) S' ) \pi_{\ga'} \beta_i = -\la_i \<\beta_i, \xi\> \nabla H$. Acting by $S'$ on both sides, using the fact that $S'^2=HS'+C\id_{\ga'}$ which follows from \eqref{eq:Gaussinv} and eliminating $S' \pi_{\ga'} \beta_i$ we obtain~\eqref{eq:pibeta}, as required. % which proves assertion~\eqref{it:reg2}.

Suppose $\mU_0 \subset \mU$ is an open subset on which $H=0$. Then $C \ne 0$ (as otherwise $\la_k=0$ on $\mU_0$ by \eqref{eq:Gaussal}, in contradiction with Lemma~\ref{l:nons}\eqref{it:nonsing}\eqref{it:nsing2}), and from \eqref{eq:pibeta} we obtain $\pi_{\ga'} \beta_i = 0$, a contradiction, as the root vectors $\beta_i$ span $\ga$. Suppose $\mU_1 \subset \mU$ is a connected, open subset on which $\nabla H=0$ and $H \ne 0$. Then for all $i$ such that $\al_i \ne 0$ and $\pi_{\ga'} \beta_i \ne 0$, equation \eqref{eq:pibeta} gives $c_i=0$. Using \eqref{eq:Gaussal} to eliminate $\la_i$ we obtain $16 H^4 \al_i^2 + 4 H^2 (H^2+4C)(-H^2+2C) \al_i + C(-2H^2+C) (H^2+4C)^2 =0$, and so $\al_i$ is locally a constant, a contradiction with Lemma~\ref{l:nons}\eqref{it:nonsing}\eqref{it:nsing2} (or again, with the fact that the root vectors $\beta_i$ span $\ga$). It follows that there exists an open subset $\mU' \subset \mU$ on which both $H$ and $\nabla H$ are nowhere zero.

Suppose at some point of $\mU'$ we have $\al_i=\al_j \ne 0$ (and so $\beta_i=\beta_j$) and $\la_i \ne \la_j$. Then from~\eqref{eq:regXkH} we get $\la_i ((-H^2 - 2C + H \la_k) \<\beta_i, X_k\> + \<\beta_i, \xi\> X_k(H)) +(H(2\al_i-C)  +2C \la_k ) \<\beta_i, X_k\> = 0$ and the same equation, with $\la_i$ replaced by $\la_j$. It follows that $(-H^2 - 2C + H \la_k) \<\beta_i, X_k\> + \<\beta_i, \xi\> X_k(H) = (H(2\al_i-C)  +2C \la_k)\<\beta_i, X_k\> = 0$. As $\nabla H \ne 0$ from the above, we can choose $X_k \in \ga'$ such that $X_k(H) \ne 0$. If $\<\beta_i, X_k\>=0$, then $\<\beta_i, \xi\>=0$, and so $\al_i=0$. It follows that $H(2\al_i-C)  +2C \la_k=0$, and so $C \ne 0$ and $\la_k=-\frac12 C^{-1}H(2\al_i-C)$. Substituting into~\eqref{eq:Gaussal} we get
\begin{equation}\label{eq:reg3alphai}
C^2 H^2 - 4 H^2 \al_i^2 + 4 C^3=0.
\end{equation}
Furthermore, from assertion~\eqref{it:reg1} we have $X_k(\la_i)=\<\beta_i,\xi\>^{-1} \<\beta_i,X_k\> (C+\la_i(H-\la_k))$ and  $X_k(\la_j)=\<\beta_i,\xi\>^{-1} \<\beta_i,X_k\> (C+\la_j(H-\la_k))$. As $\la_i+\la_j=H$ by~\eqref{eq:Gaussal} we get $X_k(H)=\<\beta_i,\xi\>^{-1} \<\beta_i,X_k\> (2C+H(H-\la_k))$. Also, from~\eqref{eq:dG1nR} we have $X_k(\al_i)=  -2 \la_k \ve \<\beta_i,\xi\>^{-1} \<\beta_i,X_k\>$, and so differentiating \eqref{eq:reg3alphai} and substituting $\la_k=-\frac12 C^{-1}H(2\al_i-C)$ from the above we obtain $(C-2 \al_i) (12 H^2 \al_i^2 + 4 C (H^2 + 2 C) \al_i + C^2 (H^2 + 4 C))=0$. As $\al_i$ is non-constant on $\mU'$ by Lemma~\ref{l:nons}\eqref{it:nonsing}\eqref{it:nsing2}, the second term on the left-hand side must be zero. Eliminating $\al_i$ from it and \eqref{eq:reg3alphai} we get $(H^2 + 4 C) (3 H^4 + 12 C H^2 - 4 C^2)=0$, and so $H$ is locally a constant, a contradiction. This proves assertion~\eqref{it:reg3}.

To prove assertion~\eqref{it:regrkle3}, take a point in $\mU'$.  Note that by \eqref{eq:pibeta}, for every $i$ such that $\al_i \ne 0$, we have either $c_i=0$, or $\pi_{\ga'} \beta_i \in \Span(\nabla H, S'\nabla H)$, that is, $\beta_i \in \Span(\xi, \nabla H, S' \nabla H)$. If $c_i=0$, then eliminating $\la_i$ by~\eqref{eq:Gaussal} we obtain (similar to the above) that $\al_i$ satisfies the equation $16 H^4 \al_i^2 + 4 H^2 (H^2+4C)(-H^2+2C) \al_i + C(-2H^2+C) (H^2+4C)^2 =0$, and so can take no more than two different values (as $H \ne 0$ by \eqref{it:reg3}). As by Lemma~\ref{l:nons}\eqref{it:nonsing}\eqref{it:nsing1} and Remark~\ref{rem:nsroots}, different nonzero $\alpha_i$ correspond to different root vectors $\pm \beta_i$, we deduce that in both cases ($c_i=0$ and $c_i \ne 0$), all but at most two root vectors $\beta_i$ lie in the space $\Span(\xi, \nabla H, S'\nabla H)$ of dimension at most $3$. This is not possible for a root system of rank greater than $3$.
\end{proof}

By Lemma~\ref{l:reg}\eqref{it:regrkle3}, it remains to consider the cases when $\rk \tM \in \{2,3\}$. The proof in both cases follows the same scheme. Choose an arbitrary point $y$ in the open set $\mU'$ defined in Lemma~\ref{l:reg}\eqref{it:reg3}. We consider the root systems of rank $2$ and $3$ (it suffices to consider only $\mathrm{A}_3, \mathrm{A}_2$ and $\mathrm{B}_2$) and compute $\al_i=\ve \<\beta_i,\xi\>$. Then equations \eqref{eq:Gaussal} and \eqref{eq:etala} (and \eqref{eq:regXkH} when $\rk \tM = 2$) give an overdetermined system of polynomial equations for $\la_i, H$ and the components of $\xi$ relative to a basis for $\ga$. After elimination, in the case $\rk \tM = 2$ we find that $H$ satisfies a non-trivial polynomial equation with constant coefficients depending on $C$ and $\ve$, and so $H$ is locally constant, in contradiction with Lemma~\ref{l:reg}\eqref{it:reg3}. Similarly, in the case $\rk \tM = 3$ we obtain that the components of $\xi$ satisfy a non-trivial polynomial equation with constant coefficients depending on $C$ and $\ve$, in contradiction with the fact that the Gauss map is nonsingular on $\mU'$.

\smallskip

% A3.mw
Suppose $\rk \tM = 3$. Then the restricted root system $\Delta$ is of one of the types $\mathrm{A}_3 \, (=\mathrm{D}_3), \mathrm{B}_3, \mathrm{C}_3$ or $\mathrm{BC}_3$. In all these cases, $\Delta$ contains a subsystem $\Delta'$ of type $\mathrm{A}_3$. Up to scaling, we can choose an orthonormal basis $e_1,e_2,e_3$ for $\ga$ such that $\Delta' = \{\pm e_1 \pm e_2, \pm e_2 \pm e_3, \pm e_3 \pm e_1\}$. By Lemma~\ref{l:nons}\eqref{it:nonsing}\eqref{it:nsing2} (see also Remark~\ref{rem:nsroots}), the corresponding root spaces are eigenspaces of $\tR_\xi$. As by \eqref{eq:Gaussinval}, every eigenspace of $S$ lies in an eigenspace of $\tR_\xi$, we can choose six eigenvectors $X_i$ of $S$ which lie in the corresponding root spaces. It will be more convenient to label them $X_i^+, X_i^-, \; i=1,2,3$, in such a way that $X_i^+ \in \gp_{e_j+e_k}, \; X_i^- \in \gp_{e_j-e_k}$, where $(i,j,k)$ is a cyclic permutation of $(1,2,3)$. We label corresponding eigenvalues of $S$ and of $\tR_\xi$ by $\la_i^+, \la_i^-$ and $\al_i^+, \al_i^-$, respectively. Let $\xi_i=\<\xi, e_i\>$. Then $\al_i^+ = \ve (\xi_j+\xi_k)^2, \, \al_i^- = \ve (\xi_j-\xi_k)^2$, where $(i,j,k)$ is a cyclic permutation of $(1,2,3)$, and so from \eqref{eq:Gaussal} we obtain
\begin{equation}\label{eq:rk3G}
  \ve (\xi_j+\xi_k)^2 = -(\la_i^+)^2 + H \la_i^+ + C, \quad \ve (\xi_j-\xi_k)^2 = -(\la_i^-)^2 + H \la_i^- + C.
\end{equation}
Furthermore, by Lemma~\ref{l:rr}\eqref{it:rr1} we have $\tR(\xi, X_i^+)X_j^+ \ne 0$ when $i \ne j$. But $\tR(\xi, X_i^+)X_j^+ \in \gp_{e_i+e_j+2e_k} \oplus \gp_{e_i-e_j}$ and as neither of the above root systems $\Delta$ contains the element $e_i+e_j+2e_k$, we obtain that $\tR(\xi, X_i^+)X_j^+$ is a nonzero vector from $\gp_{e_i-e_j}$. As by Lemma~\ref{l:reg}\eqref{it:reg3}, all the vectors in every root space $\gp_\gamma$ are eigenvectors of $S$ with the same eigenvalue, we can assume, by a small perturbation, that our eigenvectors are chosen in such a way that $\tR(\xi, X_i^+,X_j^+,X_k^-) \ne 0$, for all $\{i,j,k\}=\{1,2,3\}$. Then by Lemma~\ref{l:Rijk}\eqref{it:Rijknon0eta} we can choose $\eta_i=\xi_j+\xi_k, \, \eta_j = -(\xi_k+\xi_i), \, \eta_k=\xi_i-\xi_j$, and then \eqref{eq:etala} gives
\begin{equation*}
 -\ve (\xi_j+\xi_k) (\xi_k+\xi_i) (\xi_i-\xi_j) = \la_i^+ \la_j^+ (\xi_i-\xi_j) - \la_k^- \la_i^+ (\xi_k+\xi_i) + \la_j^+ \la_k^- (\xi_j+\xi_k),
\end{equation*}
for $\{i,j,k\}=\{1,2,3\}$. The rest of the proof is a straightforward computation. We first eliminate $\la_k^-, \; k=1,2,3$, from the above equation and the second equation of \eqref{eq:rk3G} (and their cyclic permutations), and then eliminate $\la_k^+, \; k=1,2,3$, from the three resulting equations and the first equation of \eqref{eq:rk3G} and its cyclic permutations (and cancel all the factors of the form $\xi_i \pm \xi_j$ which cannot be zero by Lemma~\ref{l:nons}\eqref{it:nonsing}\eqref{it:nsing1}). We get three polynomial equations for $\xi_1,\xi_2,\xi_3,H$ and $C$ which are obtained from one another by a cyclic permutation of $\xi_1,\xi_2,\xi_3$. Subtracting one of them from another (and using the fact that $H \ne 0$ by Lemma~\ref{l:reg}\eqref{it:reg3}) we get $8((\xi_1+\xi_2)^2H^2+\ve C^2)((\xi_3^2-\xi_1\xi_2)H^2+\ve C^2) - C^2H^2(H^2+2C)=0$. Subtracting from this equation the same equation with $\xi_1,\xi_2,\xi_3$ being cyclically permuted we get $(\|\xi\|^2 + 3(\xi_1\xi_2 + \xi_2\xi_3 + \xi_3\xi_1)) H^2 = \ve C^2$, and so eliminating $H$ from the last two equation (and using the fact that $\|\xi\|=1$) we obtain that the components of $\xi$ satisfy the equation $8 \sigma_1 \sigma_3 + 20 \sigma_1^4 - (3 C \ve + 4) \sigma_1^2 + \ve C  - C^2 =0$, where $\sigma_1=\xi_1 + \xi_2 + \xi_3, \; \sigma_3=\xi_1 \xi_2 \xi_3$. But this contradicts the fact that the Gauss map is nonsingular on $\mU'$, as the Gauss image has to have a nonempty interior in the unit sphere of $\gp$.
%$-\ve (8 \xi_1^3 \xi_2 + 16 \xi_1^2 \xi_2^2 - 8 \xi_1^2 \xi_3^2 + 8 \xi_1 \xi_2^3 - 16 \xi_1 \xi_2 \xi_3^2 - 8 \xi_2^2 \xi_3^2 + C^2) H^4 - 2 C^2 (C \ve - 4 \xi_1^2 - 4 \xi_1 \xi_2 - 4 \xi_2^2 - 4 \xi_3^2) H^2 + 8 C^4 \ve = 0.$

\smallskip

% B2again.mw
Now suppose $\rk \tM = 2$. Then the restricted root system $\Delta$ belongs to one of the types $\mathrm{A}_2, \mathrm{G}_2$ or $\mathrm{B}_2 \, (= \mathrm{C}_2), \mathrm{BC}_2$. In the last two cases, $\Delta$ contains a subsystem $\Delta'$ of type $\mathrm{B}_2$; in case of $\mathrm{BC}_2$, we choose for $\Delta'$ the set of the longest and the second longest roots. Our argument is similar to that in the case $\rk \tM=3$. Up to scaling, we can choose an orthonormal basis $e_1,e_2$ for $\ga$ such that $\Delta' = \{\pm e_1, \pm e_2, \pm e_1 \pm e_2\}$. By Lemma~\ref{l:nons}\eqref{it:nonsing}\eqref{it:nsing1} (and Remark~\ref{rem:nsroots}), we can choose four eigenvectors $X_i, \; i=1,2,3,4$, of $S$ in such a way that $X_1 \in \gp_{e_1}, \, X_2 \in \gp_{e_2}, \, X_3 \in \gp_{e_1+e_2}, \, X_4 \in \gp_{e_1-e_2}$. Let $\xi= \cos \theta e_1 + \sin \theta e_2$. Then $\al_1 = \ve \cos^2 \theta, \, \al_2 = \ve \sin^2 \theta, \, \al_3 = \ve (\cos \theta+\sin \theta)^2, \, \al_4 = \ve (\cos \theta-\sin \theta)^2$.

Next, by Lemma~\ref{l:rr}\eqref{it:rr1} we have $\tR(\xi, X_1)X_3 \ne 0$, and so $\tR(\xi, X_1)X_3$ is a nonzero vector in $\gp_{e_2}$ (as $\Delta$ does not contain the vector $2e_1+e_2$). As by Lemma~\ref{l:reg}\eqref{it:reg3}, all the vectors in every root space $\gp_\gamma$ are eigenvectors of $S$ with the same eigenvalue, we can assume, by a small perturbation, that $\tR(\xi, X_1,X_3,X_2) \ne 0$. Then by Lemma~\ref{l:Rijk}\eqref{it:Rijknon0eta} we take $\eta_1=-\cos \theta, \, \eta_3 = \cos \theta+\sin \theta, \, \eta_2=-\sin \theta$, and so \eqref{eq:etala} gives $\ve \cos \theta \sin \theta (\cos \theta+\sin \theta) = -\la_1 \la_3 \sin \theta - \la_2 \la_3 \cos \theta + \la_1 \la_2 (\cos \theta+\sin \theta)$. Applying the similar argument starting with $\tR(\xi, X_1)X_4 \ne 0$, we obtain $-\ve \cos \theta \sin \theta (\cos \theta-\sin \theta) = \la_1 \la_4 \sin \theta - \la_2 \la_4 \cos \theta + \la_1 \la_2 (\cos \theta-\sin \theta)$. We now eliminate $\la_3, \la_4$ from these two equations and the corresponding equations \eqref{eq:Gaussal}, with $\al_3, \al_4$ as above, and then eliminate $\la_1,\la_2$ from the resulting two equations and the corresponding equations \eqref{eq:Gaussal}, using $\al_1, \al_2$ given above (recall that $\cos \theta, \sin \theta, \cos \theta \pm \sin \theta \ne 0$, as $\al_i \ne 0$ for $i=1,2,3,4$). Eliminating $\cos 4\theta$ we get
$\big((2 C^3+4 \ve(C^2 -1)) H^6 + C^2 (15 C^2+ 8 C\ve-20) H^4 + 8 C^4(3C - 4\ve) H^2 - 16 C^6\big)\big((2 C^2-3) H^4 + 4 C^2 (C-4\ve) H^2 - 16 C^4\big)=0$, and so $H$ is locally constant, in contradiction with Lemma~\ref{l:reg}\eqref{it:reg3}.

% A3again.mw the bottom case (starting from f1:=) note that we showed there that \la_0 is constant
The last case to consider is when $\rk \tM = 2$ and the restricted root system $\Delta$ is of one of the types $\mathrm{A}_2$ or $\mathrm{G}_2$. In both cases, we can choose a subsystem $\Delta'$ of type $\mathrm{A}_2$; in the first case, $\Delta'=\Delta$, and in the second, $\Delta'$ is the set of long roots of $\Delta$. Up to scaling, we choose an orthonormal basis $e_1,e_2$ for $\ga$ such that $\Delta' = \{\pm \beta_1, \pm \beta_2, \pm \beta_3\}$, where $\beta_1=e_1$, $\beta_2=-\frac12 e_1+\frac{\sqrt{3}}2 e_2, \; \beta_3=-\frac12 e_1-\frac{\sqrt{3}}2 e_2$. We then choose eigenvectors $X_i \in \gp_{\beta_i}$ of $S$, with corresponding eigenvalues $\la_i, \; i=1,2,3$. Let $\xi= \cos \theta e_1 + \sin \theta e_2$. From \eqref{eq:Gaussal} we get three equations with $\al_i=\ve \<\beta_i,\xi\>^2$. Furthermore, by Lemma~\ref{l:rr}\eqref{it:rr1} and Lemma~\ref{l:reg}\eqref{it:reg3}, we can assume that $\tR(\xi, X_1,X_2,X_3) \ne 0$, and so $\la_1, \la_2, \la_3$ satisfy equation \eqref{eq:etala} with $\eta_i=\<\beta_i,\xi\>, \; i=1,2,3$. Eliminating $\la_1,\la_2,\la_3$ from these four equations we get $H^6 \cos(6\theta)+ (1-4 \ve C^3 - 6 C^2)H^6 + 6C^2(- 5 \ve C^2 - 3 C  + 3 \ve) H^4 + 48 C^4(1 - \ve C) H^2 + 32 \ve C^6 =0$.

We then consider equations~\eqref{eq:regXkH} for $i=1,2,3$. Note that there is a unique (up to a sign) choice of $X_k$ such that $\al_k=0$, as $\rk \tM = 2$. We take $X_k=-\sin \theta e_1 + \cos \theta e_2$. Multiply both sides of \eqref{eq:regXkH} with $i=1$ by $\la_2 \la_3$ and take the cyclic sum by $(1,2,3)$. As $\beta_1+\beta_2+\beta_3=0$ we get $\sigma\big((H(2\al_1 - C)+2C\la_k)\la_2\la_3 \<\beta_1,X_k\>\big)=0$, where $\sigma$ is the sum by cyclic permutations of $(1,2,3)$. Eliminating $\la_1, \la_2, \la_3$ using \eqref{eq:Gaussal} we get a polynomial equation in $H, \la_k, C$ and $\cos(6\theta)$. We can then eliminate $\cos(6 \theta)$ using the equation in the end of the previous paragraph, and then eliminate $\la_k$, using \eqref{eq:Gaussal} with $\al_k=0$. We obtain a polynomial in $H$ with constant coefficients depending only on $C$ and $\ve$. Computing the corresponding resultants we find that these coefficients cannot simultaneously be zeros, and so $H$ is locally a constant, in contradiction with Lemma~\ref{l:reg}\eqref{it:reg3}. This completes the proof of Proposition~\ref{p:allsing}.
\end{proof}

\section{Singular Gauss map}
\label{s:singg}

From Proposition~\ref{p:allsing} we know that the Gauss map at all points of an Einstein hypersurface $M$ is singular. In this section, we prove the following.

\begin{proposition} \label{p:singH=0}
Let $\tM$ be an irreducible symmetric space of rank at least $2$, and let $M$ be a connected Einstein $C^4$-hypersurface in $\tM$ \emph{(}then by Proposition~\ref{p:allsing} and Lemma~\ref{l:nons}\eqref{it:sing} we have $C=0$ and $\la_s=\al_s=0$ locally on $M'$, for some $s=1, \dots, n$\emph{)}. Then one of the following holds.
  \begin{enumerate}[label=\emph{(\Alph*)},ref=\Alph*]
    \item \label{it:eithersl3so3}
    $\tM=\SL(3)/\SO(3)$ or $\tM=\SU(3)/\SO(3)$. %and $\al_s$ and $\la_s$ are as given in Lemma~\ref{l:XkH}\eqref{it:XkH0},

    \item \label{it:orH=0}
    We have $H=0$ and all the eigenvalues of $\tR_\xi$ \emph{(}the functions $\al_i$\emph{)} are constant on $M$; moreover, $\ve=-1$, so that $\tM$ is of noncompact type. Furthermore, by \eqref{eq:Gaussinv}, the eigenvalues of $S^2$ are constant. If $M$ is orientable, then the eigenvalues of $S$ \emph{(}the functions $\la_i$\emph{)} relative to a particular choice of the global unit normal field $\xi$ are also constant.
  \end{enumerate}
\end{proposition}

In Proposition~\ref{p:solv} (Section~\ref{s:solv}) we will prove that in the second case, $M$ is a codimension $1$ solvmanifold, as in Theorem~\ref{th:main}\eqref{thit:solv}. The proofs in the cases $\tM = \SL(3)/\SO(3)$ and $\tM=\SU(3)/\SO(3)$ are given in Sections~\ref{s:su3/so3}, \ref{s:su3so3} and \ref{s:sl3so3}.

We start with the following technical fact.
\begin{lemma} \label{l:lak0H}
In the assumptions of Proposition~\ref{p:singH=0}, at any point $x \in M'$, the following holds.
  \begin{enumerate}[label=\emph{(\alph*)},ref=\alph*]
    \item \label{it:Hne0}
    Suppose $H \ne 0$. If $\tR(X_k,X_i,X_j,\xi) \ne 0$, then \emph{(}up to interchanging $k$ and $i$\emph{)} we have $\al_k=0$ and $\al_i=\al_j \ne 0$. In particular, $\tR(L_{\al_k},L_{\al_i},L_{\al_j},\xi)=0$ if $\al_k\al_i \ne 0$.

    \item \label{it:nolakH}
    If $\al_k=0$, then $\la_k = 0$. % so ker s = ker r xi|xi perp

    \item \label{it:lak0}
    If $\al_k=0, \; \al_i \ne 0$ \emph{(}and so by \eqref{eq:Gaussal}, $\la_i \ne 0$, and by \eqref{it:nolakH}, $\la_k=0$\emph{)}, then $X_k(\al_i)=0$ and
    \begin{equation*}
      X_k(\la_i)= \al_i^{-1} \la_i H \tR(X_k,X_i,X_i,\xi), \qquad X_k(H)= -\al_i^{-1} H (H-2\la_i) \tR(X_k,X_i,X_i,\xi).
    \end{equation*}
  \end{enumerate}
\end{lemma}
\begin{proof}
  For assertion~\eqref{it:Hne0}, suppose that $\tR(X_k,X_i,X_j,\xi) \ne 0$ for some triple $i, j, k = 1, \dots, n$. By Lemma~\ref{l:Rijk}\eqref{it:ali0} we have $\al_j \ne 0$. First assume $j \in \{i,k\}$, say $j=i \ne k$. Then $\al_k=0$ by Lemma~\ref{l:XkH}\eqref{it:Xkali0}, and $\al_i \ne 0$ by Lemma~\ref{l:Rijk}\eqref{it:Rijknon0eta}, as required. Now assume $j \notin \{i,k\}$. Then by Lemma~\ref{l:Rijk}\eqref{it:Rijknon0eta}, we have $\ve \eta_i \eta_j \eta_k = \la_i \la_j \eta_k + \la_k \la_i \eta_j + \la_j \la_k \eta_i$ (equation \eqref{eq:etala}), where $\eta_i+\eta_j+\eta_k=0$ and $\ve \eta_s^2=\al_s$ for $s \in \{i,j,k\}$ (and so $\ve \eta_s^2=-\la_s^2+H\la_s$ by \eqref{eq:Gaussal}). Suppose $\eta_i \eta_j \eta_k \ne 0$. Multiplying both sides of \eqref{eq:etala} by $\ve(\la_i-H)$ we get $\la_i \eta_j \eta_k + \la_k \eta_i \eta_j + \la_j \eta_k \eta_i-\ve \la_i \la_j \la_k = H(\eta_j \eta_k-\ve \la_j \la_k)$. As $H \ne 0$ we obtain that $\la_j \la_k=\ve\eta_j \eta_k+t, \; \la_k \la_i=\ve\eta_k \eta_i+t, \; \la_i \la_j=\ve\eta_i \eta_j+t$ for some $t \in \br$, and so $\eta_i \eta_j \eta_k = 0$ from \eqref{eq:etala} and $\eta_i+\eta_j+\eta_k=0$. As by Lemma~\ref{l:Rijk}\eqref{it:Rijknon0eta} no more than one of $\eta_i, \eta_j, \eta_k$ can be zero (and $\al_j=\ve\eta_j^2 \ne 0$) we can assume by interchanging $i$ and $k$ if necessary that $\eta_k=\al_k=0$ and $\eta_i \ne 0$. Then $\eta_j=-\eta_i$, and so $\al_i=\al_j \ne 0$, as required. The second statement of assertion~\eqref{it:Hne0} follows by linearity.

  To prove~\eqref{it:nolakH}, suppose $\al_k=0$. Then by \eqref{eq:Gaussal} we have $\la_k \in \{0,H\}$. If $H = 0$, there is nothing to prove. Suppose $H \ne 0$ and $\la_k=H$. If for some $i, j =1, \dots, n$ we have $\tR(X_k,X_i,X_j,\xi) \ne 0$, then by assertion~\eqref{it:Hne0} $\al_i=\al_j$. By Lemma~\ref{l:Rijk}\eqref{it:ali0} and the Bianchi identity we have $\tR(X_k,X_j,X_i,\xi) = \tR(X_k,X_i,X_j,\xi)$, in contradiction with \eqref{eq:dG2nR}. It follows that $\tR(X_k,X,Y,\xi) = 0$, for all $X,Y \in T_xM$, and hence, for all $X,Y \in T_x\tM$, which contradicts Lemma~\ref{l:rr}\eqref{it:rrnot2hyper}.

  For assertion~\eqref{it:lak0}, by \eqref{eq:codazzi} we have $\tR(X_k,X_i,X_i,\xi)  = X_k(\la_i) - \la_i \G_{ii}^{\hphantom{i}k}$, and by \eqref{eq:dG2nR}, $\G_{ii}^{\hphantom{i}k} = \alpha_i^{-1} \la_i \tR(X_k,X_i,X_i,\xi)$. Then using \eqref{eq:Gaussal} we obtain $X_k(\la_i)= \alpha_i^{-1} \la_i H \tR(X_k,X_i,X_i,\xi)$. Furthermore, from \eqref{eq:dG1nR} and assertion \eqref{it:nolakH} we get $X_k(\al_i)= 0$. Differentiating \eqref{eq:Gaussal} we find $0=X_k(\al_i)=(H-2\la_i) X_k(\la_i) + \la_i X_k(H)$, and the equation for $X_k(H)$ follows.
\end{proof}

\begin{proof}[Proof of Proposition~\ref{p:singH=0}]
Suppose $\tM$ is not one of the spaces $\SL(3)/\SO(3)$ or $\SU(3)/\SO(3)$.

If $H=0$ on $M$, then the other conditions in~\eqref{it:orH=0} follow easily: by Lemma~\ref{l:lak0H}\eqref{it:lak0} and Lemma~\ref{l:XkH}\eqref{it:Xkali0},\eqref{it:XkH0}\eqref{it:XkH0gen} we obtain that $\al_i$ are constant on every connected component of $M'$, and hence on the whole of $M$, as $M' \subset M$ is open and dense. Then $\al_i \le 0$ from \eqref{eq:Gaussal}, and so $\ve = -1$.

We can therefore suppose that $H \ne 0$ on a nonempty, open, connected subset $\mU \subset M'$, and will seek a contradiction.

As $\mU \subset M'$ and is connected, we obtain that for any $i=1, \dots, n$, if $\al_i=0$ at some point $x \in \mU$, then $\al_i=0$ on the whole of $\mU$. Then for all $i=1, \dots, n$, we have $X_k(\al_i)=0$, for all $k$ with $\al_k=0$ from Lemma~\ref{l:lak0H}\eqref{it:lak0}, and for all $k$ with $\al_k \ne 0$ from Lemma~\ref{l:XkH}\eqref{it:Xkali0}. It follows that $\al_i$ are constant on $\mU$.

Denote $\sigma=\{\al_i, \, | \, i=1,\dots,n\} \setminus \{0\}$. From \eqref{eq:Gaussal}, for every $\al \in \sigma$, the eigenspace $L_\al$ is the orthogonal direct sum of eigenspaces $V_\pm^\al$ of $S$ with corresponding eigenvalues $\la_\pm=\frac12(H \pm \sqrt{H^2-4\al})$; denote $p_\pm^\al = \dim V_\pm^\al$. We have $H=\sum_{i=1}^{n} \la_i = \sum_{\al \in \sigma} \frac12((p_+^\al + p_-^\al)H + (p_+^\al - p_-^\al)\sqrt{H^2-4\al})$. Note that $\sum_{\al \in \sigma} (p_+^\al + p_-^\al) = n - \dim \Ker (\tR_\xi)_{\gp}=\rk \tR_\xi$, and so we obtain $(\rk \tR_\xi - 2)H + \sum_{\al \in \sigma} (p_+^\al - p_-^\al)\sqrt{H^2-4\al}=0$. In this equation, all $\al \in \sigma$ are constant on $\mU$, and all $p_\pm^\al$ and $\rk \tR_\xi$ are constant as $\mU \subset M'$. It follows that $H$ must also be constant. Indeed, assume the set of values of $H$ satisfying this equation has a nonempty interior. As the left-hand side is analytic in $H$ in the interior of the domain where it is defined (and as we can assume $H>0$ by changing the sign of $\xi$ if necessary), we obtain that the equation is satisfied for all large positive values of $H$. Dividing by $H$ and denoting $t=4H^{-2}$ we get that the equation $(\rk \tR_\xi - 2) + \sum_{\al \in \sigma} (p_+^\al - p_-^\al)\sqrt{1-\al t}=0$ is satisfied for all small positive values of $t$. Expanding into the Taylor series we obtain $\sum_{\al \in \sigma} (p_+^\al - p_-^\al)\al^m=0$, for all $m \in \mathbb{N}$, which gives $p_+^\al - p_-^\al=0$, for all $\al \in \sigma$, and so $\rk \tR_\xi = 2$, in contradiction with Lemma~\ref{l:XkH}\eqref{it:rkRxi2}.

% to roots this argument L = sum gp? and the next one (same sentence), for \tR?
It follows that $H$ is a nonzero constant on $\mU$. Then from \eqref{eq:Gaussal}, all $\la_i$ are also constants on $\mU$, and so from the first equation of Lemma~\ref{l:lak0H}\eqref{it:lak0} we obtain $\tR(X_k,X_i,X_i,\xi)=0$, for all $k$ with $\al_k=0$ and all $\al_i \ne 0$. This implies that for any $\al \in \sigma$ and any $X_k \in L_0$, both subspaces $V_\pm^\al$ are isotropic subspaces of the quadratic form $\phi_{\al,k}$ on $L_\al$ defined by $\phi_{\al,k}(X)=\tR(X_k,X,X,\xi)$. Let $\ga$ be a Cartan subalgebra containing $\xi$ and let $\beta_A, \; A=1, \dots, p$, be the root vectors such that $\ve \<\xi, \beta_A\>^2=\al$, with $\gp_A$ the corresponding root subspaces. Then $\oplus_{A=1}^p \gp_A=L_\al$, and choosing $X_k \in \ga$ we get $\phi_{\al,k}(X)=\tR(X_k, X, X, \xi)=\sum_{A=1}^{p} \ve \<\xi, \beta_A\> \<X_k, \beta_A\> \|X_A\|^2$, where $X \in L_\al$ and $X_A$ is its orthogonal projection to $\gp_A$. Suppose $\xi$ is not a multiple of any of the root vectors $\beta_A, \; A=1, \dots, p$. Then we can choose $X_k \in \ga'$ in such a way that $\<X_k, \beta_A\> \ne 0$, for all $A=1, \dots, p$. As $\<\xi, \beta_A\> = \pm \sqrt{\ve \al} \ne 0$, the quadratic form  $\phi_{\al,k}$ is nonsingular on $L_\al$. Since the subspaces $V_{\pm}$ are orthogonal and isotropic, it follows that $\dim V_{\pm}=\frac12 \dim L_\al$. Then by \eqref{eq:Gaussal}, $\sum_{i:\al_i=\al} \la_i=\frac12 H \dim L_\al$. Thus if $\xi$ is not a multiple of any root vector from $\Delta$, we get $H=\sum_{i=1}^{n} \la_i = \sum_{\al \in \sigma} \frac12 H  \dim L_\al$, and so $\rk \tR_\xi = 2$, in contradiction with Lemma~\ref{l:XkH}\eqref{it:rkRxi2}.

% in the first case, even rr2 is better
We can therefore assume that $\xi$ is a multiple of some root vector $\beta \in \Delta$. First assume that there exists $\gamma \in \Delta$ such that the angle between $\beta$ and $\gamma$ is $\pi/6$ (this is possible if and only if $\Delta$ is of type $\mathrm{G}_2$). Then $\beta+\gamma \notin \Delta, \; \beta-\gamma \in \Delta$, and so by Lemma~\ref{l:rr}\eqref{it:rr1}, there exist $X \in \gp_\beta, \, Y  \in \gp_\gamma, \, Z \in \gp_{\beta-\gamma}$ such that $\tR(\xi,X,Y,Z) \ne 0$. But $X \in L_{\al_i}, \, Y \in  L_{\al_j}, \, Z \in  L_{\al_k}$, where $\al_i=\ve \<\beta, \xi\>^2, \, \al_j = \ve \<\gamma, \xi\>^2, \, \al_k=\ve \<\beta-\gamma, \xi\>^2$ and $\al_i \al_j \ne 0$, in contradiction with Lemma~\ref{l:lak0H}\eqref{it:Hne0}. Next assume that for no $\gamma \in \Delta$, the angle between $\beta$ and $\gamma$ is $\pi/6$, but there exists $\gamma \in \Delta$ such that the angle between $\beta$ and $\gamma$ is $\pi/3$ (this is always the case for $\Delta$ of type $\mathrm{A}_r \, (r \ge 2), \, \mathrm{D}_r \, (r \ge 3), \, \mathrm{E}_r  \, (r =6,7,8)$ and $\mathrm{F}_4$). Then $\beta+\gamma \notin \Delta, \; \beta-\gamma \in \Delta$, and repeating the argument (and noting that $\|\beta\|=\|\gamma\|$ and so $\<\xi, \beta-\gamma\> \ne 0$) we again obtain a contradiction with Lemma~\ref{l:lak0H}\eqref{it:Hne0}. Next assume that for no $\gamma \in \Delta$, the angle between $\beta$ and $\gamma$ is $\pi/6$ or $\pi/3$, but there exists $\gamma \in \Delta$ such that the angle between $\beta$ and $\gamma$ is $\pi/4$ and $\beta$ is longer than $\gamma$. This, together with the above cases, covers the root systems of type $\mathrm{C}_r \, (r \ge 3)$ and $\mathrm{BC}_r \, (r \ge 2)$ (note that out of two proportional roots in $\mathrm{BC}_r$, we can choose the longer one to be $\beta$), and the long roots of $\mathrm{B}_r \, (r \ge 2)$. Then the same argument works again (as $\beta+\gamma \notin \Delta, \; \beta-\gamma \in \Delta$, and $\|\beta\|=\sqrt{2}\|\gamma\|$ and so $\<\xi, \beta-\gamma\> \ne 0$) leading to a contradiction.

The only remaining case is when $\Delta$ is of type $\mathrm{B}_r \, (r \ge 2)$ and $\beta$ is a short root. We have $\Delta=\{\pm e_i, \pm e_i \pm e_j \, | \, i,j=1, \dots, r,\, i < j\}$, where $\{e_i\}$ is (up to scaling) an orthonormal basis for $\ga$, and we can take $\xi=\pm e_1$. Then the set $\sigma$ of nonzero $\al_s$ consist of a single element $\al=\ve$ (with multiplicity at least $2r-1 \ge 3$). The nonzero eigenvalues of $S$ are therefore $\la_{\pm}=\frac12(H \pm \sqrt{H^2-4\ve})$, with corresponding eigenspaces $V_{\pm}$ of dimensions $p_{\pm}$, respectively. We note that $p_++p_-\ge 3$ and $p_+,p_- > 0$ as $\tR(X_k, X_i, X_j, \xi) \ne 0$ for some $X_k \in L_0$ and some $X_i \in V_+, \, X_j \in V_-$ (this follows from Lemma~\ref{l:rr}\eqref{it:rrnot2hyper} and the fact that $\phi_{\al,k}(V_+)=\phi_{\al,k}(V_-)=0$). Then $H=p_+\la_++p_-\la_-$ which gives $(p_+-1)(p_--1)H^2+\ve(p_+-p_-)^2=0$, and so we must have $\ve=-1$ (then $\tM$ is of noncompact type) and $p_+ \ne p_-$. %We also note that $H, \la_+$ and $\la_-$ are constant on a neighbourhood of $x$, and $\al=-1$.

There are only two classes of noncompact symmetric spaces with the root system of type $\mathrm{B}_r, \, (r \ge 2)$: the noncompact dual $\SO(2r+1,\bc)/\SO(2r+1)$ to the group $\SO(2r+1)$, and the noncompact Grassmannians $\SO^0(r+m,r)/\SO(r)\SO(r+m)$, where $m \ge 1$ and $\SO^0(r+m,r)$ is the identity component of the pseudo orthogonal group $\SO(r+m,r)$. In the first case, $\gp = \ri \cdot \so(2r+1)$, the space of purely imaginary skew-symmetric $(2r+1) \times (2r+1)$ matrices, and we can choose a basis in such a way that $\xi$ is (a real multiple of) the matrix whose $(1,2)$-th entry is $\ri$, $(2,1)$-th entry is $-\ri$, and all the other entries are zeros. Then $L_{-1}=V_+\oplus V_-$ is the orthogonal complement to $\xi$ in the space of matrices from $\gp$ whose $(i,j)$-th entries are zeros for $i,j > 2$, and the conditions $\tR(\xi, V_+,V_+,L_0)=\tR(\xi, V_-,V_-,L_0)=0$ imply that, up to specifying the basis, we can take $V_+$ (respectively, $V_-$) to be the space of matrices from $L_{-1}$ whose first (respectively, second) column is zero. But then $p_+=p_-\, (=2r-1)$, a contradiction.

% need all constant!
It remains to consider the case when $\tM=\SO^0(r+m,r)/\SO(r+m)\SO(r), \; r \ge 2, m \ge 1$. Then $\gp$ is the space of $(2r+m)\times(2r+m)$ real, symmetric matrices of the form $X=\left(\begin{smallmatrix} 0 & T^t \\ T & 0 \end{smallmatrix}\right)$, where $T$ is an $r\times (r+m)$ real matrix; we denote $T=\pi(X)$ and define $\<X,Y\>=\Tr (\pi(X) (\pi(Y))^t)$ for $X, Y \in \gp$ (with this scaling, $\al=-1$). We can choose a basis in such a way that $\pi(\xi)=\left(\begin{smallmatrix} 1 & 0_{r+m-1}^t \\ 0_{r-1} & 0 \end{smallmatrix}\right)$, where $0_k$ is the column of $k$ zeros. Then $L_0=\{X \in \gp \, | \, \pi(X)=\left(\begin{smallmatrix} 0 & 0_{r+m-1}^t \\ 0_{r-1} & * \end{smallmatrix}\right)\}$ and $L_{-1}=\{X \in \gp \, | \, \pi(X)=\left(\begin{smallmatrix} 0 & u^t \\ v & 0 \end{smallmatrix}\right),\, u \in \br^{r+m-1}, \, v \in \br^{r-1}\}$. We have $L_{-1}=V_+\oplus V_-$ and $\tR(\xi, V_+,V_+,L_0)=\tR(\xi, V_-,V_-,L_0)=0$. Computing the brackets we find $V_+=\{X \in \gp \, | \, \pi(X)=\left(\begin{smallmatrix} 0 & 0_{r+m-1}^t \\ v & 0 \end{smallmatrix}\right),\, v \in \br^{r-1}\}$ and $V_-=\{X \in \gp \, | \, \pi(X)=\left(\begin{smallmatrix} 0 & u^t \\ 0_{r-1} & 0 \end{smallmatrix}\right),\, u \in \br^{r+m-1}\}$ (or vice versa). From \eqref{eq:codazzi} we now find $\nabla_{V_+}V_+ \subset V_+,\; \nabla_{V_-}V_- \subset V_-,\; \nabla_{L_0}L_0 \subset L_0$. Furthermore, from \eqref{eq:codazzi} and \eqref{eq:dG2nR} we obtain, for $X_k \in L_0, \, X_i \in V_+, \; X_j \in V_-$, that $\G_{ij}^{\hphantom{i}k}= -\la_+ \tR(X_k,X_i,X_j,\xi),\; \G_{ji}^{\hphantom{j}k}= -\la_- \tR(X_k,X_i,X_j,\xi)$ and $\G_{ki}^{\hphantom{k}j}= 2(\la_+-\la_-)^{-1}\tR(X_k,X_i,X_j,\xi)$. Note that $V_+$ and $V_-$ commute, that is, $\tR(X_i,X_j)=0$ for all $X_i \in V_+,\, X_j \in V_-$. Differentiating along $X_k \in L_0$ we obtain $\sum_{s: X_s \in V_+} \tR(X_k,X_j,X_s,\xi) \tR(X_i,X_s)=\sum_{l: X_l \in V_-} \tR(X_k,X_i,X_l,\xi) \tR(X_l,X_j)$. Acting by both sides on $X_j$ we get $\sum_{l: X_l \in V_-} \tR(X_k,X_i,X_l,\xi) \tR(X_l,X_j)X_j=0$, as $\tR(X_i,X_s)X_j=0$ by the Bianchi identity. But $\tR(X_l,X_j)X_j=(\delta_{jl}-1)X_l$ (in fact, $V_-$ is a Lie triple system tangent to the totally geodesic hyperbolic space), and so we obtain $\tR(X_k,X_i,X_l,\xi) = 0$, for all $X_l \in V_-$ with $l \ne j$. But $\dim V_-=r+m-1 \ge 2$, and so we get $\tR(X_k,X_i,X_l,\xi) = 0$, for all $X_i \in V_+, \, X_l \in V_-, \, X_k \in L_0$. This is a contradiction by Lemma~\ref{l:rr}\eqref{it:rrnot2hyper} (or by just computing the brackets).
\end{proof}

\section{Codimension one Einstein solvmanifolds}
\label{s:solv}

In this section, we prove the following proposition which shows that in case \eqref{it:orH=0} of Proposition~\ref{p:singH=0}, the hypersurface $M$ must be an open domain of an Einstein solvmanifold of codimension $1$, as in Theorem~\ref{th:main}\eqref{thit:solv}. %(in Example~\ref{ex:solv}).

\begin{proposition} \label{p:solv}
  Let $\tM$ be an irreducible symmetric space of rank at least $2$ and of noncompact type, and let $M$ be a connected Einstein $C^4$-hypersurface in $\tM$ such that $C=0, \; H=0$, and all the eigenvalues of $\tR_\xi$ are constant. Then $M$ is \emph{(}an open domain of\emph{)} a codimension~$1$ Einstein solvmanifold, as in Example~\ref{ex:solv}.
\end{proposition}

For the proof, we will need only a small piece of the hypersurface $M$, and so we can assume that $\al_i$ and $\la_i$ are constant (and then $\la_i^2=-\al_i$ by \eqref{eq:Gaussal}) for all $i=1, \dots, n$. The proof goes as follows. Starting with $\xi$ and $S$ at a single point $x \in M$, we explicitly construct an Iwasawa decomposition of $\rG$ (see Section~\ref{ss:iwa}) such that $\xi \in \ga$ and such that the codimension one Einstein subgroup $\Sh \subset \rA \rN = \tM$ with the subalgebra $\xi^\perp$ has the same shape operator $S$ at $x$. This is the most involved step, as $\xi$ may be singular and hence belong to many Cartan subalgebras; we will see that when $\xi$ is ``very singular", the choice of such Cartan subalgebra for our purposes may not even be unique. The required Iwasawa decomposition is constructed piece-by-piece in Lemmas~\ref{l:singgen}, \ref{l:LF}, \ref{l:mN} and \ref{l:Iwasawa}. Then the proof is completed by Lemma~\ref{l:analyt} which shows that by analyticity (which follows from minimality), two Einstein hypersurfaces $\Sh$ and $M$ having the second order tangency at $x$ must (locally) coincide.

Let $x \in M$; as usual, we identify $\gp$ with $T_x\tM$.

For $i, j$ such that $\la_i, \la_j \ne 0$, denote
\begin{equation}\label{eq:defPsi}
\Psi(X_i,X_j)=-\la_i^{-1} \tR(\xi, X_i)X_j+\la_j^{-1} \tR(\xi, X_j)X_i=\la_i^{-1} [[\xi, X_i],X_j] - \la_j^{-1} [[\xi, X_j],X_i],
\end{equation}
and extend $\Psi$ by bilinearity to $\gp \cap (L_0 \oplus \br \, \xi)^\perp$. Note that $\Psi$ is skew-symmetric and $\Psi(X,Y) \perp \xi$.

\begin{lemma} \label{l:singgen}
In the assumptions of Proposition~\ref{p:solv}, we have the following.
  \begin{enumerate}[label=\emph{(\alph*)},ref=\alph*]
    \item \label{it:Psi+}
    $\tR(X_k,X_i,X_j,\xi)$ can only be nonzero when one of $\la_i, \la_j, \la_k$ is the sum of the other two. Moreover, assuming $\la_i, \la_j \ne 0$ we have $\Psi(V_{\la_i}, V_{\la_j}) \subset V_{\la_i + \la_j}$.

    \item \label{it:Gammasing}
    Suppose $\la_i \ne \la_j$. Then we have $\G_{ki}^{\hphantom{k}j}= \la_k^{-1} \tR(\xi, X_k,X_j,X_i)$ when $\la_k \ne 0$, and $\G_{ki}^{\hphantom{k}j}= \la_i^{-1} \tR(\xi,X_j,X_i,X_k)$ when $\la_k =0$ and $\la_i \ne 0$.

    \item \label{it:Gauss+-}
    Suppose $\la_i =\la_r =-\la_j=-\la_k\ne 0$. Then
    \begin{multline*}
        \la_i^2\tR(X_i,X_j,X_r,X_k)= \<\tR(\xi, X_i)X_k,\tR(\xi, X_j)X_r\> - \<\tR(\xi, X_i)X_r,\tR(\xi, X_j)X_k\> \\
        - \la_i \<\tR(\xi, X_r)X_k, \Psi(X_i,X_j)\>.
    \end{multline*}
  \end{enumerate}
\end{lemma}
\begin{proof}
  Assertion~\eqref{it:Psi+} follows from Lemma~\ref{l:Rijk}\eqref{it:Rijknon0eta}. As $\al_s=-\la_s^2=-\eta_s^2$ for all $i=1, \dots, n$, we have $\la_s=\ve_s\eta_s,\; \ve_s = \pm 1$, for $s \in \{i,j,k\}$ (note that the choice of signs of $\eta_s$ in Lemma~\ref{l:Rijk}\eqref{it:Rijknon0eta} depend on a particular triple $(i,j,k)$). We have $\ve_i \la_i + \ve_j \la_j + \ve_k \la_k = 0$. If $\eta_k = 0$, then $\la_k=0$ and $\la_i = \pm \la_j$. If $\eta_i \eta_j \eta_k \ne 0$, equation \eqref{eq:etala} implies that not all three of $(\ve_i, \ve_j, \ve_k)$ are equal. In both cases, the first statement follows. For the second statement, suppose that $\<\Psi(X_i,X_j),X_k\> \ne 0$ for some $i,j,k$ with $\la_i\la_j \ne 0$. We have $\<\Psi(X_i,X_j),X_k\>=\la_j^{-1} \tR(\xi, X_j,X_i,X_k)-\la_i^{-1} \tR(\xi, X_i,X_j,X_k)$. Then the triple $r_{jki}=(\tR(\xi,X_j,X_i,X_k), \tR(\xi,X_k,X_j,X_i), \tR(\xi,X_i,X_k,X_j))$ is nonzero and hence by Lemma~\ref{l:Rijk}\eqref{it:Rijknon0eta} is proportional to the triple $(\eta_j,\eta_k,\eta_i)= (\ve_j\la_j,\ve_k\la_k,\ve_i\la_i)$, which gives $\ve_i=\ve_j$. If $\la_k=0$ we get $\ve_i (\la_i + \la_j) = 0$, as required. If $\la_k \ne 0$, we must have $\ve_k=-\ve_i$ (as not all three of $(\ve_i, \ve_j, \ve_k)$ are equal), and so, again, $\la_i+\la_j=\la_k$.

  For assertion~\eqref{it:Gammasing}, in the notation of the previous paragraph, we first assume that $r_{jki}=0$. If $\al_i \ne \al_j$ we get $\G_{ki}^{\hphantom{k}j}=0$ by \eqref{eq:dG2nR}. If $\al_i=\al_j$ we have $\la_i=-\la_j \ne 0$. If $\al_k \ne \al_i$, then $\G_{ik}^{\hphantom{k}j}=0$ by \eqref{eq:dG2nR} and so $\G_{ki}^{\hphantom{k}j}=0$ by \eqref{eq:codazzi}. If $\al_k = \al_i$, then $\la_k \in \{\la_i, \la_j\}$ and again,  $\G_{ki}^{\hphantom{k}j}=0$ by \eqref{eq:codazzi}. If we now assume that $r_{jki} \ne 0$, then by Lemma~\ref{l:Rijk}\eqref{it:Rijknon0eta} we have $r_{jki}=\mu (\ve_j\la_j,\ve_k\la_k,\ve_i\la_i)$ for some $\mu \ne 0$, where $\ve_i \la_i + \ve_j \la_j + \ve_k \la_k = 0$. Then from  \eqref{eq:dG2nR} we get $\G_{ki}^{\hphantom{k}j}(\la_i^2-\la_j^2)=\la_k(\ve_j \la_j - \ve_i\la_i)\mu = \ve_k (\la_i^2-\la_j^2)\mu$. If $\la_i^2 \ne \la_j^2$ (note that in this case $\la_k \ne 0$ by assertion~\eqref{it:Psi+}) we get $\G_{ki}^{\hphantom{k}j} = \ve_k \mu$, as required. If $\la_i^2 = \la_j^2$, then $\la_j=-\la_i$. From assertion~\eqref{it:Psi+} we have $\la_k \ne \pm\la_i$, and so by the above argument we get $\G_{ij}^{\hphantom{i}k} = \ve_i \mu$, and the claim follows from \eqref{eq:codazzi} (we separately consider the cases when $\la_k=\la_i+\la_j, \, \la_i=\la_j+\la_k$ and $\la_j=\la_k+\la_i$).

  To prove assertion~\eqref{it:Gauss+-} we substitute the expressions for $\G_{ca}^{\hphantom{c}b}$ obtained in \eqref{it:Gammasing} into the Gauss equation $\tR(X_i,X_j,X_r,X_k)+\delta_{ir}\delta_{jk}\la_i^2 =X_i(\G_{jr}^{\hphantom{j}k})-X_j(\G_{ir}^{\hphantom{i}k})+\sum_{s=1}^{n} (\G_{jr}^{\hphantom{j}s}\G_{is}^{\hphantom{i}k}-\G_{ir}^{\hphantom{i}s}\G_{js}^{\hphantom{j}k} - (\G_{ij}^{\hphantom{i}s}-\G_{ji}^{\hphantom{j}s}) \G_{sr}^{\hphantom{s}k})$. By assertions~\eqref{it:Psi+} and \eqref{it:Gammasing} we have $\G_{jr}^{\hphantom{j}k}= \la_j^{-1} \tR(\xi, X_j,X_k,X_r)=0$ (and this is satisfied locally), so $X_i(\G_{jr}^{\hphantom{j}k})=0$, and similarly $X_j(\G_{ir}^{\hphantom{i}k})=0$. Next, from assertions~\eqref{it:Psi+} and \eqref{it:Gammasing} we get $\G_{jr}^{\hphantom{j}s} = 0$ when $\la_s=\la_k$, and $\G_{is}^{\hphantom{i}k}=0$ when $\la_s=\la_r$, and so $\sum_{s=1}^{n} \G_{jr}^{\hphantom{j}s}\G_{is}^{\hphantom{i}k}=\la_i^{-2}\<\tR(\xi, X_i)X_k,\tR(\xi, X_j)X_r\>$ by assertions~\eqref{it:Gammasing}. A similar argument shows that $\sum_{s=1}^{n} \G_{ir}^{\hphantom{i}s}\G_{js}^{\hphantom{j}k}=\la_i^{-2}\<\tR(\xi, X_i)X_r,\tR(\xi, X_j)X_k\>-\delta_{ir}\delta_{jk}\la_i^2$. Applying assertions~\eqref{it:Psi+} and \eqref{it:Gammasing} again and using \eqref{eq:defPsi} we find $\sum_{s=1}^{n} (\G_{ij}^{\hphantom{i}s}-\G_{ji}^{\hphantom{j}s}) \G_{sr}^{\hphantom{s}k}= \la_i^{-1} \sum_{s: \la_s = 0}  \tR(\xi, X_k,X_r,X_s) \<\Psi(X_i,X_j),X_s\> + \sum_{s: \la_s \ne 0} \la_s^{-1} \tR(\xi, X_s,X_k,X_r) \<\Psi(X_i,X_j),X_s\>$, as $\tR(\xi, X_k,X_r,X_s) = \tR(\xi, X_r,X_k,X_s)$ for $\la_s=0$ which follows from $\tR(\xi, X_s)=0$ (Lemma~\ref{l:Rijk}\eqref{it:ali0}). But by assertions~\eqref{it:Psi+}, $\Psi(X_i,X_j) \in L_0$, and so the second sum on the right-hand side is zero, and the first one can be formally extended to all $s=1, \dots, n$. Hence
  $\sum_{s=1}^{n} (\G_{ij}^{\hphantom{i}s}-\G_{ji}^{\hphantom{j}s}) \G_{sr}^{\hphantom{s}k}= \la_i^{-1} \<\tR(\xi, X_k)X_r,\Psi(X_i,X_j)\>$, and the claim follows.
\end{proof}

We will now define a linear map $F:\gp \to \g$ whose image will be a subalgebra of $\g$. We define $F$ in three stages. First, we set
\begin{equation} \label{eq:Fnot0}
  F(X_i)= X_i + \la_i^{-1}[\xi, X_i], \qquad \text{when } \la_i \ne 0,
\end{equation}
and then extend by linearity to the subspace $W=(\br \xi \oplus L_0)^\perp \subset \gp$. Note that for $\la_i, \la_j \ne 0$ we have
\begin{equation} \label{eq:Fbr}
[F(X_i),F(X_j)]=\Psi(X_i,X_j) + ([X_i,X_j]+\la_i^{-1}\la_j^{-1}[[\xi,X_i],[\xi,X_j]])
\end{equation}
by \eqref{eq:defPsi}, where the first and the second term on the right-hand side are the $\gp$- and the $\gk$-components of $[F(X_i),F(X_j)]$, respectively. In particular, $\pi_{\gp}[F(X),F(Y)]=\Psi(X,Y)$, for $X, Y \in W$.

\begin{lemma} \label{l:LF}
  We have the following.
  \begin{enumerate}[label=\emph{(\alph*)},ref=\alph*]
    \item \label{it:xiF}
    Let $\la_i \ne 0$. Then $[\xi, F(X)]= \la_i F(X)$ for $X \in V_{\la_i}$.

    \item \label{it:lai+laj}
    Suppose $\la_i, \la_j \ne 0$ and $\la_i+\la_j \ne 0$. Then $[F(V_{\la_i}),F(V_{\la_j})] \subset F(V_{\la_i+\la_j})$.

    \item \label{it:mLFi}
    Denote $\mL=\Span([F(X_i),F(X_j)] \, | \, i,j: \la_j=-\la_i \ne 0)$. Then $[\xi,\mL] = 0$. Moreover, $[\mL, F(V_{\la_r})] \subset F(V_{\la_r})$ for all $\la_r \ne 0$, and so the subspaces $\mL, \mL \oplus F(W) \subset \g$ are subalgebras.

    \item \label{it:mLinj}
    The projection $\pi_\gp: \mL \to \gp$ is injective.
  \end{enumerate}
\end{lemma}
\begin{proof}
  Assertion~\eqref{it:xiF} directly follows from \eqref{eq:Fnot0} and the fact that $[\xi,[\xi,X]]=\la_i^2 X$ for $X \in V_{\la_i}$.

  Assertion~\eqref{it:lai+laj} follows from Lemma~\ref{l:singgen}\eqref{it:Psi+}. Indeed, if $\la_i, \la_j, \la_i+\la_j \ne 0$, we have $\Psi(X_i,X_j) \in V_{\la_i+\la_j}$, and then $[\xi, \Psi(X_i,X_j)] = (\la_i+\la_j) [X_i,X_j] + \la_i^{-1}\la_j^{-1}(\la_i+\la_j)[[\xi, X_i],[\xi,X_j]]$ by \eqref{eq:defPsi}, and so by \eqref{eq:Fbr} we have $[F(X_i),F(X_j)]=F(\Psi(X_i,X_j)) \in F(V_{\la_i+\la_j})$, as required.

  For assertion~\eqref{it:mLFi}, let $\la_j=-\la_i \ne 0$. Then $[\xi, [F(X_i),F(X_j)]]=0$ from assertion~\eqref{it:xiF} and the Jacobi identity, which gives that $[\xi,\mL] = 0$. Next, if $\la_r \ne 0, \pm \la_i$, then $[[F(X_i),F(X_j)],F(X_r)] \in F(V_{\la_r})$ by assertion~\eqref{it:lai+laj} and the Jacobi identity.

  Now suppose $\la_r =\la_i$. Denote $P=\pi_\gp [[F(X_i),F(X_j)],F(X_r)]$. Then from assertion~\eqref{it:xiF} we have $[\xi, [[F(X_i),F(X_j)],F(X_r)]] = \la_i [[F(X_i),F(X_j)],F(X_r)]$, and so $[[F(X_i),F(X_j)],F(X_r)]=P+\la_i^{-1}[\xi,P]$ and $[\xi,[\xi,P]] = \la_i^2 P$, that is, $P \in L_{-\la_i^2} = V_{\la_i} \oplus V_{-\la_i}$. It is therefore sufficient to prove that $\<P,X_k\>=0$, for all $X_k \in V_{-\la_i}$. From \eqref{eq:Fbr} we have $P = \la_i^{-2}[[[\xi,X_i],X_j]+[[\xi,X_j],X_i],[\xi,X_r]]+[[X_i,X_j],X_r]-\la_i^{-2}[[\xi,X_i],[\xi,X_j]],X_r]$. So $\<P,X_k\>=-\tR(X_i,X_j,X_r,X_k)-\la_i^{-1} \<\tR(\xi, X_r)X_k,\Psi(X_i, X_j)\>+\la_i^{-2} (\<\tR(\xi, X_i)X_k,\tR(\xi, X_j)X_r\>-\<\tR(\xi, X_j)X_k,\tR(\xi, X_i)X_r\>)=0$, by Lemma~\ref{l:singgen}\eqref{it:Gauss+-}.

  Hence $[\mL, F(V_{\la_r})] \subset F(V_{\la_r})$ for all $\la_r \ne 0$. The fact that both subspaces $\mL$ and $\mL \oplus W$ are subalgebras follows from this and assertion~\eqref{it:lai+laj} (the sum $\mL \oplus W$ is indeed direct, as $[\xi, \mL]=0$ and $[\xi,F(W)]=F(W)$). This proves assertion~\eqref{it:mLFi}.

  To prove assertion~\eqref{it:mLinj}, suppose that $U \in \mL \cap \gk$. Let $\la_j=-\la_i \ne 0$. By assertion~\eqref{it:mLFi} we have $[U,X_i] \in V_{\la_i}$, and so $\<[U,X_i],X_j\>=0$. Hence $B(U,[X_i,X_j])=0$, where $B$ is the Killing form of $\g$. Furthermore, we have $B(U,[[\xi,X_i],[\xi,X_j]])= -B([[\xi,X_i],U],[\xi,X_j])= B([\xi,[U,X_i]],[\xi,X_j])= -B([\xi,[\xi,[U,X_i]]],X_j)= -\la_i^2 B([U,X_i],X_j)=0$ (where we used the fact that $[\xi, U]=0$ and that $[U,X_i] \in V_{\la_i}$ from assertion~\eqref{it:mLFi}).

  But $U \in \Span(\pi_\gk [F(X_i),F(X_j)] \, | \, i,j: \la_j=-\la_i \ne 0)$, and so $U=\sum_{i,j: \la_j=-\la_i \ne 0} u_{ij}([X_i,X_j]- \la_i^{-2} [[\xi,X_i],[\xi,X_j]])$, for some $u_{ij} \in \br$ by \eqref{eq:Fbr}, and so $B(U,U)=0$ from the above argument. As the restriction of $B$ to $\gk$ is negative definite we obtain $U=0$.
\end{proof}

The subspace $\pi_\gp \mL$ lies in $L_0$ (by Lemma~\ref{l:LF}\eqref{it:mLFi} and as $\Psi(X_i,X_j) \perp \xi$ by \eqref{eq:defPsi}). We now define
\begin{equation} \label{eq:FinL}
  F(\pi_\gp L)= L, \qquad \text{for } L \in \mL.
\end{equation}
This is well defined by Lemma~\ref{l:LF}\eqref{it:mLinj}.

To extend $F$ to the whole of $\gp$ it remains to define it on the subspace $\mN = (L_0 \oplus \br \xi) \cap (\pi_\gp \mL)^\perp$. Note that $[\xi, \mN]=0$.

\begin{lemma} \label{l:mN}
  The following holds. % remove maybe? in many other lemmas inside the proofs of bigger propositions.
  \begin{enumerate}[label=\emph{(\alph*)},ref=\alph*]
    \item \label{it:commutes}
    Let $Z \in L_0 \oplus \br \xi$. Then the operator $A_Z$ on $T_xM$ given by $A_ZX = \tR(X,\xi)Z$ commutes with $S$ if and only if $Z \in \mN$.

    \item \label{it:Z}
    Let $Z \in \mN$. Then $[Z, F(V_{\la_i})] \subset F(V_{\la_i})$ for $\la_i \ne 0$, and $[Z, \mL] \subset \mL$.

    \item \label{it:Nlts}
    The subspace $\mN \subset \gp$ is a Lie triple system.
  \end{enumerate}
\end{lemma}
\begin{proof}
  Let $Z \in L_0 \oplus \br \xi$. As $[\xi,Z]=0$, the operator $A_Z$ is symmetric and $T_xM$ is its invariant subspace. Moreover, $A_Z$ commutes with $S$ if and only if $\tR(X_i,\xi,Z,X_j)=0$, for all $i,j$ such that $\la_i \ne \la_j$. By Lemma~\ref{l:singgen}\eqref{it:Psi+}, this is equivalent to saying that $Z \perp \Psi(X_i,X_j)$, for all $i,j$ such that $\la_j=-\la_i \ne 0$, that is, to the fact that $Z \perp \pi_\gp \mL$. This proves assertion~\eqref{it:commutes}.

  For assertion~\eqref{it:Z}, take $Z \in \mN$ and let $\la_i \ne 0$. Then by \eqref{eq:Fnot0} $[Z,F(X_i)]=[Z,X_i]+\la_i^{-1}[Z,[\xi,X_i]]= -\la_i^{-2} [\xi,A_ZX_i]-\la_i^{-1}A_ZX_i$. By assertion~\eqref{it:commutes}, $A_ZX_i \in V_{\la_i}$, and so $[Z,F(X_i)]= -\la_i^{-1}F(A_ZX_i) \in F(V_{\la_i})$. As $\mL=\Span([F(X_i),F(X_j)] \, | \, i,j: \la_j=-\la_i \ne 0)$, the second statement of the assertion follows from the Jacobi identity.

  For assertion~\eqref{it:Nlts}, take $Z_1,Z_2,Z_3 \in \mN$ and let $Z=[[Z_1,Z_2],Z_3]$. Then $Z \in L_0$, as $L_0$ is a Lie triple system. Moreover, by assertion~\eqref{it:Z} and from the Jacobi identity, for $\la_i \ne 0$ we have $[Z,F(V_{\la_i})] \subset F(V_{\la_i})$. Projecting to $\gp$ we obtain $A_Z X_i \in V_{\la_i}$, and so $A_Z$ commutes with $S$. Then $Z \in \mN$ by assertion~\eqref{it:commutes}.
\end{proof}

% careful with gv: Iwasawa for reductive
By Lemma~\ref{l:mN}\eqref{it:Nlts}, if $\mN \ne 0$, it is tangent to a totally geodesic submanifold of $\tM$, which must then be a symmetric space of non-positive curvature. Note that it is reducible: for example, as $[\xi, \mN]=0$, its de Rham decomposition has a non-trivial flat factor. The Lie algebra $\gv=\mN \oplus [\mN,\mN]$ is reductive, and the subalgebra $[\mN,\mN]=\gv \cap \gk$ is its maximal compact subalgebra (so that the above decomposition of $\gv$ is a Cartan decomposition). Consider an arbitrary Iwasawa decomposition $\gv=[\mN,\mN] \oplus \ga' \oplus \gn'$ of $\gv$, where $\ga' \oplus \gn'$ is solvable, $\gn'$ is nilpotent and $\ga'$ abelian \cite[Ch.~VI, \S 3]{Hel}. Note that $\ga' \subset \mN, \, [\ga' \oplus \gn', \ga' \oplus \gn'] = \gn'$ by the properties of Iwasawa decomposition and that $\pi_\gp: \ga' \oplus \gn' \to \mN$ is a linear isomorphism (and its restriction to $\ga'$ is the identity map). We now define (the restriction of) the map $F$ on $\mN$ to be the inverse of this isomorphism:
\begin{equation} \label{eq:FinN}
  F(Z)= ((\pi_\gp)_{|\ga' \oplus \gn'})^{-1}Z , \qquad \text{for } Z \in \mN.
\end{equation}
Note that $\xi \in \ga'$ (as $[\xi, \gv]=0$), and so $F(\xi)=\xi$.

The following lemma is central for the proof of Proposition~\ref{p:solv}. Note that most parts of it are well-known from the general theory. The most important point (which required all the construction above) is assertion \eqref{it:shape} which states that the resulting Einstein solvmanifold $\Sh$ ``agrees" with $M$ at the point $x$ up to the second order.
\begin{lemma} \label{l:Iwasawa}
  Let $\ogs=F(\gp)$, where the linear map $F:\gp \to \g$ is defined by \eqref{eq:Fnot0}, \eqref{eq:FinL} and \eqref{eq:FinN}. We have the following.
    \begin{enumerate}[label=\emph{(\alph*)},ref=\alph*]
    \item \label{it:subal}
    The subspace $\ogs \subset \g$ is a subalgebra and $\xi \in \ogs$.

    \item \label{it:exps}
    Let $\oSh \subset \rG$ be the subgroup of $\rG$ with the Lie algebra $\ogs$. Define the inner product $\ip_{\ogs}$ on $\ogs$ by the push-forward from $\gp$ \emph{(}so that $F:\gp \to \ogs$ is a linear isometry\emph{)}, and equip $\oSh$ with the left-invariant metric defined by $\ip_{\ogs}$. Then the metric Lie group $\oSh$ acts simply transitively on $\tM$, so that $\tM$ is isometric to $\oSh$.

    \item \label{it:solvI}
    The group $\oSh$ and the algebra $\ogs$ are solvable, and $\g = \gk \oplus \ogs$ is an Iwasawa decomposition of $\g$.

    \item \label{it:shape}
    The orthogonal complement $\gs$ to $\xi$ in $\ogs$ is a solvable subalgebra. The corresponding subgroup $\Sh$ with the induced metric is an Einstein hypersurface in $\oSh$ \emph{(}in $\tM$\emph{)} whose shape operator at $x$ is the push-forward of $S$ under $F$.
  \end{enumerate}
\end{lemma}
\begin{proof}
  Assertion~\eqref{it:subal} follows from the construction of $F$: by \eqref{eq:Fnot0}, \eqref{eq:FinL} and \eqref{eq:FinN}, the subspace $\ogs$ is the direct sum $F(W) \oplus \mL \oplus (\ga' \oplus \gn')$, where $\ga' \oplus \gn'$ is a subalgebra of $\gv = \mN \oplus [\mN,\mN]$. Now $\mL \oplus F(W) \subset \g$ is a subalgebra by Lemma~\ref{l:LF}\eqref{it:mLFi}. It is $\ad_\mN$-invariant by Lemma~\ref{l:mN}\eqref{it:Z}, and hence is $\ad_\gv$-invariant and therefore is $\ad_{\ga' \oplus \gn'}$-invariant. The fact that $\xi \in \ogs$ follows as $F(\xi)=\xi$.

  For assertion~\eqref{it:exps} note that the metric group $\oSh$ (with the left-invariant metric defined by $\ip_{\ogs}$) is complete and that the left action of $\oSh$ on both $\tM$ and itself is isometric. As the metrics on $\tM$ and on $\oSh$ agree at $x$ and as both $\oSh$ and $\tM$ are complete (and have the same dimension), the left action of $\oSh$ on $\tM$ is a local isometry, which is then a global isometry as $\tM$ is simply connected (see \cite[Lemma~2.4]{AW1} for details).

  The first statement of assertion~\eqref{it:solvI} follows from \cite[Corollary~2.6]{AW1}. For the second statement, we first prove that the subgroup $\oSh$ is ``in standard position" in the sense of \cite[Definition~6.4]{AW2}. To see that, we note that the subspace $F(W) \oplus \mL \subset \ogs$ lies in the derived algebra of $\ogs$ (for $F(W)$, this follows from Lemma~\ref{l:LF}\eqref{it:xiF}, and for $\mL$, from its definition in  Lemma~\ref{l:LF}\eqref{it:mLFi}). Furthermore, $[\ga' \oplus \gn', \ga' \oplus \gn'] = \gn'$ and $\ga' \subset \mN$ (as $\gv=[\mN,\mN] \oplus \ga' \oplus \gn'$ is an Iwasawa decomposition for $\gv$, by construction). It follows that the orthogonal complement of the derived algebra of $\ogs$ in $\ogs$ is $\ga'$. As $\ga' \subset \mN \subset \gp$, we have $B(\gk, \ga')=0$, where $B$ is the Killing form of $\g$, which shows that $\oSh$ is a subgroup ``in standard position". The claim now follows from \cite[Theorem~6.9]{AW2} (see also \cite[Remark~6.5(c)]{AW2}).

  To prove assertion~\eqref{it:shape} we first note that $\ogs=\ga' \oplus \gn$ (the decomposition is orthogonal with respect to $\ip_{\ogs}$), where $\gn=F(W) \oplus \mL \oplus \gn' = [\ogs,\ogs]$ (as we have shown in the proof of \eqref{it:solvI}). Moreover, $\ga'$ is abelian and $\gn$ is the nilradical of the solvable algebra $\ogs$ (see e.g. \cite[Theorem~5.2]{AW1}). As $\xi \in \ga'$, its orthogonal complement $\gs$ in $\ogs$ is a (solvable) subalgebra of $\ogs$. For a vector $T \in \gp$, denote $\tilde{T}$ the left-invariant vector field on $\oSh$ whose value at $x$ is $F(T)$. Then $\tilde{\xi}$ is a unit normal field along the hypersurface $\Sh \subset \oSh$, and its tangent space at every point is spanned by vectors $\tilde{X}$, where $X \in T_xM$. From the Koszul formula, the push-forward $\overline{S}$ of the shape operator of $\Sh \subset \oSh$ at $x$ is given by $\<\overline{S}X,Y\>=\<\on_{\tilde{X}} \tilde{Y}, \tilde{\xi}\>_{\ogs}=\frac12(\<[F(X), F(Y)], F(\xi)\>_{\ogs}+\<[F(\xi), F(X)], F(Y)\>_{\ogs}+\<[F(\xi), F(Y)], F(X)\>_{\ogs})= \frac12(\<[\xi, F(X)], F(Y)\>_{\ogs}+\<[\xi, F(Y)], F(X)\>_{\ogs})$, as $F(\xi)=\xi$ and as $\xi \perp \gn= [\ogs,\ogs]$. But $[\xi, \mN]=0$, and hence $[\xi, F(\mN)]=0$ as $F(\mN) \subset \gv=\mN \oplus [\mN,\mN]$. Moreover, $[\xi, \mL]=0$ and $[\xi, F(X)]= \la_i F(X)$ for $X \in V_{\la_i}$ by Lemma~\ref{l:LF}(\ref{it:mLFi}, \ref{it:xiF}). It follows that $[\xi, F(X)]= F(SX)$ for $X \in T_xM$. So $\<\overline{S}X,Y\>= \frac12(\<F(SX), F(Y)\>_{\ogs}+\<F(SY), F(X)\>_{\ogs})= \frac12(\<SX, Y\>+\<SY, X\>)$ by the definition of $\ip_{\ogs}$. Hence $\overline{S}=S$.

  It remains to prove that $\Sh$ is Einstein. As $F(\xi)=\xi$ and $[\xi, F(X)]= F(SX)$ for $X \in T_xM$, we obtain $\Tr ((\ad_{F(\xi)})_{|\ogs})= \Tr S=H= 0$. Then by \cite[Lemma~1.4]{Wol}, the Ricci curvature of $\Sh$ is the same as the Ricci curvature of $\oSh$ in the same direction, and so $\Sh$ is Einstein, with the same Einstein constant as $\oSh$.
\end{proof}

Identifying $\tM$ with $\oSh$ as in Lemma~\ref{l:Iwasawa}\eqref{it:exps}, we get two Einstein hypersurfaces $M, \Sh \subset \tM$, both passing through $x$, and having the same unit normal and the same shape operator at $x$. Note that $M$ is minimal, and hence also is $\Sh$ as its mean curvature is constant and is zero at $x$. The following Lemma shows that they (locally) coincide.

\begin{lemma} \label{l:analyt}
  In the assumptions of Proposition~\ref{p:solv}, the Einstein hypersurface $M$ is locally uniquely determined by a unit normal vector and the shape operator \emph{(}relative to that vector\emph{)} at a single point.
\end{lemma}
\begin{proof}
First of all note that $M$ is analytic as it is a minimal submanifold of the analytic Riemannian symmetric space $\tM$ (e.g. \cite[Theorem~6.7.6]{Mry}, \cite[Lemma~1]{Leu}). Let $x \in M$ and choose geodesic normal coordinates $(y_1, \dots, y_n, y_{n+1})$ for $\tM$ centred at $x$ on a neighbourhood of $x$ in $\tM$ in such a way that $M$ is locally given by the equation $y_{n+1}=f(y_1, \dots, y_n)$, where $f$ is an analytic function with $f(0)=0, \, df(0)=0$.

The distribution $L_0$ is integrable and $\tR$-invariant, with the leaves being totally geodesic in both $M$ and $\tM$ (this easily follows from \eqref{eq:codazzi} and from Lemma~\ref{l:lak0H}\eqref{it:nolakH}). Let $F_0$ be the leaf passing through $x$. We can assume that it is locally given by equations $y_i=\phi_i(y_{n-d_0+1}, \dots, y_n)$, for $i=1, \dots, n-d_0$, where $d_0=\dim L_0$ and $\phi_i(0)=0, \, d\phi(0)=0$ (note that the functions $\phi_i$ are analytic, as $F_0$ is totally geodesic). At the point $x$, the condition that $F_0$ is totally geodesic in $\tM$ gives that the second derivatives of $\phi_i$ at zero vanish and $f_{kl}(0)=0$ for $k, l > n-d_0$ (where the subscripts of $f$ denote the corresponding partial derivatives). Then the coordinates $y'_i$ given by $y'_i=y_i-\phi_i(y_{n-d_0+1}, \dots, y_n)$, for $i=1, \dots, n-d_0$, $y'_k=y_k$ for $k=n-d_0+1, \dots, n$, and $y'_{n+1}=y_{n+1}-f(0, \dots, 0, y_{n-d_0+1}, \dots, y_n)$ are still normal at $x$, as we only change the terms of order greater than $2$. Note that $F_0$ (and hence the change of coordinates) is uniquely determined by the vector $\xi(x)$ ($F_0$ is the unique totally geodesic submanifold passing through $x$ and tangent to $\Ker \tR_\xi (x) \cap \xi^\perp(x)$).

Dropping the dashes from and changing $f$ accordingly we obtain that $M$ is locally given by the equation $y_{n+1}=f(y_1, \dots, y_n)$, where $f$ is an analytic function with $f(0)=0, \, df(0)=0$ and $f(0, \dots, 0, y_{n-d_0+1}, \dots, y_n)=0$. We can further specify the (analytic normal) coordinates $y_i$ in such a way that $f_{ij}(0)=\la_i \delta_{ij}$, for all $i,j =1, \dots, n$ (note that $\la_i \ne 0$ for $i \le n-d_0$). Denote $\bar{g}_{ab}$ and $\bar{\G}_{ab,c}, \, a,b,c =1, \dots, n+1$, the matrix of the first fundamental form and the Christoffel symbols of $\tM$ relative to $y_a, \, a=1, \dots, n+1$ (we have $\bar{g}_{ab}(0)=\delta_{ab}$ and $\bar{\G}_{ab,c}(0)=0$), and $g_{ij}$ the matrix of the induced first fundamental form on $M$ relative to $y_i, \, i=1, \dots, n$ (we have $g_{ij}(0)=\delta_{ij}$). The components $\xi^a$ of $\xi$ depend only on $y_i, f$ and the first derivatives of $f$, and $\xi^a(0)=\delta_{a,n+1}$. Denote $\partial_i$ the tangent vector to $M$ in the direction of $y_i, \, i=1, \dots, n$, that is, $(\partial_i)^i=1, \, (\partial_i)^{n+1}=f_i$, and $(\partial_i)^a=0$ for $a \ne i, n+1$. The components of the second fundamental form $h_{ij}$ of $M$ are given by $h_{ij}=\bar{g}_{a,n+1}\xi^a f_{ij} + \bar{\G}_{ab,c} \psi^{abc}_{ij}$, for $i,j=1, \dots, n$, where we use the Einstein summation convention, and where the functions $\psi^{abc}_{ij}=(\partial_i)^a (\partial_j)^b \xi^c$ depend only on $y_i, f$ and the first derivatives of $f$. Note that $h_{ij}(0)=f_{ij}(0)=\la_i \delta_{ij}$.

By \eqref{eq:Gaussinv} we now have $g^{ls}h_{kl}h_{sj}=-\tR(\partial_k, \xi, \xi, \partial_j)$ which gives the equations $g^{ls} \Omega f_{kl}f_{sj} + f_{kl} \bar{\G}_{ab,c} \rho^{abcl}_{j}+ f_{sj} \bar{\G}_{ab,c} \rho^{abcs}_{k} = \Phi_{kj}$, where the functions $\Omega = (\bar{g}_{a,n+1}\xi^a)^2, \rho^{abcl}_{j}, \rho^{abcs}_{k}$ and $\Phi_{kj}$ depend only on $y_i, f$ and the first derivatives of $f$ (note that $\Omega(0)=1$). Differentiating this equation $m \ge 1$ times with respect to $y_{i_1}, \dots, y_{i_m}, \; i_1, \dots, i_m =1, \dots, n$, we obtain $g^{ls} \Omega (f_{kli_1\dots i_m}f_{sj} + f_{kl}f_{sji_1\dots i_m})+ f_{kli_1\dots i_m} \bar{\G}_{ab,c} \rho^{abcl}_{j}+ f_{sji_1\dots i_m} \bar{\G}_{ab,c} \rho^{abcs}_{k} = \Phi_{kji_1\dots i_m}$, where $\Phi_{kji_1\dots i_m}$ depend only on $y_i, f$ and the derivatives of $f$ up to order $m+1$. Evaluating at $x$ (that is, at $y=0$) we get $f_{kji_1\dots i_m}(0) (\la_j + \la_k) = (\Phi_{kji_1\dots i_m})_{|y=0}$ (no summation by $j, k$). If we know all the partial derivatives of $f$ at $y=0$ up to order $m+1$, we can then compute all the partial derivatives $f_{i_1\dots i_mi_{m+1}i_{m+2}}(0)$ of order $m+2$, provided $\la_{i_p}+\la_{i_q} \ne 0$ for at least one pair of $p,q=1, \dots, m+2$ with $p \ne q$. But if the latter condition is violated, then $i_1,\dots, i_m,i_{m+1},i_{m+2} > n-d_0$, and so $f_{i_1\dots i_mi_{m+1}i_{m+2}}(0)=0$, as $f(0, \dots, 0, y_{n-d_0+1}, \dots, y_n)=0$ from the above. It follows that all the partial derivatives of all orders of the function $f$ at $y=0$ are uniquely determined, and so by analyticity, $M$ is uniquely determined by $\xi(x)$ and the shape operator $S$ at $x$.
\end{proof}

This completes the proof of Proposition~\ref{p:solv}.

\section{Einstein hypersurfaces in \texorpdfstring{$\SL(3)/\SO(3)$}{SL(3)/SO(3)} and \texorpdfstring{$\SU(3)/\SO(3)$}{SU(3)/SO(3)}}
\label{s:su3/so3}

In this section, we consider Einstein hypersurfaces in the symmetric spaces $\SL(3)/\SO(3)$ and $\SU(3)/\SO(3)$ (see Proposition~\ref{p:singH=0}\eqref{it:eithersl3so3}). Recall that by Proposition~\ref{p:allsing}, the Gauss map is singular at every point of $M$, and in particular we have $C=0$ (so that $M$ and $\tM$ have the same Einstein constants), by Lemma~\ref{l:nons}\eqref{it:sing}. We start with the following proposition.

\begin{proposition} \label{p:su3so3gen}
Let $\tM$ be one of the spaces $\SL(3)/\SO(3)$ or $\SU(3)/\SO(3)$, and let $M$ be an Einstein $C^4$-hypersurface in $\tM$. We have the following.
  \begin{enumerate}[label=\emph{(\alph*)},ref=\alph*]
      \item \label{it:su3so3allaal}
      The vector field $\xi$ is singular at every point of $M$. Up to relabelling, we have $\al_3=\al_4=0$, $\al_1=\al_2=\al \ne 0$ \emph{(}where $2\al$ is the Einstein constant of $\tM$\emph{)}, and $\la_3=\la_4=0, \; \la_1 \la_2=\al$.

      \item \label{it:fibered}
      The hypersurface $M$ is foliated by totally geodesic $2$-dimensional submanifolds of $\tM$ of constant curvature $\frac43 \al$. At every point $x \in M$, the tangent space to the leaf is the subspace $L_0=\Span(X_3,X_4) \subset T_xM$ of vectors commuting with $\xi(x)$. Moreover, the vector field $\xi$ is parallel along the leaves \emph{(}so that $M$ is of conullity $2$\emph{)}. % explain the constant and add CP^2!  \emph{(}the kernel of $\tR_\xi$\emph{)}

      \item \label{it:la1nela2}
      If $\tM=\SL(3)/\SO(3)$, then $M'=M$. If $\tM=\SU(3)/\SO(3)$, then $M'$ is the set of points of $M$ at which $\la_1 \ne \la_2$, and moreover, $M'$ has a nonempty intersection with every leaf defined in~\eqref{it:fibered}.

      \item \label{it:cp2}
      The hypersurface $M$ is locally isometric to $\bc P^2$ when $\al >0$ (respectively, to $\bc H^2$ when $\al < 0$), with the same Einstein constant $2\al$.
  \end{enumerate}
\end{proposition}
\begin{proof}
By Lemma~\ref{l:nons}\eqref{it:sing} we have $\la_i=\al_i=0$ for some $i \in \{1, 2, 3, 4\}$ (that is, $\tR_\xi$ and $S$ have a non-trivial common kernel), and $C=0$.

For assertions~\eqref{it:su3so3allaal} and \eqref{it:fibered}, take an arbitrary $x \in M'$. First suppose that $\xi$ is regular at $x$, and denote $\ga \subset T_x\tM$ the Cartan subalgebra containing $\xi$. We have $\dim \ga = 2$, and the restricted root system $\Delta$ has type $\mathrm{A}_2$. Then $T_xM$ is the direct orthogonal sum of the three one-dimensional root spaces corresponding to the root vectors $\beta_1, \beta_2, \beta_3$ (which we can choose in such a way that $\beta_1+\beta_2+\beta_3=0$), and the one-dimensional space $\ga'$. Up to relabelling, we have $\al_4=0$ and $\al_i=\ve \<\beta_i, \xi\>^2$ for $i=1,2,3$. As $\xi$ is regular, $\al_1, \al_2, \al_3 \ne 0$ (and so $\la_4=0$). Note that we cannot have $\al_1=\al_2=\al_3$, as $\beta_1+\beta_2+\beta_3=0$. Suppose two of them are equal, say $\al_2=\al_3 \ne \al_1$. Then by \eqref{eq:Gaussal} we have $-\la_2^2+H\la_2=-\la_3^2+H\la_3=\al_2$. If $\la_2 \ne \la_3$, then $\la_2 + \la_3 =H$, and so from $\sum_{i=1}^{4} \la_i = H$ we obtain $\la_1=0$, which by \eqref{eq:Gaussal} implies $\al_1=0$, a contradiction. It follows that if  $\al_2=\al_3$, then $\la_2=\la_3$, and so the vectors $X_i, \, i=1,2,3$, can be chosen in the corresponding root spaces $\gp_{\beta_i}$ (if $\al_1, \al_2, \al_3$ are pairwise unequal, this is also true, as the root spaces are one-dimensional). Furthermore, we have $\la_1+\la_2+\la_3=H$, and so from \eqref{eq:Gaussal} we obtain $\la_i\la_j=\frac12 (\al_i+\al_j-\al_k)$ for $\{i,j,k\}=\{1,2,3\}$. Note that $\tR(X_1,X_2,X_3,\xi) \ne 0$ by Lemma~\ref{l:rr}\eqref{it:rr1} (as $\Delta$ is of type $\rA_2$), and that the numbers $\eta_i=\<\beta_i,\xi\>,\; i=1,2,3$, satisfy $\eta_1+\eta_2+\eta_3=0$ (and $\al_i=\ve \eta_i^2$). Then equation~\eqref{eq:etala} of Lemma~\ref{l:Rijk} takes the form $\eta_1 \eta_2 \eta_3 = - \eta_1^3 - \eta_2^3 - \eta_3^3$. But as $\eta_3=-\eta_1-\eta_2$ we get $\eta_1\eta_2(\eta_1+\eta_2)=0$, and so one of $\eta_i, \, i=1,2,3$, is zero contradicting the fact that $\al_i \ne 0$ for $i=1,2,3$.

It follows that $\xi$ is singular. Then the restriction of $\tR_\xi$ to $T_xM$ has two different eigenvalues: up to relabelling, we have $\al_3=\al_4=0, \; \al_1=\al_2=\al \ne 0$ (note that $\al$ is constant, as $2\al$ is the Einstein constant of $\tM$). By Lemma~\ref{l:nons}\eqref{it:sing} we have $\la_i=0$ for at least one of $i=3,4$; let $\la_4=0$. Then $\la_1+\la_2+\la_3=H$, and if $\la_3 \ne 0$, we get $\la_3 = H \ne 0$ from \eqref{eq:Gaussal}, and so $\la_1+\la_2=0$ which gives $\al=0$ by \eqref{eq:Gaussal}, a contradiction. It follows that $\la_3=0$ and then $\la_1\la_2 = \al$.

From \eqref{eq:dG2nR} we obtain
\begin{equation} \label{eq:su3so3G1}
  \G_{ij}^{\hphantom{i}k}=0, \qquad \G_{kr}^{\hphantom{k}i}= \al^{-1} \la_k \tR(X_i, X_k, X_r, \xi),
\end{equation}
where $k,r \in \{1,2\}, \, i,j \in \{3,4\}$. It follows that the distribution $\Span(X_3, X_4)$ is integrable, with the leaves being totally geodesic in $M$. Moreover, $\xi$ is parallel along this distribution, and the subspace $\Span(X_3,X_4,\xi)$ is $\tR$-invariant. Then by \cite[Theorem~1]{DSV}, through every point of $M$ passes a $3$-dimensional totally geodesic submanifold of $\tM$ which locally intersects $M$ by the corresponding totally geodesic leaf. It follows that $M$ is locally foliated by $2$-dimensional leaves which are totally geodesic in $\tM$, and hence are of constant curvature. This curvature equals the sectional curvature $\kappa$ of $\tM$ along $\Span(X_3, X_4)$. A direct calculation shows that $\kappa=\frac43 \al$ (one can use the basis from the proof of assertion~\eqref{it:cp2} below). This proves assertions~\eqref{it:su3so3allaal} and \eqref{it:fibered} at the points of $M'$, and hence on the whole of $M$, as $M' \subset M$ is open and dense, and by continuity, the distribution $L_0$ is $2$-dimensional, integrable and totally geodesic on $M$.

For assertion~\eqref{it:la1nela2}, we note that by assertions~\eqref{it:su3so3allaal}, at all points of $M$, the set of eigenvalues of $\tR_\xi$ has cardinality $2$. The set of eigenvalues of $S$ can have cardinalities $2$ or $3$ depending on whether $\la_1=\la_2$ or not, respectively. The former cannot occur when $\tM=\SL(3)/\SO(3)$, as $\al < 0$, which proves the first statement. To prove the second statement it suffices to show that the (relative) interior of the intersection of the set on which $\la_1=\la_2$ with every leaf defined in assertion~\eqref{it:fibered} is empty. Assume that we have $\la_1=\la_2 =\la$ on an open subset $\mU_0$ of one of the leaves. Then by assertions~\eqref{it:su3so3allaal}, on $\mU_0$ we have $\la= \pm \sqrt{\al}$, which is a nonzero constant (note that $\mU_0$ may lie in $M \setminus M'$, so that we cannot apply equations from Section~\ref{s:pfgen} directly). On a neighbourhood of $\mU_0$ in $M$, choose a vector field $Y$ tangent to the distribution $L_0$ (to the leaves) and a vector field $X$ orthogonal to $L_0$. Note that $\tR(Y,\xi)=0$, and that at the points of $\mU_0$ we have $SX=\la X$. From Codazzi equation on $\mU_0$ we get $\tR(Y,X,X,\xi)= -\la\<\nabla_XX,Y\>$. On the other hand, differentiating the equation $\tR(Y, \xi,\xi,X)=0$ along $X$ we obtain $\tR(Y,X,X,\xi)= \la\<\nabla_XX,Y\>$ at the points of $\mU_0$. It follows that $\tR(Y,X,X,\xi)=0$. As the bilinear form $(T_1,T_2) \mapsto \tR(Y,T_1,T_2,\xi)$ is symmetric on $\gp \times \gp$ by the first Bianchi identity, and is zero when $T_2 \in L_0 \oplus \br \xi$, it must then be identically zero, in contradiction with Lemma~\ref{l:rr}\eqref{it:rrnot2hyper}.

To prove assertion~\eqref{it:cp2} we can assume that $x \in M'$, as the inner metric of $M$ is analytic by \cite[Theorem~5.2]{DTK}. Define an almost Hermitian structure $J$ on $T_xM$ in such a way that $JX_1=X_2$ and $JX_3=X_4$ (and we will specify it further a little later). The K\"{a}hler  condition $\nabla J=0$ is then equivalent to $\G_{i3}^{\hphantom{i}2}+\G_{i4}^{\hphantom{i}1} = \G_{i4}^{\hphantom{i}2} -\G_{i3}^{\hphantom{i}1}=0$, for $i=1,2,3,4$. By the first equation of \eqref{eq:su3so3G1}, this is trivially satisfied for $i=3,4$. For $i=1,2$, from the second equation of \eqref{eq:su3so3G1} we get
\begin{equation}\label{eq:Kahler}
\tR(X_4, X_1, X_1, \xi)+\tR(X_3, X_1, X_2, \xi)=\tR(X_4, X_1, X_2, \xi)-\tR(X_3, X_1, X_1, \xi)=0,
\end{equation}
and another two equations obtained from these two by simultaneously interchanging $X_1$ with $X_2$, and $X_3$ with $X_4$; they are equivalent to~\eqref{eq:Kahler}, as for $i=3,4$, we have $\tR(X_i, X_1, X_2, \xi) = \tR(X_i, X_2, X_1, \xi)$ by the first Bianchi identity, and $\tR(X_i, X_1, X_1, \xi) = -\tR(X_i, X_2, X_2, \xi)$ since $\tM$ is Einstein.

In case $\tM=\SU(3)/\SO(3)$, the space $\gp$ can be identified with $\ri \cdot \Sym^0(3,\br)$, where $\Sym^0(3,\br)$ is the space of real, symmetric $(3 \times 3)$-matrices with zero trace, and with the inner product defined by $\<P_1, P_2\>=\Re \Tr(P_1P_2^*)$ (which has Einstein constant $3$). Up to the adjoint action of the isotropy group $\SO(3)$, we choose an orthonormal basis for $T_x\tM$ such that $\xi=\frac{1}{\sqrt{6}}\ri \,  \diag(-2,1,1), \; X_3= \frac{1}{\sqrt{2}}\ri \,  \diag(0,1,-1), \; X_4=\frac{1}{\sqrt{2}} \ri \,  \left(\begin{smallmatrix} 0 & 0 & 0 \\ 0 & 0 & 1 \\0 & 1 & 0 \end{smallmatrix}\right), \; X_1=\frac{1}{\sqrt{2}} \ri \,  \left(\begin{smallmatrix} 0 & 1 & 0 \\ 1 & 0 & 0 \\0 & 0 & 0 \end{smallmatrix}\right), \; X_2=\frac{1}{\sqrt{2}} \ri \,  \left(\begin{smallmatrix} 0 & 0 & 1 \\0 & 0 & 0 \\ 1 & 0 & 0 \end{smallmatrix}\right)$ (this particular choice specifies $J$). Then a direct calculation shows that \eqref{eq:Kahler} is satisfied, and so $M$ is K\"{a}hler, and that the holomorphic sectional curvature is constant. In case $\tM=\SL(3)/\SO(3)$, the argument is word-by-word the same, up to replacing $\ri$ by $1$. This proves assertion~\eqref{it:cp2}.
\end{proof}

\section{\texorpdfstring{$\SU(3)/\SO(3)$}{SU(3)/SO(3)} and Legendrian geometry}
\label{s:su3so3}

Let $\tM=\SU(3)/\SO(3)$ (the \emph{Wu manifold}). We use the standard isometric embedding of $\SU(3)/\SO(3)$ in the Euclidean space of $3 \times 3$ complex matrices $\mathrm{M}_{3\times 3} (\bc)$ defined by $Q \mapsto QQ^t$ for $Q \in \SU(3)$ \cite[\S~5]{Kob}. The image of this embedding (which we will identify with $\SU(3)/\SO(3)$) is the set $\Sym(3, \bc) \cap \SU(3)$ of all symmetric matrices in $\SU(3)$, with the left action of the group $\SU(3)$ defined by $\mL_gP=gPg^t$, for $g \in \SU(3), \; P \in \SU(3)/\SO(3)$. The Euclidean inner product on $\mathrm{M}_{3\times 3} (\bc)$ defined by $\<Q_1,Q_2\>=\Re \Tr (Q_1Q_2^*)$ induces the metric on $\SU(3)/\SO(3)$ with the Einstein constant $\frac34$, so that $\al=\frac38$ by Proposition~\ref{p:su3so3gen}\eqref{it:su3so3allaal}. %(note that this inner product is proportional to the one in the proof of Proposition~\ref{p:su3so3gen}\eqref{it:cp2} which has Einstein constant $3$)

In the space $\br^6=\bc^3$ with the standard Euclidean inner product and the complex structure, consider the unit sphere $S^5$. For $Z \in S^5$, define the subset $S_Z$ of $\SU(3)/\SO(3)=\Sym(3, \bc) \cap \SU(3)$ in one of the two equivalent ways:
\begin{equation}\label{eq:S_Z}
  S_Z= \{x \in \Sym(3, \bc) \cap \SU(3) \, | \, x \overline{Z} = Z\} = \{QQ^t \, | \, Q \in \SU(3), \, Q (1,0,0)^t=Z\}.
\end{equation}
For $Z=(1,0,0)^t$, we have $S_Z= \{\left(\begin{smallmatrix} 1 & 0 \\ 0 & hh^t \end{smallmatrix}\right) \, | \, h \in \SU(2) \}$, which is a totally geodesic $2$-sphere $\SU(2)/\SO(2) \subset \SU(3)/\SO(3)$. It is easy to see that $S_{gZ}=gS_Zg^t$ for $g \in \SU(3)$, and so $S_{Z_1}=S_{Z_2}$ if and only if $Z_1=\pm Z_2$. Therefore we obtain a two-to-one map $Z \mapsto S_Z$ from $S^5$ to the space of totally geodesic $2$-spheres in $\SU(3)/\SO(3)$. It is not hard to see that this map is surjective, so that the space of (non-oriented) totally geodesic $2$-spheres of $\SU(3)/\SO(3)$ can be identified with the homogeneous space $\br P^5=\SU(3)/(\SU(2) \times \mathbb{Z}_2)$ (note that $\SU(3)/\SO(3)$ has another family of totally geodesic, congruent, $2$-dimensional submanifolds of constant positive Gauss curvature, but they are real projective planes, and have a different curvature).

Given a $2$-dimensional surface $F^2 \subset S^5$, the above construction gives a hypersurface in $\SU(3)/\SO(3)$ foliated by totally geodesic $2$-spheres. We prove the following (Theorem~\ref{th:main}\eqref{thit:wu}).

\begin{proposition} \label{p:su3so3Leg}
  Let $M \subset \SU(3)/\SO(3)$ be a connected $C^4$-hypersurface. Then $M$ is Einstein if and only if $M$ is an open domain of $M_{F^2}=\cup_{Z \in F^2} S_Z$, where $F^2 \subset S^5$ is a special Legendrian surface with the Legendrian angle $\frac{\pi}{2} $.
\end{proposition}

%(note that without choosing a particular orientation, we the Legendrian angle function is defined up to adding an integer multiple of $\pi$).
% (oriented)  positively oriented  (where $Z_1, Z_2$ and $Z$ are viewed as elements of $\bc^3$) , \cite[Proposition~2.2]{CLU}
A $2$-dimensional surface $F^2 \subset S^5 \subset \br^6=\bc^3$ is called \emph{Legendrian} if at every point $Z \in F^2$, the vector $\ri Z$ is normal to $F^2$. This is equivalent to the fact that the cone over $F^2$ is \emph{Lagrangian}, that is, its tangent space at every point other than the origin is totally real. Given a Legendrian surface $F^2 \subset S^5$, the \emph{Legendrian angle function} $\beta: F^2 \to \br/ \pi\mathbb{Z}$ is defines as follows: at $Z \in F^2$, choose an orthonormal basis $Z_1, Z_2$ for $T_ZF^2$; then $\det_{\bc}(Z, Z_1, Z_2)=\pm e^{i\beta(Z)}$. It is well-known (and is easy to see) that a Legendrian surface in $S^5$ is \emph{minimal} if and only if the (Lagrangian) cone over it is minimal in $\br^6$. Such a cone is minimal if and only if the Legendrian angle function is constant \cite[Proposition~2.17]{HL}, \cite{Mor}. Legendrian surfaces with constant Legendrian angle $\beta$ are called \emph{$\beta$-special Legendrian}.

\begin{remark} \label{rem:spLe} % by i ln \det g?
  As a side remark, note that if a surface $F^2 \subset S^5$ is $\beta$-special Legendrian and $g \in \mathrm{U}(3)$, then the surface $gF^2 \subset S^5$ is also $(\beta-\ri \ln \det g)$-special Legendrian. In this sense, special Legendrian surfaces with different (constant) Legendrian angles are ``indistinguishable" from the point of view of Legendrian geometry. However, the hypersurfaces in $\SU(3)/\SO(3)$ obtained from them using the above construction will have different geometric properties. In particular, choosing $g \in \mathrm{U}(3)$ with $\det g = \pm \ri$ we obtain a hypersurface in $\SU(3)/\SO(3)$ which is \emph{parallel to $M$}; it will no longer be Einstein, but will be minimal (and still foliated by totally geodesic spheres). The above construction can therefore be used for finding examples of minimal hypersurfaces in $\SU(3)/\SO(3)$ starting from $0$-special Legendrian surfaces in $S^5$.
\end{remark}

There are many special Legendrian surfaces in $S^5$. A local description can be obtained from \cite[Theorem~2.3]{HL}. If $z_j=x_j + \ri y_j, \, j=1,2,3$, are complex coordinates in $\bc^3$, a $(0)$-special Lagrangian submanifold is locally given by the graph $y_j=\frac{\partial}{\partial x_j}\Phi$, where $\Phi=\Phi(x_1,x_2,x_3)$ satisfies the equation $\Delta \Phi = \det (\mathrm{Hess} \, \Phi)$. To obtain a Legendrian surface we additionally require that $\Phi$ is $2$-homogeneous (then the resulting Lagrangian submanifold is a cone) and take its intersection with $S^5$. Note that to obtain a Legendrian surface with $\beta = \frac{\pi}{2} $ we will then need to rotate by some $g \in \mathrm{U}(3)$ with $\det g = \pm \ri$ (or alternatively require that $\Phi$ satisfies equation~$(2.4)'$ of \cite{HL}, with $\Im$ replaced by $\Re$). In the compact case, a special Legendrian surface of genus zero is a totally geodesic (Legendrian) $S^2 \subset S^5$ \cite{Yau}, but there exist infinite families of pairwise non-congruent special Legendrian $S^1$-invariant tori \cite{Has}.
% Th 7 of part 1 + legen correspond to min in CP2

The following are two simplest examples of Einstein hypersurfaces in $\SU(3)/\SO(3)$ (which are the orbits of certain subgroups of $\SU(3)$).
\begin{example} \label{ex:totrealS2}
  Let $F^2 \subset S^5$ be a totally geodesic Legendrian sphere with the Legendrian angle $\frac{\pi}{2} $. If $z_j=x_j + \ri y_j, \, j=1,2,3$, are complex coordinates in $\bc^3$ we can define $F^2$ as the unit sphere in the subspace $x_1=x_2=x_3=0$. The corresponding Einstein hypersurface $M \subset \SU(3)/\SO(3)$ is the set of matrices $gg^t, \; g \in \SU(3)$, such that the first column of $g$ is purely imaginary. This hypersurface is in fact the $\SO(3)$-orbit of a single totally geodesic $2$-sphere $S_Z \subset \SU(3)/\SO(3)$, where $Z=(\ri, 0, 0)^t$, or even of a single geodesic (or alternatively, an open domain of the set of matrices in $\SU(3)/\SO(3)=\Sym(3, \bc) \cap \SU(3)$ having an eigenvalue $-1$).
\end{example}

\begin{example} \label{ex:totrealtorus}
  Let $F^2 \subset S^5$ be the Legendrian torus (with the Legendrian angle $\frac{\pi}{2} $) given by $Z(u,v)=\frac{1}{\sqrt{3}} \ri (e^{iu},e^{iv},e^{-i(u+v)}), \; u,v \in \br$. Then, again, the corresponding Einstein hypersurface $M \subset \SU(3)/\SO(3)$ is the set of matrices $gg^t, \; g \in \SU(3)$, such that the first column of $g$ is $Z(u,v)$. This hypersurface is the orbit of a single totally geodesic $2$-sphere $S_Z \subset \SU(3)/\SO(3)$, with $Z=\frac{1}{\sqrt{3}}\ri (1, 1, 1)^t$, under the action of the maximal torus of diagonal matrices of $\SU(3)$.
\end{example}

\begin{proof}[Proof of Proposition~\ref{p:su3so3Leg}]
Suppose $M$ is Einstein. Let $x \in M$. By Proposition~\ref{p:su3so3gen}\eqref{it:fibered}, there is an open domain of a totally geodesic $2$-sphere $S^2=S^2(x) \subset \SU(3)/\SO(3)$ lying in $M$ and passing through $x$ such that $\xi(x)$ commutes with $T_xS^2$ (moreover, by Proposition~\ref{p:su3so3gen}\eqref{it:la1nela2}, every such sphere passes through a point of $M'$, and so we can assume $x \in M'$). By a left action $\mL_{g^{-1}}, \; g \in \SU(3)$, we can map $x$ to the identity matrix $I \in \SU(3)/\SO(3)$; note that $T_I \SU(3)/\SO(3)=\Sym(3, \bc) \cap \su(3) = \ri \cdot \Sym^0(3, \br)$, where $\Sym^0$ is the space of symmetric matrices with trace zero. Then $d\mL_{g^{-1}}\xi(x) \in T_I \SU(3)/\SO(3)$ is a unit vector with the three-dimensional centraliser. Up to the isotropy action by an element of $\SO(3)$, we can assume that $d\mL_{g^{-1}}\xi(x)=\frac{1}{\sqrt{6}}\ri \diag(-2,1,1)=\frac{1}{\sqrt{6}}\ri (I-3(1,0,0)^t(1,0,0))$. We obtain $x=gg^t$ and $\xi(x)=\frac{1}{\sqrt{6}}\ri (x-3(g(1,0,0)^t)(g(1,0,0)^t)^t)$. Note that such an element $g \in \SU(3)$ is defined up to the right multiplication by an element of $\mathrm{O}(2) \subset \SO(3)$ of the form $\left(\begin{smallmatrix} \det A & 0 \\ 0 & A \end{smallmatrix}\right), \; A \in \mathrm{O}(2)$, and so we obtain a unit vector $Z(x)=g(1,0,0)^t$ in $\bc^3=\br^6$ defined up to a sign. As we work locally, we choose and fix a particular sign such that $Z(x)$ is continuous. Then we have $\xi(x)=\frac{1}{\sqrt{6}}\ri (x-3ZZ^t)$ and $S^2(x)=S_{Z(x)}$. Moreover, as $g \in \SU(3)$, we obtain $x\overline{Z}=gg^t\overline{g}(1,0,0)^t=g(1,0,0)^t=Z$ by~\eqref{eq:S_Z}. We now have $d\mL_{g^{-1}}T_xM=\{ \ri \left(\begin{smallmatrix} 0 & u^t \\ u & P \end{smallmatrix}\right) \, | \, u \in \br^2, \, P \in \Sym^0(2, \br) \}$. The centraliser of $d\mL_{g^{-1}}\xi(x)$ in $d\mL_{g^{-1}}T_xM$ is the $2$-dimensional space $d\mL_{g^{-1}}T_xS^2(x)=\{ \ri \left(\begin{smallmatrix} 0 & 0 \\ 0 & P \end{smallmatrix}\right) \, | \, P \in \Sym^0(2, \br) \}$. Therefore at $x \in M$, we get $\Span(X_3,X_4)=T_xS^2(x)=\{ \ri g\left(\begin{smallmatrix} 0 & 0 \\ 0 & P \end{smallmatrix}\right)g^t \, | \, P \in \Sym^0(2, \br) \}$ and $\Span(X_1,X_2)=\{ \ri g\left(\begin{smallmatrix} 0 & u^t \\ u & 0 \end{smallmatrix}\right)g^t \, | \, u \in \br^2\}$.

For $X \in T_xM$ denote $Z_X=(dZ)_x X$. We have $\xi=\frac{1}{\sqrt{6}}\ri (x-3ZZ^t)$, and so $\on_X \xi=\pi_{T_x\SU(3)/\SO(3)}\frac{1}{\sqrt{6}}\ri (X-3Z_XZ^t-3ZZ^t_X)$, where $\pi_{T_x\SU(3)/\SO(3)}$ is the orthogonal projection. Without loss of generality (as both the assumption and the conclusion of the proposition are $\SU(3)$-invariant) we can now take $x=I, \, g= I$ and $Z(x)=(1,0,0)^t$. Then $\ri X$ is a real symmetric matrix, and so $\pi_{T_I\SU(3)/\SO(3)}(\ri X)=0$. Moreover, differentiating the equation $x\overline{Z}= Z$ (see \eqref{eq:S_Z}) we get $X(1,0,0)^t+\overline{Z_X}=Z_X$. Then for $X \in \Span(X_3,X_4)$ we have $X(1,0,0)^t=0$ and so $Z_X \in \br^3$. Then $\on_X \xi=-\frac{3}{\sqrt{6}}\ri (Z_XZ^t+ZZ^t_X)$. As $\on_X \xi=0$ for $X \in \Span(X_3,X_4)$ we obtain $Z_X = 0$. For $j=1,2$, let $u_j \in \br^2$ be defined by $X_j = \ri \left(\begin{smallmatrix} 0 & u_j^t \\ u_j & 0 \end{smallmatrix}\right)$. Then we have $\ri \left(\begin{smallmatrix} 0 \\ u_j \end{smallmatrix}\right) + \overline{Z_{X_j}} = Z_{X_j}$, and so $Z_{X_j}=\left(\begin{smallmatrix} r_j \\ \frac12 \ri u_j + w_j \end{smallmatrix}\right)$ for $j=1,2$, where $r_j \in \br, \, w_j \in \br^2$. As $Z(x)$ is unit, we get $r_j=0$ which gives $Z_{X_j}=\left(\begin{smallmatrix} 0 \\ \frac12 \ri u_j + w_j \end{smallmatrix} \right)$. It follows that $Z_{X_1}, Z_{X_2}$ are $(\br$-)linearly independent (so that locally the set of points $Z(x)$ is a $2$-dimensional surface $F^2 \subset S^5$ with the tangent space spanned by $Z_{X_1}, Z_{X_2}$; note that $X_1$ and $X_2$ are well-defined as $x \in M'$). Moreover, we have $Z \perp Z_{X_j}, \ri Z_{X_j}$ for $j=1,2$, and so the surface $F^2 \subset S^5$ is Legendrian. Then its tangent space $\Span(Z_{X_1}, Z_{X_2})$ is totally real which gives $\<u_1,w_2\>=\<u_2,w_1\>$.

For $j=1,2$, from the above we get $\on_j \xi=-\frac{3}{\sqrt{6}}\ri \left(\begin{smallmatrix} 0 & w_j^t \\ w_j & 0 \end{smallmatrix}\right)$. We compute $\<X_j,X_k\>=2\<u_j,u_k\>$ and $\<\on_j \xi, X_k\>=-\sqrt{6}\<u_k,w_j\>$ for $j,k=1,2$. It follows that the eigenvalues $\la_1$ and $\la_2$ of $S$ are the roots of the equation $\det (\sqrt{6}(w_1\,|\,w_2)^t(u_1\,|\,u_2)-2 (u_1\,|\,u_2)^t(u_1\,|\,u_2) \la)=0$. From Proposition~\ref{p:su3so3gen}\eqref{it:su3so3allaal}, the Einstein condition for $M$ gives $\la_1\la_2=\al$, where $2\al=\frac34$, the Einstein constant of $\SU(3)/\SO(3)$. We obtain $4\det(w_1\,|\,w_2)=\det(u_1\,|\,u_2)$ which is equivalent to $\det(w_1+\frac12 \ri u_1\,|\,w_2+\frac12 \ri u_2) \in \ri \br$. It follows that $\det(Z\,|\,Z_{X_1}\,|\,Z_{X_2}) \in \ri \br$. This condition means that the Legendrian angle of the Legendrian surface $F^2$ is constant and is equal to $\frac{\pi}{2}$, as required.

To prove the converse, we start with a Legendrian surface $F^2 \subset S^5$ with a constant Legendrian angle $\frac{\pi}{2}$ and revert the argument (at the regular points of $M_{F^2}$ --- see Remark~\ref{rem:Legsing}).
\end{proof}

\begin{remark} \label{rem:Legsing}
  Note that $M_{F^2}$, being an immersed hypersurface at all the ``good" points, has singularities on every single leaf $S_Z$. To see this, take an arbitrary point $Z$ on a $\frac{\pi}{2}$-special Legendrian surface $F^2 \subset S^5$. Up to an action by an element $g \in \SU(3)$ one can assume that $Z=(1,0,0)^t, \, Z_1=\partial_{u_1}Z=(0,1,0)^t, \, Z_2=\partial_{u_2}Z=(0,0,\ri)^t$ (for some local coordinates $(u_1,u_2)$ on a neighbourhood of $Z \in F^2$). Then $S_Z= \{\left(\begin{smallmatrix} 1 & 0 \\ 0 & A \end{smallmatrix}\right) \, | \, A \in \SU(2) \cap \Sym(2, \bc)\}$, and one can see that all the points on the circle lying in $S_Z$ and given by $A=\left(\begin{smallmatrix} \cos \theta & \ri \sin \theta \\ \ri \sin \theta & \cos \theta \end{smallmatrix}\right), \, \theta \in \br$, are singular points of $M_{F^2}$. Indeed, differentiating the equation $x \overline{Z} = Z$ from \eqref{eq:S_Z} at the points of this circle we find that the vector $(\cos\frac{\theta}{2} \partial_{u_1} + \sin\frac{\theta}{2} \partial_{u_2}) x$ is tangent to the leaf $S_Z$ (cf. Proposition~\ref{p:Tzrk}\eqref{it:Tz2sing}).
\end{remark}

% separate proposition: 1. there is a 2-1 correspondence between Legendrian surfaces in S5 and conullity 2 hypersurfaces in SU(3)/\SO(3). 2. min then min with angle 0, pi. 3. Einstein then minimal with angle pi/2 % Maslov form?

\section{\texorpdfstring{$\SL(3)/\SO(3)$}{SL(3)/SO(3)} and equi-affine geometry}
\label{s:sl3so3}

% partial uniformly everywhere (\partial_subscript)
Our approach in the case when $\tM=\SL(3)/\SO(3)$ is somewhat similar to that for $\tM=\SU(3)/\SO(3)$. We use the standard isometric embedding of $\SL(3)/\SO(3)$ in the pseudo Euclidean space of pairs of $3 \times 3$ real matrices $\mathrm{M}=\mathrm{M}_{3\times 3} (\br) \oplus \mathrm{M}_{3\times 3} (\br)$ defined by $Q \mapsto (QQ^t,(QQ^t)^{-1})$ for $Q \in \SL(3)$. The image of this embedding (which we will identify with $\SL(3)/\SO(3)$) is the set $\{(P,P^{-1}) \, | \, P \in \Sym^+(3, \br) \cap \SL(3)\}$, where $\Sym^+(3, \br)$ is the set of positive definite, symmetric, $3 \times 3$ real matrices. We define a pseudo Euclidean inner product on $\mathrm{M}$ by $\<(P_1,P_2),(Q_1,Q_2)\>= -\frac12\Tr(P_1Q_2+Q_1P_2)$, for $P_1,P_2,Q_1,Q_2 \in \mathrm{M}_{3\times 3} (\br)$. The left action of the group $\SL(3)$ on $\mathrm{M}$ defined by $\mL_g(P_1,P_2)=(gP_1g^t,g^{-1t}P_2g^{-1})$, for $g \in \SL(3), \; (P_1,P_2) \in \mathrm{M}$, is an action by isometries. Moreover, the submanifold $\SL(3)/\SO(3)=\{(P,P^{-1}) \, | \, P \in \Sym^+(3, \br) \cap \SL(3)\}$ is the orbit of $(I,I)$, the induced metric is Riemannian, and is in fact the symmetric space metric on $\SL(3)/\SO(3)$, with the Einstein constant $-\frac34$. Note that $T_{(I,I)}\SL(3)/\SO(3)=\{(T,-T)\, | \, T \in \Sym^0(3, \br)\}$.

In the space $\br^6=\br^3 \oplus \br^3$, introduce with the para-Hermitian inner product given by $\<(p,q),(p,q)\>'=\<p,q\>$, for $p,q \in \br^3$, relative to the para-complex structure $J(p,q)=(p,-q)$. Define the ``unit sphere" $S^5_-$ and the subset $S_{(p,q)}$ of $\SL(3)/\SO(3), \; (p,q) \in S^5_-$, as follows:
\begin{equation}\label{eq:S_pq}
\begin{gathered}
    S_{(p,q)}= \{x = (P,P^{-1}) \, | \, P \in \Sym^+(3, \br) \cap \SL(3), \, P q = p\}, \\
   \text{where } (p,q) \in S^5_- =\{(p,q) \, | \, p,q \in \br^3,\, \<p,q\>=1\}\subset \br^6.
\end{gathered}
\end{equation}
The action of the group $\SL(3)$ given by $\mL_g(p,q)=(gp,g^{-1t}q),\; g \in \SL(3)$, is isometric relative to $\ip'$, is transitive on $S^5_-$, and makes $S^5_-$ into the pseudo Riemannian homogeneous space $\SL(3)/\SL(2)$. For $p=q=(1,0,0)^t$, we have $S_{(p,q)}= \{(P,P^{-1}) \, | \, P=\left(\begin{smallmatrix} 1 & 0 \\ 0 & hh^t \end{smallmatrix}\right) \, h \in \SL(2) \}$, which is a totally geodesic hyperbolic plane $\SL(2)/\SO(2) \subset \SL(3)/\SO(3)$ of curvature $-\frac12$. Furthermore, $S_{\mL_g(p,q)}=gS_{(p,q)}g^t$ for $g \in \SL(3)$, and so $S_{(p_1,q_1)}=S_{(p_2,q_2)}$ if and only if $(p_1,q_1)=\pm (p_2,q_2)$. We obtain a two-to-one map $(p,q) \mapsto S_{(p,q)}$ from $S_-^5$ to the space of totally geodesic hyperbolic planes of curvature $-\frac12$ of $\SL(3)/\SO(3)$. This map is surjective, so that the space of (non-oriented) totally geodesic hyperbolic planes of $\SL(3)/\SO(3)$ of curvature $-\frac12$ can be identified with the homogeneous space $S_-^5/\mathbb{Z}_2$ (note that $\SL(3)/\SO(3)$ also has totally geodesic hyperbolic planes of curvature $-\frac18$).

Given a $2$-dimensional surface $F^2 \subset S_-^5$, the above construction produces a hypersurface in $\SL(3)/\SO(3)$ foliated by totally geodesic hyperbolic planes of curvature $-\frac12$. We prove the following proposition which concludes the proof of Theorem~\ref{th:main}\eqref{thit:ncwu} (note that our definition of $S_{(p,q)}$ is effectively equivalent to how $S_{(p,\rho)}$ is defined in Example~\ref{ex:inncWu}).

\begin{proposition} \label{p:sl3so3Tz}
    Let $M \subset \SL(3)/\SO(3)$ be a connected $C^4$-hypersurface. Then $M$ is Einstein if and only if $M$ is an open domain of $M_{F^2}=\cup_{(p,q) \in F^2} S_{(p,q)}$, where $F^2 \subset S_-^5$ is a special para-Legendrian surface.
\end{proposition}

A regular surface $F^2 \subset S_-^5 \subset \br^6$ is called \emph{para-Legendrian} if at every point $(p,q) \in F^2$, the vector $J(p,q)$ is (para-Hermitian) normal to $F^2$. Let $(u_1,u_2)$ be a local regular parametrisation of a para-Legendrian surface $F^2 \subset S_-^5 \subset \br^6$, so that $F^2$ is locally given by $(u_1,u_2) \mapsto (p(u_1,u_2),q(u_1,u_2))$, where $p(u_1,u_2),q(u_1,u_2) \in \br^3$. The surface $F^2$ is called \emph{special  para-Legendrian} if
  \begin{equation}\label{eq:pq}
    \<p,q\>=1, \quad \<p,\partial_{u_i}q\>=\<q,\partial_{u_i}p\>=0, \, i=1,2, \quad \det(p,\partial_{u_1}p,\partial_{u_2}p)+ \det(q,\partial_{u_1}q, \partial_{u_2}q) =0
  \end{equation}
(note that the first equation means that $F^2 \subset S^5_-$, and the second one, that $F^2$ is para-Legendrian). It is easy to see that this definition of special para-Legendrian surfaces is equivalent to the ``equi-affine" definition in Example~\ref{ex:inncWu}, in which we view $\br^6$ as the cotangent bundle over $\br^3$. The connection between para-Legendrian and equi-affine geometries is well known in the literature (see, for example \cite{CFG}, \cite{CLS}). We complement the statement of Proposition~\ref{p:sl3so3Tz} with the following fact (some parts of which are also known).

\begin{proposition}\label{p:Tzrk}
In the above notation, let $F^2 \subset S_-^5$ be a special para-Legendrian surface locally given by a regular parametrisation $(u_1,u_2) \mapsto (p(u_1,u_2),q(u_1,u_2))$, so that  $p(u_1,u_2),q(u_1,u_2) \in \br^3$ satisfy \eqref{eq:pq}. Then the following holds.
    \begin{enumerate}[label=\emph{(\alph*)},ref=\alph*]
        \item \label{it:nodp0}
        We have $dp(u) \ne 0$ at all points $u=(u_1,u_2)$.

        \item \label{it:eitherh}
        If $\rk (dp)= 1$ on a neighbourhood of a point $u=(u_1,u_2)$, then up to reparametrisation and an action of $\SL(3)$ one has $p=(1,u_1,0)^t,\, q=(1,0,u_2)$. The corresponding Einstein hypersurface $M \subset M_{F^2} \subset \SL(3)/\SO(3)$ is a domain of an Einstein solvmanifold of codimension $1$, as in Example~\ref{ex:solv} and Section~\ref{s:solv}.

        \item \label{it:orTz} % maybe N; n is dim. here and in intro
        If $\rk (dp)= 2$ on a neighbourhood of a point $u=(u_1,u_2)$, then the surface $p=p(u_1,u_2)$ is a domain of a \emph{proper indefinite affine sphere} \emph{(Tzitz\`{e}ica surface)} in $\br^3$ centred at the origin, of centro-affine curvature $-1$ \emph{(}equivalently, relative to an arbitrary Euclidean inner product on $\br^3$, one has $K\<p,n\>^{-4} = -1$, where $K$ is the Gauss curvature and $n$ is a unit normal of the surface $p(u_1,u_2)$\emph{)}. Then $q=q(u_1,u_2)$, and hence $F^2$, are uniquely determined by~\eqref{eq:pq}.

        \item \label{it:Tz2sing}
        If $\rk (dp)= 2$ at some point $u=(u_1,u_2)$, then the leaf $S_{(p(u),q(u))}$ contains singular points of $M_{F^2}$. %\emph{(}so that $M_{F^2}$ is not an immersed hypersurfaces on a neighbourhood of such points\emph{)}
\end{enumerate}
\end{proposition}

Let $\br^3$ be an equi-affine space equipped with a flat affine connection $\nabla^0$ and a volume form $\omega$. Given a regular $C^2$-surface $V^2 \subset \br^3$ and a arbitrary transversal vector field $\eta$ along $V^2$, one defines, for vector fields $X,Y$ tangent to $V^2$, the decompositions $\nabla^0_X \eta= -SX+f\eta,\; \nabla^0_X Y= \nabla_X Y+h(X,Y)\eta$, where $SX$ and $\nabla_X Y$ are tangent to $V^2$. Assuming the quadratic form $h$ is non-degenerate (this condition does not depend on the choice of $\eta$), there exists a unique, up to a sign, choice of the vector field $\eta$ such that $f=0$ and $|h(X,X)h(Y,Y)-h(X,Y)^2|=(\omega(X,Y,\eta))^2$, for any vector fields $X,Y$ tangent to $V^2$. Such vector field $\eta$ is called the \emph{affine normal} (the \emph{Blaschke normal}). A surface $V^2$ (with a non-degenerate $h$) is called a \emph{proper affine sphere} (a \emph{Tzitz\`{e}ica surface}), if all its affine normals pass through one point in $\br^3$ (called the \emph{centre}). Placing the centre at the origin and choosing an arbitrary Euclidean inner product on $\br^3$, one obtains that $V^2$ is a proper affine sphere if and only if $K\<p,n\>^{-4} = c$, for some constant $c \ne 0$, where $p$ is the position vector, $K$ is the Gauss curvature and $n$ is a unit normal \cite{Tz}, \cite[Theorem~II.5.11]{NS}; a proper affine sphere is \emph{indefinite}, if $c < 0$ (equivalently, $h$ is indefinite). Note that in general the function $K\<p,n\>^{-4}$ is $\SL(3)$-invariant; it is called the \emph{centro-affine curvature}.

There exist many proper affine spheres of centro-affine curvature $-1$. We mention another description, similar to that for special Legendrian surfaces given in Section~\ref{s:su3so3}. Choose Cartesian coordinates $(x_1,x_2,x_3)$ in $\br^3$ such that $\omega(\partial_{x_1}, \partial_{x_2}, \partial_{x_3})=1$. Let $\Phi=\Phi(x_1,x_2,x_3)$ be a $2$-homogeneous function which satisfies the equation $\det (\mathrm{Hess} \, \Phi)=-8$. Then the equation $\Phi(x_1,x_2,x_3)=1$ defines a proper (indefinite) affine sphere of centro-affine curvature $-1$.

% Jonas Kelch? "kelyh", kubok
\begin{example} \label{ex:affsph}
  The simplest examples of proper affine spheres of centro-affine curvature $-1$ are the hyperboloid of one sheet $x_1^2+x_2^2-x_3^2=1$ and the ``hexenhut" $x_3(x_1^2+x_2^2)=\frac{2}{3\sqrt{3}}$. One can show that the Einstein hypersurface in $\SL(3)/\SO(3)$ constructed from the hyperboloid $x_1^2+x_2^2-x_3^2=1$ is the orbit of a totally geodesic hyperbolic plane of $\SL(3)/\SO(3)$ (in fact, of a single geodesic) under the action of the subgroup $\SO(2,1) \subset \SL(3)$. Any ruled affine sphere of centro-affine curvature $-1$ is given by $(u_1,u_2) \mapsto \gamma'(u_1)+u_2 \gamma(u_1)$, where the curve $\gamma \subset \br^3$ is parametrised in such a way that $\det(\gamma, \gamma', \gamma'')=1$ \cite[Theorem~III.5.5]{NS}. Note that any ``generic" $C^2$-curve has such a parametrisation; we also note that the corresponding Einstein hypersurface $M$ may not even be $C^3$-regular, if we choose a non-$C^3$ curve $\gamma$ in this construction (see Remark~\ref{rem:Ck}).
\end{example}

\begin{proof}[Proof of Proposition~\ref{p:sl3so3Tz}]
Suppose $M$ is Einstein. Let $x \in M$. By Proposition~\ref{p:su3so3gen}\eqref{it:fibered}, there is a domain of a totally geodesic hyperbolic plane $H^2(x) \subset \SL(3)/\SO(3)$ of curvature $-\frac12$ lying in $M$ and passing through $x$ such that $\xi(x)$ commutes with $T_xH^2(x)$ (recall that by Proposition~\ref{p:su3so3gen}\eqref{it:la1nela2}, all the points of $M$ are regular). Then there exists $g \in \SL(3)$ such that $x=\mL_g(I,I)=(gg^t,g^{-1t}g^{-1})$ and (up to a sign) we have $\xi(x)=(d\mL_g)_{(I,I)}(\frac{1}{\sqrt{6}}(\diag(-2,1,1),\diag(2,-1,-1)))$ $= \frac{1}{\sqrt{6}}(g\diag(-2,1,1)g^t,g^{-1t}\diag(2,-1,$ $-1)g^{-1})$. Note that such an element $g \in \SL(3)$ is defined up to the right multiplication by an element of $\SO(2) \subset \SO(3)$ of the form $\left(\begin{smallmatrix} \det A & 0 \\ 0 & A \end{smallmatrix}\right), \; A \in \mathrm{O}(2)$. We obtain the pair $(p,q) \in \br^3 \oplus \br^3, \,  q=g^{-1t}(1,0,0)^t, \, p=g(1,0,0)^t$ (note that the pair $(p,q)$ is defined up to a sign; we can locally choose and fix a particular sign such that $(p(x),q(x))$ is continuous). So we obtain $\xi(x)=\frac{1}{\sqrt{6}}(gg^t-3pp^t,-g^{-1t}g^{-1}+3qq^t)$. Moreover, in the notation of Proposition~\ref{p:su3so3gen} we have $\Span(X_3,X_4)=T_xS_{(p,q)}=\{(gQg^t, -g^{-1t}Qg^{-1})\}$, where $Q=\left(\begin{smallmatrix} 0 & 0 \\ 0 & P \end{smallmatrix}\right), \, P \in \Sym^0(2, \br)$, and $\Span(X_1,X_2)=\{(gQg^t, -g^{-1t}Qg^{-1})\}$, where
$Q=\left(\begin{smallmatrix} 0 & v^t \\ v & 0 \end{smallmatrix}\right), \, v \in \br^2$.

For $X \in T_xM$ denote $p_X=(dp)_x X, \, q_X=(dq)_x X$. Then $\on_X \xi=\pi_{T_x\SL(3)/\SO(3)}\frac{1}{\sqrt{6}} (X(gg^t)-3p_Xp^t-3pp^t_X, -X(g^{-1t}g^{-1})+3q_Xq^t+3qq^t_X)$, where $\pi_{T_x\SL(3)/\SO(3)}$ is the orthogonal projection. Without loss of generality (as both the assumption and the conclusion of the proposition are $\SL(3)$-invariant) we can now take $x=(I,I), \, g= I$, and $p(x)=q(x)=(1,0,0)^t$. Projecting to $T_{(I,I)}\SL(3)/\SO(3)=\{(T,-T)\, | \, T \in \Sym^0(3, \br)\}$ we get $\on_X \xi=-\frac{3}{2\sqrt{6}} (T, -T)$, where $T=(p_X+q_X)(1,0,0)^t+ (1,0,0) (p_X+q_X)^t$ (note that $\<p_X+q_X,(1,0,0)^t\>=0$ as $\<p,q\>=1$). Now suppose $X \in \Span(X_3,X_4)$. Then $X=\{(Q, -Q)\}$, where $Q=\left(\begin{smallmatrix} 0 & 0 \\ 0 & P \end{smallmatrix}\right), \, P \in \Sym^0(2, \br)$. Moreover, we have $\on_X \xi= 0$, as $S_{(p,q)}$ is totally geodesic, and so $p_X+q_X=0$. On the other hand, differentiating the equation $gg^tq=p$ along $X$ we get $Q(1,0,0)^t+q_X=p_X$, and so $q_X=p_X$. It follows that $p_X=q_X=0$, for $X \in \Span(X_3,X_4)$. Next, for $j=1,2$, we have $X_j=\{(Q_j, -Q_j)\}$, where $Q=\left(\begin{smallmatrix} 0 & v_j^t \\ v_j & 0 \end{smallmatrix}\right), \, v_j \in \br^2$ (note that $v_1,v_2$ are linearly independent). Differentiating the equation $gg^tq=p$ along $X_j, \, j=1,2$, we get $Q_j(1,0,0)^t+q_j=p_j$ (where we abbreviate the subscript $X_j$ to $j$), and so $p_j=q_j+\left(\begin{smallmatrix} 0 \\ v_j \end{smallmatrix}\right)$. But $\<(p_j+q_j),(1,0,0)^t\>=0$, and so $q_j=\left(\begin{smallmatrix} 0 \\ w_j \end{smallmatrix}\right), \, p_j=\left(\begin{smallmatrix} 0 \\ v_j+w_j \end{smallmatrix}\right)$, for some $w_1,w_2 \in \br^2$ (note that this proves that the subset $\{(p,q)\} \subset \br^6$ is a 2-dimensional regular surface $F^2$, as $p_X=q_X=0$ for $X \in \Span(X_3,X_4)$ and $(p_1,q_1),(p_2,q_2)$ are linearly independent). It follows that $\<p,q_j\>=\<p_j,q\>=0$ for $j=1,2$, as required by \eqref{eq:pq} (and these equations are invariant under the action of $\SL(3)$ given by $\mL_g(p,q)=(gp,g^{-1t}q)$). Furthermore, we have $\on_j \xi = -\la_j X_j$ for $j=1,2$, and so $-\frac{3}{2\sqrt{6}}(2w_j+v_j)=\la_j v_j$. From Proposition~\ref{p:su3so3gen}\eqref{it:su3so3allaal}, the Einstein condition for $M$ gives $\la_1\la_2=\al$, where $2\al=-\frac34$, the Einstein constant of $\SL(3)/\SO(3)$. We obtain $\det(2w_1+v_1\,|\,2w_2+v_2)=-\det(v_1\,|\,v_2)$ which is equivalent to $\det(w_1+ v_1\,|\,w_2+ v_2) + \det(w_1\,|\,w_2)=0$. It follows that $\det(p,p_1,p_2) = \det(q,q_1,q_2)$ (which is again $\SL(3)$-invariant), as required by \eqref{eq:pq}.

To prove the converse, we start with a surface $F^2 \subset \br^6$ satisfying \eqref{eq:pq} and revert the argument (at the regular points of $M_{F^2}$ --- see Proposition~\ref{p:Tzrk}\eqref{it:Tz2sing}).
\end{proof}

\begin{proof}[Proof of Proposition~\ref{p:Tzrk}]
To prove assertion \eqref{it:nodp0}, suppose that $dp(u)=0$ at some point $u=(u_1,u_2)$. Then for the surface $F^2$ to be regular at $u$, the vectors $\partial_{u_1}q(u)$ and $\partial_{u_2}q(u)$ must be linearly independent. But then from \eqref{eq:pq}, the vector $q(u)$ lies in their span, while $p(u)$ is orthogonal to it, which is a contradiction with $\<p,q\>=1$.

For assertion \eqref{it:eitherh}, we suppose that $\rk (dp)= 1$ on a neighbourhood $U$ of a point $u=(u_1,u_2)$. If at some point $u' \in U$ the vectors $\partial_{u_1}q(u')$ and $\partial_{u_2}q(u')$ are linearly independent, then from \eqref{eq:pq} (as above), the vector $q(u')$ lies in their span, and $p(u')$ is orthogonal to it, contradicting the equation $\<p,q\>=1$. As \eqref{eq:pq} is symmetric relative to $p,q$, we have $dq \ne 0$ by assertion \eqref{it:nodp0}, and so $\rk (dq) = 1$ on $U$. As $F^2$ is regular, we can choose a local parametrisation in such a way that $p=p(u_1),\, q=q(u_2)$, and both $\partial_{u_1}p$ and $\partial_{u_2}q$ are nonzero. Then from \eqref{eq:pq} we find that $\partial_{u_1}p$ and $\partial_{u_2}q$ are constant nonzero vectors in $\br^3$. Then from $\<p,q\>=1$ we obtain that up to affine changes of $u_1$ and $u_2$ and an action of $\mL_g, \, g \in \SL(3)$, one has $p=(1,u_1,0)^t$ and $q=(1,0,u_2)^t$, as required. To see that $M$ is a (domain of a) codimension $1$ solvmanifold, note that $\tM$ is isometric to the solvable subgroup $\oSh \subset \SL(3)$ of matrices of the form $\left(\begin{smallmatrix} b^{-2} & 0 & c \\ d & ba & f \\0 & 0 & ba^{-1} \end{smallmatrix}\right)$, where $c,d,f \in \br$ and $a, b \in \br \setminus \{0\}$ (clearly, $\oSh$ is conjugate to the subgroup of upper-triangular matrices of $\SL(3)$), with the left-invariant metric defined by the pull-back via $\oSh \subset \SL(3) \to \SL(3)/\SO(3)$. Then $M_{F^2}$ is the subgroup of $\oSh$ given by $b=1$.

For assertion~\eqref{it:orTz} we assume that $\rk (dp)(u) = 2$. Then $p=p(u_1,u_2)$ is a regular surface in a neighbourhood of $u$. Then $q=q(u_1,u_2)$ is uniquely determined by \eqref{eq:pq} as $\det(p,\partial_{u_1}p, \partial_{u_2}p) q = \partial_{u_1}p \times \partial_{u_2}p$, and the equation $K\<p,n\>^{-4} = -1$ (where $K$ is the Gauss curvature of the surface $p=p(u_1,u_2)$, and $n$ is a unit normal) follows from the last equation of \eqref{eq:pq}. %The fact that the surface $p=p(u_1,u_2)$ is a proper affine sphere follows from Remark~\ref{rem:aff}\eqref{itr:Tz}.

For assertion~\eqref{it:Tz2sing} note that if $\rk (dp)(u) = 2$, then also $\rk (dq)(u) = 2$ (by the last equation of \eqref{eq:pq}). Up to the action by an element $g \in \SL(3)$ and from \eqref{eq:pq} one can assume that $p=q=(1,0,0)^t, \, \partial_{u_1}p=\partial_{u_1}q  =(0,1,0)^t, \, \partial_{u_2}p=(0,0,1)^t,\; \partial_{u_2}q=(0,0,-1)^t$. Then by \eqref{eq:S_pq} we have $S_{(p,q)}= \{x=(P,P^{-1}) \, | \, P= \left(\begin{smallmatrix} 1 & 0 \\ 0 & A \end{smallmatrix}\right), \, A \in \SL(2) \cap \Sym^+(2, \br)\}$. Then all the points on the geodesic lying in $S_{(p,q)}$ and given by  $A=\left(\begin{smallmatrix} \cosh \theta & \sinh \theta \\ \sinh \theta & \cosh \theta \end{smallmatrix}\right), \, \theta \in \br$, are singular points of $M_{F^2}$. Indeed, differentiating the equation $Pq=p$ from \eqref{eq:S_pq} at the points of this geodesic we find that the vector $(\cosh\frac{\theta}{2} \partial_{u_1} + \sinh\frac{\theta}{2} \partial_{u_2}) x$ is tangent to the leaf $S_{(p,q)}$ (compare to Remark~\ref{rem:Legsing}).
\end{proof}
% can we have both on the same surface? fronts (ref to this recent preprint, but note improper)? add "indefinite proper"

\begin{proof}[Proof of the Corollary] % ref {cor:...}?
  By Remark~\ref{rem:Legsing}, if $\tM=\SU(3)/\SO(3)$, the hypersurface $M$ cannot contain the whole leaf $S_Z$, for any $Z \in F^2$. Similarly, by  Proposition~\ref{p:Tzrk}\eqref{it:Tz2sing}, if $\tM=\SL(3)/\SO(3)$, the hypersurface $M$ cannot contain the whole leaf $S_{(p,q)}$, for any $(p,q) \in F^2$ such that $\rk (dp)=2$. It follows from Proposition~\ref{p:Tzrk}\eqref{it:nodp0}, \eqref{it:eitherh} that a complete Einstein hypersurface $M \subset \SL(3)/\SO(3)$ must be a codimension one solvmanifold. If $\tM$ is different from both $\SU(3)/\SO(3)$ and $\SL(3)/\SO(3)$, the claim follows from Theorem~\ref{th:main}\eqref{thit:solv}.
\end{proof}

% maybe somewhere comment on the reducible case (or at least, reducible, Einstein case)
% examples: cohom 1?
% not-simply connected... eg SL(Z)\SL3/SO3 do we need to pay much attention? Only when \tM=...?


\begin{thebibliography}{KNP2}

\bibitem[Ale]{Ale}
D.\,V.\,Alekseevski\u{i}, \emph{The conjugacy of polar decompositions of Lie groups}, Math. USSR-Sb. \textbf{13} (1971), 12--24.

\bibitem[AW1]{AW1}
R.\,Azencott, E.\,Wilson, \emph{Homogeneous manifolds with negative curvature. I}, Trans. Amer. Math. Soc. \textbf{215} (1976), 323--362.

\bibitem[AW2]{AW2}
R.\,Azencott, E.\,Wilson, \emph{Homogeneous manifolds with negative curvature. II}, Mem. Amer. Math. Soc. 8 (1976), no. 178.

\bibitem[Ber]{Ber}
J.\,Berndt, \emph{Hyperpolar homogeneous foliations on symmetric spaces of noncompact type}, Proceedings of the 13th International Workshop on Differential Geometry and Related Fields \textbf{13} (2009), 37--57.

\bibitem[BO]{BO}
J.\,Berndt, C.\,Olmos, \emph{On the index of symmetric spaces}, J. Reine Angew. Math. \textbf{737} (2018), 33-–48.

\bibitem[Bes]{Bes}
A.\,Besse, \emph{Einstein manifolds}. Springer, Berlin, Heidelberg, New York, 1987.

\bibitem[BL]{BL}
C.\,B\"{o}hm, R.\,Lafuente, \emph{Non-compact Einstein manifolds with symmetry}, \href{https://arxiv.org/abs/2107.04210}{arXiv: 2107.04210 [math.DG]}.

\bibitem[CR]{CR}
T.\,Cecil, P.\,Ryan, \emph{Geometry of Hypersurfaces}. Springer Monographs in Mathematics. Springer-Verlag, New York, 2015.

\bibitem[CLS]{CLS}
V.\,Cort\'{e}s, M.\,Lawn, L.\,Sch\"{a}fer, \emph{Affine hyperspheres associated to special para-K\"{a}hler manifolds}, Int. J. Geom. Methods Mod. Phys. \textbf{3} (2006), 995-–1009.

\bibitem[CFG]{CFG}
V.\,Cruceanu, P.\,Fortuny, P.\,M.\,Gadea, \emph{A Survey on Paracomplex Geometry}, Rocky Mountain J. Math. \textbf{26}(1996), 83--115.

\bibitem[DTK]{DTK}
D.\,M.\,DeTurck, J.\,L.\,Kazdan, \emph{Some regularity theorems in Riemannian geometry},
Ann. Sci. \'{E}cole Norm. Sup. (4) \textbf{14} (1981), 249--260.

\bibitem[DSV]{DSV}
A.\,J.\,Di Scala, F.\,Vittone, \emph{Codimension reduction in symmetric spaces}, J. Geom. Phys. \textbf{79} (2014), 29--33.

\bibitem[DDS]{DDS}
J.\,C.\,D\'{\i}az-Ramos, M.\,Dom\'{\i}nguez-V\'{a}zquez, V.\,Sanmart\'{\i}n-L\'{o}pez, \emph{Submanifold geometry in symmetric spaces of noncompact type}, S\~{a}o Paulo J. Math. Sci. \textbf{15} (2021), 75--110.

\bibitem[Fia]{Fia}
A.\,Fialkow, \emph{Hypersurfaces of a Space of Constant Curvature}, Ann. of Math. \textbf{39} (1938), 762--785.

\bibitem[HL]{HL}
R.\,Harvey, H.\,B.\,Lawson, \emph{Calibrated geometries}, Acta Math. \textbf{148} (1982), 47--157.

\bibitem[Has]{Has}
M.\,Haskins, \emph{Special Lagrangian cones}, Amer. J. Math. \textbf{126} (2004), 845--871.

\bibitem[Hel]{Hel}
S.\,Helgason, \emph{Differential geometry, Lie groups, and symmetric spaces}. Pure and Applied Mathematics, 80. Academic Press, Inc. New York--London, 1978.

\bibitem[KNP1]{KNP1}
S.\,Kim, Y.\,Nikolayevsky and J.\,H.\,Park, \emph{Einstein hypersurfaces of Damek-Ricci spaces}, J. Geom. Phys. \textbf{167} (2021), 104278.

\bibitem[KNP2]{KNP2}
S.\,Kim, Y.\,Nikolayevsky and J.\,H.\,Park, \emph{Einstein hypersurfaces of the Cayley projective plane}, Differ. Geom. Appl. \textbf{69} (2020), 101594.

\bibitem[Kob]{Kob}
S.\,Kobayashi, \emph{Isometric imbeddings of compact symmetric spaces}, T\^{o}hoku Math. J. \textbf{20} (1968), 21--25.

\bibitem[Kon]{Kon}
M.\,Kon, \emph{Pseudo-Einstein real hypersurfaces in complex space forms}, J. Differ. Geom. \textbf{14} (1979), 339--354.

\bibitem[LPS]{LPS}
B.\,Leandro, R.\,Pina, J.\,P.\,dos Santos, \emph{Einstein hypersurfaces of $S^n \times \br$ and $H^n \times \br$}, Bull. Braz. Math. Soc. (N.S.) \textbf{52} (2021), 537--546.

\bibitem[Leu]{Leu}
D.\,Leung, \emph{The reflection principle for minimal submanifolds of Riemannian symmetric spaces}, J. Differ. Geom. \textbf{8} (1973), 153--160.

\bibitem[MT]{MT}
F.\,Manfio, R.\,Tojeiro, \emph{Hypersurfaces with constant sectional curvature of $S^n \times \br$ and $H^n \times \br$}, Ill. J. Math. \textbf{55} (2011), 397--415.

\bibitem[MP]{MP}
A.\,Martinez, J.D.\,P\'{e}rez, \emph{Real hypersurfaces in quaternionic projective space}, Ann. Mat. Pura Appl. (4) \textbf{145} (1986), 355--384.

\bibitem[Mon]{Mon}
S.\,Montiel, \emph{Real hypersurfaces of a complex hyperbolic space}, J. Math. Soc. Japan \textbf{37} (1985), 515--535.

\bibitem[Mry]{Mry}
C.\,B.\,Morrey, \emph{Multiple Integrals in the Calculus of Variations}. Springer-Verlag, Berlin, Heidelberg, 1966.

\bibitem[Mor]{Mor}
J.-M.\,Morvan, \emph{Classe de Maslov d'une immersion Lagrangienne et minimalit\'{e}}, C. R. Acad. Sci. Sér. I, \textbf{292} (1981), 633--636.

\bibitem[Nik]{Nik}
Y.\,Nikolayevsky, \emph{Totally geodesic hypersurfaces of homogeneous spaces}, Israel J. Math. \textbf{207} (2015), 361--375.

\bibitem[NS]{NS}
K.\,Nomizu, T.\,Sasaki, \emph{Affine differential geometry}. Cambridge Tracts in Mathematics, 111. Cambridge University Press, Cambridge, 1994.

\bibitem[On]{On}
A.\,L.\,Onishchik, \emph{On totally geodesic submanifolds of symmetric spaces} (Russian), Geom. Metody Zadachakh Algebry Anal. \textbf{2} (1980), 64-–85.

\bibitem[OP]{OP}
M.\,Ortega, J.D.\,P\'{e}rez, \emph{On the Ricci tensor of a real hypersurface of quaternionic hyperbolic space}, Manuscripta Math. \textbf{93} (1997), 49--57.

\bibitem[RR]{RR}
\'{A}.\,Ramos, J.\,Ripoll, \emph{An extension of Ruh-Vilms' theorem to hypersurfaces in symmetric spaces and some applications}, Trans. Amer. Math. Soc. \textbf{368} (2016), 4731--4749.

\bibitem[T1]{T1}
H.\,Tamaru, \emph{The local orbit types of symmetric spaces under the actions of the isotropy subgroups}, Differential Geom. Appl. \textbf{11} (1999), 29--38.

\bibitem[T2]{T2}
H.\,Tamaru, \emph{Parabolic subgroups of semisimple Lie groups and Einstein solvmanifolds}, Math. Ann. \textbf{351} (2011), 51--66.

\bibitem[Tz]{Tz}
M.\,G.\,Tzitz\'{e}ica, \emph{Sur une nouvelle classe de surfaces}, Rend. Circ. Matem. Palermo \textbf{25} (1908), 180--187.

\bibitem[Wol]{Wol}
T.\,H.\,Wolter, \emph{Einstein Metrics on solvable groups}, Math. Z. \textbf{206} (1991), 457--471.

\bibitem[Yau]{Yau}
S.\,T.\,Yau, \emph{Submanifolds with constant mean curvature. I}, Amer. J. Math. \textbf{96} (1974), 346--366; \emph{II}, Amer. J. Math. \textbf{97} (1975), 76--100.


\end{thebibliography}
\end{document}